\documentclass{amsart}
\usepackage{dsfont}
\usepackage{helvet}
\usepackage{color}
\usepackage{mathrsfs}
\usepackage{cancel}

\usepackage{hyperref}

\usepackage{amsmath,amsxtra,amsthm,amssymb,xr,soul,fullpage,comment}
\usepackage[all]{xy}

\numberwithin{equation}{section}
\newtheorem{theorem}{Theorem}[subsection]
\newtheorem{lemma}[theorem]{Lemma}

\newtheorem{proposition}[theorem]{Proposition}
\newtheorem{corollary}[theorem]{Corollary}

\newtheorem{defn}[theorem]{Definition}

\newtheorem{remark}[theorem]{Remark}

\newtheorem{conjecture}[theorem]{Conjecture}

\newtheorem{lthm}{Theorem} 

\makeatletter
\newcommand{\mylabel}[2]{#2\def\@currentlabel{#2}\label{#1}}
\makeatother

\newcommand{\cBF}{\mathcal{BF}}

\newcommand{\Gal}{\operatorname{Gal}}
\newcommand{\Fil}{\operatorname{Fil}}
\newcommand{\Sym}{\operatorname{Sym}}
\newcommand{\DD}{\mathbb{D}}
\newcommand{\BB}{\mathbb{B}}
\newcommand{\NN}{\mathbb{N}}
\newcommand{\QQ}{\mathbb{Q}}
\newcommand{\Qp}{\mathbb{Q}_p}
\newcommand{\Zp}{\mathbb{Z}_p}
\newcommand{\ZZ}{\mathbb{Z}}
\renewcommand{\AA}{\mathbb{A}}

\newcommand{\vp}{\varphi}
\newcommand{\cL}{\mathcal{L}}
\newcommand{\cO}{\mathcal{O}}
\newcommand{\Iw}{\mathrm{Iw}}
\newcommand{\HIw}{H^1_{\mathrm{Iw}}}

\newcommand{\Brig}{\BB_{{\rm rig},\Qp}^+}
\newcommand{\AQp}{\AA_{\Qp}^+}
\newcommand{\col}{\mathrm{Col}}
\newcommand{\image}{\mathrm{Im}}
\newcommand{\cyc}{\textup{cyc}}
\newcommand{\cH}{\mathcal{H}}

\newcommand{\loc}{\mathrm{loc}}

\newcommand{\Hom}{\mathrm{Hom}}
\newcommand{\Char}{\mathrm{char}}

\newcommand{\LL}{\Lambda}
\newcommand{\TT}{\mathbb{T}}

\newcommand{\Sel}{\mathrm{Sel}}

\newcommand{\lra}{\longrightarrow}
\newcommand{\ra}{\rightarrow}
\newcommand{\res}{\textup{res}}

\newcommand{\pr}{\textup{pr}}

\newcommand{\BF}{\textup{BF}}

\newcommand{\Dcris}{\DD_{\rm cris}}
\renewcommand{\col}{\mathrm{Col}}

\newcommand{\Tw}{\mathrm{Tw}}

\definecolor{Green}{rgb}{0.0, 0.5, 0.0}

\newcommand{\cF}{\mathcal{F}}

\newcommand{\p}{\mathfrak{p}}
\newcommand*{\defeq}{\mathrel{\vcenter{\baselineskip0.5ex \lineskiplimit0pt
			\hbox{\scriptsize.}\hbox{\scriptsize.}}}%
	=}

\begin{document}

\title[Non-Ordinary Symmetric Squares]{Iwasawa theory for symmetric squares of non-$p$-ordinary eigenforms}

\begin{abstract}
Let $f$ be a normalized cuspidal  eigen-newform of level coprime to $p$ with $a_p(f)=0$. We formulate both integral signed Iwasawa main conjectures and analytic Iwasawa main conjectures attached to the symmetric square motive of $f$ twisted by an auxiliary Dirichlet character. We show that the Beilinson--Flach elements attached to the symmetric square motive factorize into integral signed Beilinson--Flach elements,  giving evidence towards the existence of a rank-two Euler system predicted by Perrin-Riou. We use these integral  elements to prove one inclusion in the integral and analytic Iwasawa main  conjectures.
\end{abstract}

\author{K\^az\i m B\"uy\"ukboduk}
\address{K\^az\i m B\"uy\"ukboduk\newline
UCD School of Mathematics and Statistics \\
University College Dublin \\ 
Belfield, Dublin 4, Ireland}
\email{kazim.buyukboduk@ucd.ie}

\author{Antonio Lei}
\address{Antonio Lei\newline
D\'epartement de Math\'ematiques et de Statistique\\
Universit\'e Laval, Pavillion Alexandre-Vachon\\
1045 Avenue de la M\'edecine\\
Qu\'ebec, QC\\
Canada G1V 0A6}
\email{antonio.lei@mat.ulaval.ca}

\author{Guhan Venkat}
\address{Guhan Venkat\newline
Morningside Center of Mathematics\\
Academy of Mathematics and Systems Science\\
Chinese Academy of Sciences\\
Beijing, 100190\\
China.}
\email{guhan@amss.ac.cn}

\thanks{The authors' research is partially supported by an Alexander von Humboldt Fellowship for Experienced Researchers (B\"uy\"ukboduk), the NSERC Discovery Grants Program RGPIN-2020-04259 and RGPAS-2020-00096 (Lei, Venkat) and a CRM-Laval postdoctoral fellowship (Venkat).}
\subjclass[2010]{11R23 (primary); 11F11, 11R20 (secondary) }
\keywords{Iwasawa theory, elliptic modular forms, symmetric square representations, non-ordinary primes}
\maketitle
\tableofcontents
\section{Introduction}
\subsection{Background}
Throughout this article, we fix an odd prime $p \geq 7$ and embeddings $\iota_\infty:\overline{\QQ}\hookrightarrow \mathbb{C}$ and $\iota_p:\overline{\QQ}\hookrightarrow \mathbb{C}_p$. Let  $f$ be a normalised,  cuspidal eigen-newform of weight $k + 2$, level $N$ and nebentype $\epsilon_f$. We assume that $p \nmid N$, $p>k+1$ and $a_p(f) = 0$. We shall write $\pm\alpha$ for the roots of the Hecke polynomial $X^2+\epsilon_f(p)p^{k+1}$ of $f$ at $p$.

Let $L/\QQ$ be a number field containing the Hecke field {$K_{f}\defeq\QQ(\{a_{n}(f)\}_{n\geq1})$} of $f$ as well as $\alpha^{2}$. Let $\mathfrak{p}$ be a prime in $L$ above $p$. We denote by $E$ the completion of $L$ at $\mathfrak{p}$. Let $\cO$ denote the ring of integers of $E$. We fix a Galois-stable $\cO$-lattice $R_f$ inside Deligne's $E$-linear representation $W_f$ of $G_\QQ$. Let $\Gamma=\Gal(\QQ(\mu_{p^\infty})/\QQ)$.  We write $\Lambda_{\cO}(\Gamma) = \cO[[\Gamma]]$ for the Iwasawa algebra on $\Gamma$. We have the decomposition $\Gamma = \Gamma_{\text{tors}} \times \Gamma_{1}$, where $\Gamma_{\text{tors}}$ is a finite group of order $p-1$ and $\Gamma_{1} = \Gal(\QQ(\mu_{p^\infty})/\QQ(\mu_{p}))$. We fix a topological generator $\gamma$ of $\Gamma_{1}$, which in turn determines an isomorphism $\Gamma_{1} \cong \ZZ_p$.  {Also, let $\QQ_{\infty}/\QQ$ denote the cyclotomic $\ZZ_{p}$-extension of $\QQ$.} For any $\ZZ_p$-module $M$, we denote its Pontryagin dual $\Hom_{\textup{cts}}(M,\QQ_p/\ZZ_p)$ by $M^{\vee}$.

In \cite{BLLV}, we studied the cyclotomic Iwasawa theory of the Rankin-Selberg convolution of two modular forms $f$ and $g$ that are non-ordinary at $p$, making use of the Beilinson--Flach Euler systems constructed by Loeffler and Zerbes in \cite{LZ1}. In this paper, we concentrate on the case where $f=g$ and $a_p(f)=0$. The results we obtain in this {set-up} do not rely on the conjectural existence of a  rank-two Euler system, as some of our main results in \cite{BLLV} do. Our treatment naturally goes through the study of the symmetric square motive $\Sym^2f$. This extends the work of Loeffler and Zerbes \cite{LZ2} in the ordinary case (which we briefly summarize in Section \ref{subsec_p-ordinary} below).

Let us put {$W_f^* :=\Hom(W_f,E)$ and endow it with the contragredient Galois action.} For $\lambda,\mu\in\{\pm \alpha\}$ and an integer $m$ that is coprime to $p$, recall from \cite{LZ1} the Beilinson--Flach elements 
\[
\BF_{m}^{\lambda,\mu}\in H^1(\QQ(\mu_{m}),W_{f}^*\otimes W_{f}^*(1)\otimes\cH_{E,k+1}(\Gamma)^\iota)\\
\]
where $\cH_{E,k+1}(\Gamma)$ denotes the set of $E$-valued  tempered distributions of order $k+1$ on $\Gamma$ and $\cH_{E,k+1}(\Gamma)^\iota=\cH_{E,k+1}(\Gamma)\otimes_{\Lambda_{\cO}(\Gamma)}\Lambda_{\cO}(\Gamma)^\iota$ (here, $\Lambda_{\cO}(\Gamma)^{\iota}$ denotes the free rank-one $\Lambda_{\cO}(\Gamma)$-module on which $G_{\QQ}$ acts via the inverse of the canonical character $G_{\QQ} \twoheadrightarrow \Gamma \hookrightarrow \Lambda_{\cO}(\Gamma)^{\times}$). Consider the decomposition
\begin{equation}
\label{eqn:decomposethesymproduct}
W_f^*\otimes W_f^*=\Sym^2W_f^*\oplus{\bigwedge\!}^2W_f^*.
\end{equation}
In Section \ref{subsec:plocalBF}, we explain that the twist of the Beilinson--Flach classes by an even Dirichlet character $\chi$ take values in the corresponding twist $\Sym^2W_f^*(1+\chi)$. This equips us with a non-integral collection of cohomology classes that verify a close variant of the Euler system distribution relation. The non-integrality of these classes is the source of main difficulty in the non-ordinary set-up. The main task we carry out here is to obtain an integral collection which we may plug {into} the Euler system machinery. 

This goal has been partially achieved in \cite{BLLV}, employing ideas from signed Iwasawa theory (expanding on \cite{BL16b} where  the semi-ordinary case is treated) and taking inspiration from  Perrin-Riou's theory of higher-rank Euler systems.  The theory of higher-rank Euler systems suggests a signed factorization of the four collections of Beilinson--Flach elements (see Theorem~\ref{thm_theoremA} below for the shape of this factorization). However, the interpolative properties of Beilinson--Flach classes cover only half of the critical range for the symmetric square motive and as a result, the standard techniques only enable us to prove a weaker form of this factorization (and resulting in still non-integral collections of cohomology classes). We develop a new method (see Sections~\ref{S:signedBF} and \ref{subsec:theproofoftheorem39} below) which allows us to improve this factorization statement to cover the full critical range.

\subsection{Main results}
\label{subsec_main_results}
Our first result in this paper is {the existence of integral \emph{Beilinson--Flach Euler systems} in the current  set-up  under suitable hypotheses. This proves \cite[Conjecture 5.3.1]{BLLV} in this particular setting.}

Given a Dirichlet character $\psi$ of conductor $N_\psi$, we let $\mathcal{R}_\psi$ denote the collection of square-free products of primes which are coprime to $pNN_\psi$. For any prime $\ell$, let $\QQ(\ell)$ denote the unique abelian $p$-extension in $\QQ(\mu_\ell)$. For an element $r=\ell_1\cdots\ell_s \in \mathcal{R}_\psi$, we define $\QQ(r)$ to be the compositum of the (linearly disjoint) fields $\QQ(\ell_1),\cdots, \QQ(\ell_s)$. We also set $\Delta_r=\textup{Gal}(\QQ(r)/\QQ)$ and note that $\Delta_r=\Delta_{\ell_1}\times\cdots\times\Delta_{\ell_s}$. For a factor $\ell$ of $r$, we shall think of $\Delta_{\ell}$ both as a subgroup and as a quotient of $\Delta_r$ through this identification. We finally let $\LL_r$ denote the ring $\cO[[\Delta_r\times\Gamma]]$.

We fix forever an even Dirichlet character $\chi$ {whose conductor $N_\chi$ is coprime to $Np$}. For $m\in \mathcal{N}_\chi$ (where the set of integers $\mathcal{N}_\chi $ is given in Definition~\ref{def:Kolyvaginprimesfortwists}), we let  {$\BF_{m,\chi}^{\lambda,\mu}\in \HIw(\QQ(m),W_{f}^*\otimes W_{f}^*(1+\chi) {\otimes} \cH_{E,k+1}(\Gamma)^\iota)$ denote the natural image of the Beilinson--Flach elements defined in \cite{LZ1} (see Definition~\ref{define:twistedBF} below). We write $ L_{p}^{\mathrm{geom}}(\mathrm{Sym}^{2} f_\lambda \otimes \chi^{-1}, s)$ for the \emph{geometric} $p$-adic $L$-function attached $\mathrm{Sym}^{2} f_\lambda \otimes \chi^{-1}$ defined as in \eqref{eqn_define_padicL_sym2_weight_k}\,. Until the end of this article, we assume that the following non-vanishing condition holds true:}
\begin{itemize}
\item[\mylabel{item_NV}{\textup{\textbf{(NV)}}}]  $L_{p}^{\mathrm{geom}}(\mathrm{Sym}^{2} f_\lambda \otimes \chi^{-1}, j)\neq 0$ for 
 every even integer $k+2\leq j\leq 2k+2$. 
\end{itemize}
We further consider the following hypotheses.
\begin{itemize}
\item[\mylabel{item_Psi_1}{$(\Psi_1)$}] There exists $u \in (\ZZ/NN_\chi\ZZ)^\times$ such that $\epsilon_f\chi^{-1}(u) \not\equiv \pm1\, (\textup{mod}\, \p)$ and $\chi(u)$ is a square modulo $\p$. 
\item[\mylabel{item_Psi_2}{$(\Psi_2)$}] $\epsilon_f\chi^{-1}(p)\neq \pm1$ and $\phi(N)\phi(N_\chi)$ is coprime to $p$, where $\phi$ is Euler's totient function. 
\item[\mylabel{item_Psi_3}{$(\Psi_3)$}] The prime $\p$ over $p$ in  $K_f$  has degree $1$ and ${\rm im}(\chi)\subset \ZZ_p^\times$.
\item[\mylabel{item_Im}{\textup{\textbf{(Im)}}}] $\textup{im}\left(G_\QQ\ra \textup{Aut}(R_f\otimes\QQ_p)\right)$ contains a conjugate of $\textup{SL}_2(\ZZ_p)$. 
\end{itemize}
\begin{remark}
{The main results of this paper all rely on the hypothesis  \ref{item_NV}.} {It follows from Corollary~\ref{coro_trivialzeroes} below that if we assume \ref{item_Psi_2}, then \ref{item_NV}would follow from a generalization of Dasgupta's factorization result in~\cite{dasgupta1} to the non-ordinary setting. This is the subject of Arlandini's forthcoming work, which we record as Theorem~\ref{thm_DLZfactorizationtheorem} below. With Arlandini's work, we will be able to remove the condition\ref{item_NV} on our results.}
\end{remark}
\begin{lthm}[Corollary~\ref{cor_BFfactorisation}, Proposition~\ref{prop:latticeBF}]  \label{thm_theoremA}
Suppose that $\chi$ verifies the hypotheses \ref{item_Psi_1} and \ref{item_Psi_2}. Assume also that \ref{item_NV} and \ref{item_Im} hold true. Then for every $m\in \mathcal{N}_\chi$, there exist
$$\BF_{m,\chi}^{+},\BF_{m,\chi}^{-},\BF_{m,\chi}^{\bullet}, \BF_{m,\chi}^{\circ}\in H^1_\Iw(\QQ(m),W_f^*\otimes W_f^*(1+\chi))$$ 
that verify Euler system distribution relations and such that
\[
\begin{pmatrix}
1&1&1&1\\
\alpha^2&\alpha^2 &-\alpha^2&-\alpha^2\\
2\alpha&-2\alpha&0&0\\
0&0&-2\alpha&2\alpha
\end{pmatrix}
\begin{pmatrix}
\BF^{\alpha,\alpha}_{m,\chi}\\
\BF_{m,\chi}^{-\alpha,-\alpha}\\
\BF_{m,\chi}^{\alpha,-\alpha}\\
\BF_{m,\chi}^{-\alpha,\alpha}
\end{pmatrix}= \begin{pmatrix}
\log_{p,2k+2}^{+,(1)}\BF_{m,\chi}^{+}\\ \log_{p,2k+2}^{-,(1)}\BF_{m,\chi}^{-}\\ \log_{p,k+1}^{(1)}\BF_{m,\chi}^{\bullet}\\\log_{p,k+1}^{(1)}\BF_{m,\chi}^{\circ}
\end{pmatrix}\,.
\]
Here, $\log_{p,2k+2}^{\pm,(1)}$ and $\log_{p,k+1}^{(1)}$ are some explicit {functions} defined in Section~\ref{sec:signediwasawatheory}. Furthermore, there exists an integer $C$ independent of $m$ such that 
\[
C\times\BF_{m,\chi}^{\clubsuit}\in H^1_\Iw(\QQ({m}),R_f^*\otimes R_f^*(1+\chi))
\]
for all four choices of $\clubsuit\in\{+,-,\bullet,\circ\}$.
\end{lthm}

Under our assumption that $\chi$ is even, {the final assertion in the statement of} {Theorem~\ref{thm_theoremA} can be recast in the following form:} 

\begin{corollary}[Corollary~\ref{cor:signedBFinsymmsquare}] \label{cor:corollaryB} 
In the setting of Theorem~\ref{thm_theoremA}, the signed classes 
\[C \times \BF_{m,\chi}^{+},\ C \times \BF_{m,\chi}^{-}, \ 
C \times \BF_{m,\chi}^{\bullet}
\]
{are elements of} $H^1_\Iw(\QQ({m}),\Sym^2 R_f^*(1+\chi))$. {Furthermore,  $\BF_{m,\chi}^{\circ}=0$ for all $m$.}
\end{corollary}

In particular, each one of {the} four collections $\{C \times  \BF_{m,\chi}^{\clubsuit}\}$, where $\clubsuit \in \{ +, -, \bullet,\circ \}$, form a (rank-one) Euler system for $\Sym^2 R_f^*(1+\chi)$. In order to apply the Euler system machinery, we need to ensure that at least one of the Euler systems in Corollary~\ref{cor:corollaryB} is non-trivial. In order to do this, it suffices to prove that one of the four non-integral classes $\{\BF^{\alpha,\alpha}_{m,\chi}, \BF^{-\alpha,-\alpha}_{m,\chi}, \BF^{\alpha,-\alpha}_{m,\chi}, \BF^{-\alpha,\alpha}_{m,\chi}\}$ are non-trivial.

To achieve this, we appeal to the reciprocity laws of Leoffler and Zerbes in \cite{LZ1}, which enable us to reduce the required non-vanishing to the non-triviality of the Rankin-Selberg $p$-adic $L$-functions associated to $f \otimes f\otimes{\chi}$. Note that the motive associated to $f\otimes f\otimes \chi$ does not possess any critical values and as a result, one may not appeal to non-vanishing statements on complex $L$-values to deduce the required non-triviality. However, Arlandini's work in progress (extending Dasgupta's result~\cite[Theorem 1]{dasgupta1} in the $p$-ordinary case) shows that the $p$-adic $L$-functions in question factors as a product of the symmetric square $p$-adic $L$-function and a Kubota--Leopoldt $p$-adic $L$-function. The required non-triviality easily follows from generic non-vanishing statements for symmetric square $L$-values; see Section~\ref{sec:p-adicLfunctions} for details.

Let us set $T:=\Sym^{2}R_{f}^{*}(1+\chi)$ to ease our notation. Our running hypothesis $a_{p}(f) = 0$ yields a $G_{\QQ_p}$-equivariant decomposition
\[ T = R_{1,\chi}^{*} \oplus R_{2,\chi}^{*}\,. \]
Exploiting this decomposition, we define  \emph{signed Coleman maps} as in \cite{lei12}. More precisely, we define $\Lambda_{\cO}(\Gamma)$-morphisms 
\[ \mathrm{Col}^{\clubsuit} : H^{1}_{\mathrm{Iw}}(\QQ_{p}(\mu_{p^{\infty}}),T) \rightarrow \Lambda_{\cO}(\Gamma) \]  
for $\clubsuit \in \{+, -, \bullet \}$ in Section~\ref{sec:signedcoleman}. For each $\mathfrak{S} = (\clubsuit, \spadesuit) \in \{(+, -), (+, \bullet), (-, \bullet)\}$, we define the \emph{doubly signed Beilinson--Flach $p$-adic $L$-function} in Section ~\ref{sec:nonordIwasawaESargument}, by setting
\[ \mathcal{L}_{\mathfrak{S}} := \col^{\clubsuit} \circ \mathrm{res}_{p} (\BF_{1, \chi}^{\spadesuit})  \in \Lambda_{E}(\Gamma)\,. \]

Still using the signed Coleman maps alluded to above, we define also doubly signed Selmer groups which we denote by  $\Sel_{\mathfrak{S}}(T^\vee(1)/\QQ(\mu_{p^\infty}))$ (where $\mathfrak{S}$ is as above). This allows us to formulate  \emph{Doubly Signed Iwasawa Main Conjecture} (Conjecture~\ref{conj:signedmainconjecture} below), relating $\cL_{\mathfrak{S}}$ and $\Sel_{\mathfrak{S}}(T^\vee(1)/\QQ(\mu_{p^\infty}))^\vee$. 

The following is one of our main results towards the Iwasawa main conjectures for non-ordinary symmetric squares. For a given integer $j$, we let $e_{\omega^j}$ denote the idempotent attached to the character $\omega^j$ {(where $\omega$ is the Teichm\"uller character)}.

\begin{lthm}[Theorem~\ref{thm_signedmainconjecture}]\label{thm:intro} Suppose that the Dirichlet character $\chi$ verifies the hypotheses \ref{item_Psi_1},  \ref{item_Psi_2} and  \ref{item_Psi_3}. Assume also that  \ref{item_NV} and \ref{item_Im} hold true. 
\item[i)]{For all even $j \in \{k+2,\ldots,2k+2 \}$, there exists}  $\mathfrak{S} \in \{(+, -), (+, \bullet), (-, \bullet)\}$ such that $e_{\omega^{j}}\mathcal{L}_{\mathfrak{S}}\neq 0$.
\item[ii)] For $j$ and $\mathfrak{S}$ as in \textup{i)}, the $\Lambda(\Gamma_{1})$-module $\Sel_{\mathfrak{S}}(T^\vee(1)/\QQ(\mu_{p^\infty}))^{\vee}$ is torsion.
\item[iii)] For $j$ and $\mathfrak{S}$ as in \textup{i)}, 
$$\mathrm{char}_{\Lambda(\Gamma_{1})}\left( e_{\omega^{j}}\Sel_{\mathfrak{S}}(T^\vee(1)/\QQ(\mu_{p^\infty}))^{\vee}\right) \,\,\big{\vert}\,\, (e_{\omega^{j}}\mathcal{L}_{\mathfrak{S}})$$
as ideals of $\Lambda(\Gamma_{1}) \otimes \QQ_{p}$\,.
\end{lthm}

{Theorem~\ref{thm:intro} has consequences towards the Pottharst-style (analytic) Iwasawa main conjectures for non-ordinary symmetric squares. Let  $\lambda\in\{\pm\alpha\}$. The $(\vp,\Gamma)$-module attached to the $\lambda$-eigenspace in the Dieudonn\'e module of $W_f$ gives rise to a Pottharst-style analytic Selmer group $\widetilde{H}^2_{\Iw}(\QQ,V,\mathbb{D}_{\chi}^{\lambda})$ (see \S\ref{subsec_AnaylticMainConj} for details). We prove the following partial result towards the analytic main conjecture relating it to the geometric $p$-adic $L$-function $L_{p}^{\mathrm{geom}}(\Sym^{2}f_{\lambda}\otimes\chi^{-1})$.}

\begin{lthm}[Theorem~\ref{thm_analyticmainconjecture}]\label{thmC}Suppose that the Dirichlet character $\chi$ verifies the hypotheses \ref{item_Psi_1},  \ref{item_Psi_2} and  \ref{item_Psi_3}. Assume also that  \ref{item_NV} and \ref{item_Im}  hold true. For $j$ and $\mathfrak{S}=\{\clubsuit,\spadesuit\}$ as in Theorem~\ref{thm:intro} \textup{i)}, 
\[ \mathrm{char}_{\mathcal{H}}\left(e_{\omega^j}\widetilde{H}^2_{\Iw}(\QQ,V,\mathbb{D}_{\chi}^{\lambda}) \right)\,\,\big{\vert}\,\, \mathrm{char}(e_{\omega^j}\mathrm{coker}\col^{\clubsuit})\,e_{\omega^j}L_{p}^{\mathrm{geom}}(\Sym^{2}f_{\lambda}\otimes\chi^{-1},{s})\,L_{p,NN_{\chi}}(\chi^{-1}\epsilon_f,{s-k-1})\cdot\mathcal{H}\,. \]
Here, $\cH:=\varinjlim_m \cH_{E,m}(\Gamma_1)$ and $L_{p,NN_\chi}(\chi^{-1}\epsilon_f)$ is the Kubota-Leopoldt $p$-adic $L$-function attached to the Dirichlet character $\chi^{-1}\epsilon_f$, with Euler factors at primes dividing $NN_\chi$ removed.
\end{lthm}

\begin{remark}
{In \cite{BL-TAMS}, we build on the results of the present article  to prove the existence of a non-trivial rank-2 Euler system, whose non-triviality is ensured by the non-vanishing of certain $L$-values.}
\end{remark}

{ \subsection{Review of earlier related work}\label{subsec_p-ordinary}
In this section, we compare the results in the present article to previous related work.

We first recall the main results of \cite{LZ2}. Let $g$ be a normalized cuspidal new eigenform of level $N_g$, weight $k_{g} + 2$ and nebentype $\epsilon_g$. Assume that $p \nmid N_g$ and $p$ is an ordinary prime for $g$, i.e. $a_{p}(g)$ is a $p$-adic unit  under our fixed embeddings. Let $\alpha_g$ be the unit root of the Hecke polynomial of $g$ at $p$.

For a Dirichlet character $\chi$ as in the previous section, we let 
\[    
\BF_{m,\chi}^{g,g} \defeq \BF_{m,\chi}^{\alpha_{g},\alpha_{g}} \in H^1_\Iw(\QQ(m),W_g^*\otimes W_g^*(1+\chi))
\]
denote the $\chi$-twisted Beilinson--Flach element attached to the Rankin--Selberg convolution $g_{\alpha_{g}}\otimes g_{\alpha_{g}}$ of the ordinary $p$-stabilization of $g$ with itself. It follows from \cite[Corollary~4.1.3]{LZ2} (see also Proposition~\ref{prop:thedichotomy} below) that $\BF_{m,\chi}^{g,g}\in H^1_\Iw(\QQ(m),\Sym^2 W_g^*(1+\chi))$. Furthermore, \cite[Theorem 4.1.6]{LZ2} shows that the collection of classes $\lbrace \BF_{m,\chi}^{g,g} \rbrace_{m \in \mathcal{N}_{\chi}}$ give rise to an (integral) Euler system for $\Sym^2 W_g^*(1+\chi)$.

In order to verify that the aforementioned Euler system is non-trivial, Loeffler and Zerbes make use of the explicit reciprocity laws for the Beilinson--Flach elements (c.f. \cite{KLZ2}, Theorem B) and Dasgupta's factorization formula in \cite{dasgupta1}. According to \cite[Theorem 4.2.5]{LZ2}, the image of the localization of $\BF_{1,\chi}^{g,g}$ at $p$ under the Perrin--Riou regulator map is a non-zero scalar multiple of the product
\[
 L_{p}(\mathrm{Sym}^{2} g \otimes \chi^{-1}, s)\,L_{p, NN_{\chi}}(\chi^{-1}\epsilon_g, s - k_g -1),
\]
where $L_{p}(\mathrm{Sym}^{2} g \otimes \chi^{-1}, s)$ is the (\emph{bounded}) symmetric square $p$-adic $L$-function of Schmidt. 
The final main result in \cite{LZ2} is a one-sided inclusion in the Iwasawa main conjecture for (twists of) $p$-ordinary symmetric square motives (c.f. \cite{LZ2}, Theorem 5.4.2). Theorem~\ref{thmC} stated above is the non-ordinary analogue of this result where Pottharst-style analytic Selmer groups appear as the non-ordinary counterparts of Greenberg's Selmer groups in the ordinary case.

As we have remarked in the earlier portions of this introduction, we  prove Theorem~\ref{thmC} through the doubly-signed Iwasawa main conjectures, which are the subject of Theorem~\ref{thm:intro} and  require the (integral) signed Beilinson--Flach Euler systems as an input. In the non-ordinary case,  Loeffler and Zerbes  \cite{LZ1} constructed four families of Beilinson--Flach classes, depending on the choices of $p$-stabilizations in the Rankin--Selberg product. While these classes are no longer integral, it is predicted that there exists an integral rank-2 Euler system, whose rank reduction via the Perrin-Riou functionals  gives rise to all four collections of unbounded  Beilinson--Flach classes (see \cite[Conjecture 3.5.1]{BLLV} for a precise formulation of this prediction). On generalizing  the plus and minus Iwasawa theory for modular forms (as developed in \cite{kobayashi03,pollack03,lei09}) to the Rankin--Selberg setting, we may decompose the non-integral Perrin-Riou functionals into integral signed functionals (see \S\ref{sec:signediwasawatheory} for details). The existence of an integral rank-2 Euler system would then give the factorization of the non-integral Beilinson--Flach classes into integral signed classes, as stated in Theorem~\ref{thm_theoremA} above  (we refer the reader to \cite[Conjecture~5.3.1]{BLLV}, where the existence of such integral classes is discussed for more general Rankin--Selberg products).

The main task in the present article is  to obtain integral signed Beilinson--Flach classes building on our earlier joint work \cite{BLLV} with Loeffler, where we have obtained a partial decomposition of the non-integral Beilinson--Flach classes. More precisely,  Theorem~ 5.4.1 in op. cit. exploits the interpolative properties of the Beilinson--Flach classes  at the twists $W_f^*\otimes W_f^*(1+\chi-j)$ with $j=1,2,\ldots, k+1$ to show that certain linear combinations of these classes are divisible by twists of Pollack's plus and minus logarithms $($see Lemma~\ref{lem:firsttwists} below$)$.

This divisibility  originating from the  interpolative properties  of Beilinson--Flach classes alone does not give integral classes in the setting of symmetric squares, because the denominators of the non-integral classes are bigger than those of the said plus and minus logarithms. In order to establish Theorem~\ref{thm_theoremA}, it is necessary to study the images of  Beilinson--Flach classes for the twists $W_f^*\otimes W_f^*(1+\chi-j)$ with $j=k+2,\ldots, 2k+2$, which is outside the geometric range. In this non-geometric range, a direct comparison of Beilinson--Flach classes (for different $p$-stabilizations) is no longer possible.

 The main technical component of the current paper (presented in Section~\ref{sec:analyticmainconjectures}) relies on the theory of $(\vp,\Gamma)$-modules, Selmer complexes and reciprocity laws satisfied by Beilinson--Flach classes to study the properties of these classes outside the geometric range. In particular, we show that the characteristic ideals of certain analytic Selmer groups are coprime to the factors of logarithmic functions corresponding to the twists in the non-geometric range. This allows us to prove \cite[Conjecture 3.5.1]{BLLV} up to a controlled error term and in turn, deduce Theorem~\ref{thm_theoremA}.
This refinement is the novel technical development in the present work, which is in contrast to  previous works (for example, \cite{wansupersingularellipticcurves,sprung16,BL16b}) where it suffices to exploit interpolative properties  to obtain signed integral classes.
}

\subsection*{Acknowledgements} The authors thank David Loeffler for many enlightening conversations and the anonymous referee for valuable suggestions and remarks on an earlier draft of this paper. 

\section{Beilinson--Flach elements for symmetric squares} 
\subsection{Twisted Beilinson--Flach elements}  \label{subsec:plocalBF}
  For $\lambda,\mu\in\{\pm\alpha\}$, $c > 1$ coprime to $6Np$, $m \ge 1$ coprime to $pc$, and {$a{\in} (\ZZ / m\ZZ)^\times\times  \Zp^\times$}, let
   \[
    {}_c\BF^{\lambda,\mu}_{m, a} \in
    H^1(\QQ(\mu_m),W_{f}^* \otimes W_{f}^*\otimes\cH_{E,k+1}(\Gamma)^\iota)
   \]
   be the Beilinson--Flach element constructed in \cite[Theorem 5.4.2]{LZ1}.
  
\begin{remark} Note that ${}_c\BF^{\lambda,\mu}_{m, a}$ are built out of ``non-$p$-stabilized classes" (denoted by ${}_c\mathcal{BF}^{[f,g,j]}_{mp^r,a}$ in \cite{LZ1}), which are defined over $E$. The $p$-stabilized Beilinson--Flach classes are given by  ${}_c\mathcal{BF}^{[f,g,j]}_{mp^r,a}$ multiplied by $(\lambda\mu)^{-r}$ or $1-p^j\sigma_p/\lambda\mu$ {(where $\sigma_p$ is the Frobenius at $p$)} depending on whether $r>0$ or $r=0$ respectively. They are therefore still defined over $E$ since $\lambda\mu \in E$ by assumption.
\end{remark}

We shall take $a=1$ throughout and omit it from the notation. Let $\chi$ be a Dirichlet character of conductor $N_\chi$, which we assume to be prime to $p$. Enlarging the number field $L$ if necessary, we assume that $\chi$ takes values in $E$. Let $\mathcal{R}_{\chi}$ denote the set of positive {square-free} integers prime to $6pNN_{\chi}$. For $m \in \mathcal{R}_{\chi}$, we may consider $\chi$ as a continuous character of $\Gal(\QQ(\mu_{mN_{\chi}p^\infty})/\QQ) \cong(\ZZ/mN_{\chi}p^{\infty}\ZZ)^{\times}$. 
\begin{defn}
\label{define:twistedBF}
 For all $m \in \mathcal{R}_{\chi}$, we define the Beilinson--Flach class twisted by the character $\chi$
  $${}_{c}\BF_{m,\chi}^{\lambda,\mu}\in H^1(\QQ(m),W_{f}^* \otimes W_{f}^*(1 + \chi)\otimes\cH_{E,k+1}(\Gamma)^\iota)$$ 
  by setting it as the image of ${}_{c}\BF_{m}^{\lambda,\mu}$ under 
  \begin{align*}H^1(\QQ(\mu_{mN_\chi}),W_{f}^* \otimes W_{f}^*\otimes\cH_{E,k+1}(\Gamma)^\iota)&\cong H^1(\QQ(\mu_{mN_\chi}),W_{f}^* \otimes W_{f}^*(1 + \chi)\otimes\cH_{E,k+1}(\Gamma)^\iota)\\
  &\stackrel{\textup{cor}}{\lra}H^1(\QQ(m),W_{f}^* \otimes W_{f}^*(1 + \chi)\otimes\cH_{E,k+1}(\Gamma)^\iota),
  \end{align*}
{where the first isomorphism is  the natural twisting map given by $\chi$.}  
  \end{defn}

\begin{proposition}
\label{prop:thedichotomy}
We have the dichotomy
\[
{}_{c}\BF_{m,\chi}^{\lambda,\lambda}\in\begin{cases}
H^1(\QQ(m),\Sym^2 W_f^*(1+\chi)\otimes\cH_{E,k+1}(\Gamma)^\iota)&\text{if }\chi(-1)= +1;\\
H^1(\QQ(m),\bigwedge^2 W_f^*(1+\chi)\otimes\cH_{E,k+1}(\Gamma)^\iota)&\text{if }\chi(-1)= -1.
\end{cases}
\]
Moreover, 
\[
{}_{c}\BF_{m,\chi}^{\lambda,\mu}+{}_{c}\BF_{m,\chi}^{\mu,\lambda}\in\begin{cases}
H^1(\QQ(m),\Sym^2 W_f^*(1+\chi)\otimes\cH_{E,k+1}(\Gamma)^\iota)&\text{if }\chi(-1)= +1;\\
H^1(\QQ(m),\bigwedge^2 W_f^*(1+\chi)\otimes\cH_{E,k+1}(\Gamma)^\iota)&\text{if }\chi(-1)= -1,
\end{cases}
\]
\[
{}_{c}\BF_{m,\chi}^{\lambda,\mu}-{}_{c}\BF_{m,\chi}^{\mu,\lambda}\in\begin{cases}
H^1(\QQ(m),\bigwedge^2 W_f^*(1+\chi)\otimes\cH_{E,k+1}(\Gamma)^\iota)&\text{if }\chi(-1)= +1;\\
H^1(\QQ(m),\Sym^2 W_f^*(1 + \chi)\otimes\cH_{E,k+1}(\Gamma)^\iota)&\text{if }\chi(-1)= -1.
\end{cases}
\]
\end{proposition}
\begin{proof}
The proof of \cite[Corollary~4.1.3]{LZ2} goes through verbatim.
\end{proof}
For the rest of the article, we shall fix an even Dirichlet character $\chi$ of conductor $N_{\chi}$, which is assumed to be coprime to $p$. We note that we will not solely work with Beilinson--Flach classes twisted by $\chi$ but by a range of Dirichlet characters.
\subsection{Imprimitive $L$-functions and $p$-adic $L$-functions}\label{sec:p-adicLfunctions} 

In this section, we introduce the geometric $p$-adic $L$-function mentioned in the hypothesis (\textbf{NV}) in the introduction. Furthermore, we discuss the validity of the hypothesis (\textbf{NV}) and its consequence on the non-vanishing of the Beilinson--Flach elements introduced in \S\ref{subsec:plocalBF}.

 For $\lambda\in\{\pm\alpha\}$, let $f_\lambda$ be the $p$-stabilization at $\lambda$. Let $(\cF,\mathcal{\epsilon}_{\cF})$ be the Coleman family, defined over some affinoid disc $U$ in the weight space $\mathcal{W}$, passing through $f_\lambda$. {For  any affinoid $V$, we let $A(V)$ denote the ring of rigid analytic functions on $V$.} Loeffler and Zerbes in \cite[Definition 9.1.1]{LZ1} define a three-variable geometric $p$-adic $L$-function, $L_{p}^{\mathrm{geom}}(\cF,\cF\otimes \chi^{-1}) \in {A(U\times U\times\mathcal{W})}$. On restricting $L_{p}^{\mathrm{geom}}(\cF,\cF\otimes \chi^{-1})$ to the image of $U\times\mathcal{W} {\hookrightarrow} U\times U\times\mathcal{W}$ induced by the diagonal embedding $\Delta: U\hookrightarrow U\times U$, we will henceforth treat it as an element of $A(U\times\mathcal{W})$.
\begin{defn}
Let $L_{p,NN_{\chi}}(\chi^{-1}\mathcal{\epsilon}_{\cF}) \in {A(\mathcal{W})}$ denote the Kubota--Leopoldt $p$-adic $L$-function that interpolates the values of the Dirichlet $L$-series $L_{NN_{\chi}}(-)$ with the Euler factors at primes dividing $NN_{\chi}$ removed. We define the \emph{geometric symmetric square $p$-adic $L$-function} $L_{p}^{\mathrm{geom}}(\Sym^{2} \mathcal{F} \otimes \chi^{-1})\in {\mathrm{Frac}(A(U \times \mathcal{W}))}$  by setting
\[L_{p}^{\mathrm{geom}}(\Sym^{2} \mathcal{F} \otimes \chi^{-1})(\kappa, \sigma):=\frac{L_{p}^{\mathrm{geom}}(\cF,\cF\otimes \chi^{-1})(\kappa,\sigma)}{ L_{p,NN_{\chi}}(\chi^{-1}\mathcal{\epsilon}_{\mathcal{F}})(\sigma-\kappa+1)}\,.\]
\end{defn}
In particular, on restricting this definition to $(k+2,s) \in U \times \mathcal{W}$, we have 
\begin{equation}
\label{eqn_define_padicL_sym2_weight_k}
 L_{p}^{\mathrm{geom}}(\mathrm{Sym}^{2} f_\lambda \otimes \chi^{-1}, s)=\frac{L_{p}^{\mathrm{geom}}(f_{\lambda}, f_{\lambda} \otimes \chi^{-1}, s)} {L_{p, NN_{\chi}}({\chi^{-1}}\epsilon_f, s - k -1)}\,. 
\end{equation}
We remark that our ad hoc definition~\eqref{eqn_define_padicL_sym2_weight_k} of the $p$-adic $L$-function $L_{p}^{\mathrm{geom}}(\mathrm{Sym}^{2} f_\lambda \otimes \chi^{-1}, s)$ is based on the Artin-formalism and reflects the decomposition 
\begin{equation}
\label{eqn_decompose_duals}
W_f\otimes W_f\otimes\chi^{-1}=\left(\Sym^2W_f^*\otimes\chi^{-1}\right)\oplus\left( {\bigwedge\!}^2W_f\otimes\chi^{-1}\right)\,.
\end{equation} 
The following (forthcoming) result of Arlandini, which extends \cite[Theorem~1]{dasgupta1} to our case of interest, relates $L_{p}^{\mathrm{geom}}(\mathrm{Sym}^{2} f_\lambda \otimes \chi^{-1}, s)$ to the complex $L$-values and serves as a justification as to why $L_{p}^{\mathrm{geom}}(\mathrm{Sym}^{2} f_\lambda \otimes \chi^{-1}, s)$ deserves to be called a $p$-adic $L$-function. Before stating Arlandini's result, we first introduce the following notation.

\begin{defn}
For $\chi$ as above, we let $L^{\textup{imp}}(\textup{Sym}^2f\otimes\chi,s)$ denote the imprimitive $L$-function given as in \cite[Definition 2.1.3]{LZ1}.  
\end{defn}

\begin{theorem}[Arlandini, forthcoming] 
\label{thm_DLZfactorizationtheorem}
We have $L_{p}^{\mathrm{geom}}(\Sym^{2} \mathcal{F} \otimes \chi^{-1}) \in A(U \times \mathcal{W})$
and it verifies the following interpolation property.
\item[i)] Let $1 \leq j \leq k + 1$ be an odd integer. Then
\[ L_{p}^{\mathrm{geom}}(\Sym^{2} \mathcal{F} \otimes \chi^{-1})(k+2,j) = \frac{(-1)^{j - k - 1}j!}{2^{2k +4}i^{a}} \mathcal{E}_{p}(j) \frac{L^{\mathrm{imp}}(\mathrm{Sym}^{2}f \otimes \chi^{-1}, j)}{(2\pi i)^{j - k - 1} \pi^{k + 1} \langle f, f \rangle}\,, \]
where $a = 0$ if $k$ is even and $a = 1$ if $k$ is odd. The Euler factor $\mathcal{E}_{p}(j)$ is given by 
\begin{equation}
\label{eqn_Arlandini_1}
 (1 - p^{j-1}\chi(p)\lambda^{-2})(1 +\chi^{-1}(p)\lambda^2 p^{-j})(1 - \chi^{-1}(p)\lambda^{2}p^{-j})\,. 
\end{equation}
\item[ii)] Let $k + 2 \leq j \leq 2k + 2$ be an even integer. Then
\[ L_{p}^{\mathrm{geom}}(\Sym^{2} \mathcal{F} \otimes \chi^{-1})(k+2,j) = \frac{(j-k-1)!j!}{2^{2j + 1}} \mathcal{E}'_{p}(j) \frac{L^{\mathrm{imp}}(\mathrm{Sym}^{2}f \otimes \chi^{-1}, j)}{\pi^{2j-k-1} \langle f, f \rangle}\,, \]
where the Euler factor $\mathcal{E}'_{p}(s)$ is given by 
\begin{equation}
\label{eqn_Arlandini_2} (1 - p^{j-1}\chi(p)\lambda^{-2})(1 + p^{j-1}\chi(p)\lambda^{-2})(1 - \chi^{-1}(p)\lambda^{2}p^{-j})\,. 
\end{equation}
\end{theorem} 

The imprimitive $L$-values that appear on the right-hand side of the interpolation formulae in part (ii) has the following non-vanishing property.

\begin{theorem}[Gelbart--Jacquet \cite{GJ78}, Jacquet--Shalika \cite{JS76}, Schmidt \cite{schmidt}]
\label{thm:nonvanishingofSym2Lvalues}
Suppose that $f$ has minimal level among its {twists} by Dirichlet characters. Then for every integer $j\geq k+2$, we have
$$L^{\textup{imp}}(\textup{Sym}^2f\otimes\chi^{-1},j)\neq 0.$$
\end{theorem}
\begin{proof}
The  \emph{primitive} $L$-function $L(\textup{Sym}^2f\otimes\chi^{-1},s)$ is non-zero at integers $j>k+2$ since the Euler product defining these $L$-functions converges absolutely in that range. Considering the Gelbart--Jacquet lift of $\textup{Sym}^2_f$ to $\textup{GL}_3$ as given in {\cite[\S 3]{GJ78}} and making use of a non-vanishing result for $\textup{GL}_n$ automorphic $L$-functions due to Jacquet and Shalika {\cite[Theorem (1.3)]{JS76}}, it follows that $L(\textup{Sym}^2f\otimes\chi^{-1},k+2)\neq 0$ as well.

The desired non-vanishing for $L^{\textup{imp}}(\textup{Sym}^2f\otimes\chi^{-1},j)$ for $j\ge k+2$ now follows from \cite[Lemmas 1.5 and 1.6]{schmidt}. More specifically, our  hypothesis that $f$ has minimal level among its twists by Dirichlet characters implies that the quotient 
$L^{\textup{imp}}(\textup{Sym}^2f\otimes\chi^{-1},s)/L(\textup{Sym}^2f\otimes\chi^{-1},s)$ is an entire function with zeroes only on the line $\textup{Re}(s)=k+1$. The result follows.
\end{proof}

{One immediate consequence of Theorems~\ref{thm_DLZfactorizationtheorem} and \ref{thm:nonvanishingofSym2Lvalues} is the following result on the existence of exceptional zeros and the non-vanishing of the geometric $p$-adic $L$-function of $\Sym^2f_\lambda$. }  

\begin{corollary}[Exceptional Zeroes] \label{coro_trivialzeroes}
{The $p$-adic $L$-function $L_{p}^{\mathrm{geom}}(\mathrm{Sym}^{2} f_\lambda \otimes \chi^{-1}, s)$ has an exceptional zero at $s$ if and only if $\epsilon_{f}\chi^{-1}(p) = \pm1$ and $s = k + 1$ or $s = k + 2$. In particular, if we assume \ref{item_Psi_2}, then $L_{p}^{\mathrm{geom}}(\mathrm{Sym}^{2} f_\lambda \otimes \chi^{-1}, s)$ is non-vanishing for all even integers $k + 2 \leq s \leq 2k + 2$.}
\end{corollary}
\begin{proof}
{By weight considerations, the interpolation factors \eqref{eqn_Arlandini_1} and \eqref{eqn_Arlandini_2} cannot vanish unless $j=k+1$ or $j=k+2$.} Since $\lambda^2 = -\epsilon_{f}(p)p^{k+1}$, we have
\[ 1 \pm \frac{\chi^{-1}(p)\lambda^2}{p^{k+1}} = 1 \mp {\chi^{-1}(p)\epsilon_{f}(p)} \]
and hence the Euler factor $\mathcal{E}_{p}(k+1)$ vanishes only when $\epsilon_{f}\chi^{-1}(p) = \pm1$.
Similarly, $\mathcal{E}'_{p}(k + 2) = 0$ only when $\epsilon_{f}\chi^{-1}(p) = \pm1$. {The second} part of the corollary follows from Theorems \ref{thm:nonvanishingofSym2Lvalues} and ~\ref{thm_DLZfactorizationtheorem}(ii). 
\end{proof} 


{We now explain the link between the geometric $p$-adic $L$-functions and Beilinson--Flach elements.}
{Let
\[
\mathcal{L} : H^{1}_{\Iw}(\QQ_p, \Sym^{2} W_{f}^{*}(1 + \chi)) \rightarrow \mathcal{H}_{E,k+1}(\Gamma)\otimes\mathbb{D}_{\mathrm{cris}}(\Sym^{2} W_{f}^{*}(1 + \chi)) 
\]
denote the Perrin-Riou regulator map as given  in \cite[\S 3.1]{leiloefflerzerbes11} and \cite[Appendix B]{LZ-IJNT}.}
\begin{proposition} \label{prop:explicitreciprocitylaw}
Let $\xi_{f_{\lambda}, \chi} \in \mathbb{D}_{\mathrm{cris}}(\Sym^{2} W_{f}\otimes\chi^{-1})$ be the vector chosen as in \cite[Definition 4.2.4]{LZ2}. We then have,
{\begin{align*} 
\Big \langle \mathcal{L}({}_{c}\BF^{\lambda, \lambda}_{1, \chi}), \xi_{f_{\lambda}, \chi} \Big \rangle =& (-1)^{s}(c^2-c^{2s-2k-2}\chi^{2}(c)\epsilon_{f}^{-2}(c))G(\chi^{-1})^{2}G(\epsilon_{f}^{-1})^2 L_{p}^{\mathrm{geom}}(f_{\lambda}, f_{\lambda} \otimes \chi^{-1}, s) \\
=& (-1)^{s}(c^2-c^{2s-2k-2}\chi^{2}(c)\epsilon_{f}^{-2}(c))G(\chi^{-1})^{2}G(\epsilon_{f}^{-1})^2 \\
&\,\,\,\,\,\,\,\,\,\,\,\,\,\, \,\,\,\,\,\,\,\,\,\,\,\,\,\,\,\,\,\,\,\,\,\,\,\,\,\,\,\,\times L_{p}^{\mathrm{geom}}(\mathrm{Sym}^{2} f_\lambda \otimes \chi^{-1}, s)\,L_{p, NN_{\chi}}(\chi^{-1}\epsilon_f, s - k -1)\,. 
\end{align*} }
\end{proposition}
\begin{proof}
This is the same as \cite[Theorem 4.2.5]{LZ2}. 
\end{proof}

\begin{lemma} \label{lem:junkfactors}
 For all even integers $j \in \{ k+2,\ldots, 2k + 2 \}$, there is a choice of $c > 1$ coprime to $6NN_{\chi}p$ for which the product
 $(c^2-c^{2j-2k-2}{\chi^2(c)\epsilon_{f}^{-2}(c)}){L_{p, NN_{\chi}}(\chi^{-1}\epsilon_f, j - k -1)}$ is {non-zero whenever $\chi^{-2}\epsilon_{f}^{2}\neq \mathds{1}$ or $j \neq k+2$. }
\end{lemma}
\begin{proof} 
This follows from \cite[Proposition 4.3.1]{LZ2}.
\end{proof}
\begin{remark}
{Since we assume that the conductor $N_\chi$ of $\chi$ is coprime to $N$, it follows that $\chi^{-2}\epsilon_{f}^{2} =\mathds{1}$ only when both $\chi$ and $\epsilon_f$ have order dividing $2$. Our hypothesis \ref{item_Psi_1} rules out this possibility.}
\end{remark}
\begin{corollary} \label{cor:nontriviality}
{Assume that \ref{item_Psi_2} holds true and either that $\chi^{-2}\epsilon_{f}^2\neq \mathds{1}$ or $j \neq k+2$. Then for all even integers $k+2 \leq j \leq 2k+2$, the image of ${}_{c}\BF^{\lambda, \lambda}_{1, \chi}$ in $H^{1}(\QQ_{p}, \Sym^2 W_{f}^{*}(1 + \chi)(-j))$ is non-zero. In particular, if \ref{item_Psi_1} and \ref{item_Psi_2} hold true, then the class $\res_p({}_{c}\BF^{\lambda, \lambda}_{1, \chi}) \in H^{1}_{\Iw}(\QQ_p(\mu_{p}), \Sym^{2} W_{f}^{*}(1 + \chi))$ is non-zero. }
\end{corollary}
\begin{proof}
By Proposition ~\ref{prop:explicitreciprocitylaw} and Lemma ~\ref{lem:junkfactors}, we only have to verify  the non-vanishing of the quantity $L_{p}^{\mathrm{geom}}(\mathrm{Sym}^{2} f_\lambda \otimes \chi^{-1}, j)$. {This is immediate from Corollary~\ref{coro_trivialzeroes}.}
\end{proof}

\begin{remark} \label{remark:fixingc}
\item[i)] When $j=k+2$ is even and if $\chi\epsilon_{f}$ is quadratic but non-trivial, then neither can we dispense off with the factor ${(c^2-c^{2j-2k-2}\chi^{2}(c)\epsilon_{f}^{-2}(c))}$, nor can we use it to cancel a pole in the $p$-adic $L$-series (since no such pole exists). 
\item[ii)]{By Lemma ~\ref{lem:junkfactors}, we may choose the {auxiliary} integer $c$ for which ${(c^2-c^{2j-2k-2}\chi^{2}\epsilon_{f}^{-2}(c))}$ is non-zero for $k+2 \leq j \leq 2k+2$ assuming \ref{item_Psi_1}. For the rest of the paper, we fix such a value and dispense  with the factor $c$ from the notation.}
\end{remark}

We now proceed to show the non-triviality of the classes $\BF_{1,\chi}^{\lambda,-\lambda}$ for $\lambda \in \{\pm\alpha\}$ using anti-symmetry relations in slight variations of Proposition~\ref{prop:explicitreciprocitylaw}. 

\begin{corollary} 
The class $\res_p(\BF^{\lambda,-\lambda}_{1, \chi}) \in H^{1}_{\Iw}(\QQ_p(\mu_{p}), \Sym^{2} W_{f}^{*}(1+\chi))$ is non-zero.
\end{corollary}
\begin{proof}
Let $v_{f,\lambda} \in \mathbb{D}_{\mathrm{cris}}(W_{f})$ be the $\varphi$-eigenvector as chosen in \cite[Section 3.5]{BLLV}. Also set 
$$v_{\lambda,-\lambda,\chi}:=G(\chi^{-1})\,v_{f,\lambda} \otimes v_{f,-\lambda}\,\,,\,\,v_{\lambda,\lambda,\chi}:=G(\chi^{-1})\,v_{f,\lambda} \otimes v_{f,\lambda} \in  \mathbb{D}_{\mathrm{cris}}(W_{f} \otimes W_{f}\otimes\chi^{-1})\,.$$ By \cite[Theorem 3.6.5]{BLLV}, we have
\begin{equation} \label{eqn:antisymmetryrelation}
\Big\langle\mathcal{L}(\BF^{\lambda,\lambda}_{1,\chi}), v_{\lambda,-\lambda,\chi}\Big\rangle = -\Big\langle\mathcal{L}(\BF^{\lambda,-\lambda}_{1,\chi}), v_{\lambda,\lambda,\chi}\Big\rangle = \frac{A_{f}\,\mathrm{log}_{p,k+1}^{(1)}}{2\lambda}L_{p}^{\mathrm{geom}}(f_{\lambda},f_{\lambda} \otimes \chi^{-1},s) ,
\end{equation}
where $A_{f}$ is a non-zero constant independent of $\lambda$ and $\mathrm{log}_{p,k+1}^{(1)}$ is {a non-zero logarithmic function (see  (\ref{eqn:log(1)}) below for an explicit description of this function).} 

We have seen in the proof of Corollary~\ref{cor:nontriviality} that  $L_{p}^{\mathrm{geom}}(f_{\lambda},f_{\lambda} \otimes \chi^{-1},s)$ is non-zero. Hence, \eqref{eqn:antisymmetryrelation} tells us that  $\res_p({}_{c}\BF^{\lambda,-\lambda}_{1, \chi}) \in H^{1}_{\Iw}(\QQ_p(\mu_{p}), \Sym^{2} W_{f}^{*}(1+\chi))$ is non-zero.  
\end{proof} 

\section{Structure of elementary Selmer modules}
\label{sec_selmerstructure}
{Our main objective in this section is to prove Theorem~\ref{thm_mainSelmerstructure}, where we determine the structure of certain Iwasawa theoretic Selmer groups. The main ingredient is the \emph{horizontal} Beilinson--Flach Euler system, which we use to obtain our key technical input (Theorem~\ref{thm_horizontalESforSym2}) in the proof of Theorem~\ref{thm_mainSelmerstructure}.}
\subsection{Set-up}
Throughout Section~\ref{sec_selmerstructure}, we fix an even Dirichlet character $\psi$ of conductor $N_\psi$ co-prime to $p$. Enlarging $L$ if necessary, we shall assume that $\psi$ may be realized over $L$. We will take $\psi$ to be $\chi\nu$, where $\chi$ is the character fixed in Section~\ref{subsec:plocalBF} whereas $\nu$ is some Dirichlet character of conductor prime to $pNN_\chi$ and $p$-power order.  

We shall assume the validity of the following \emph{big image} hypothesis throughout Section~\ref{sec_selmerstructure}:
\begin{itemize}
\item[\ref{item_Im}] $\textup{im}\left(G_\QQ\ra \textup{Aut}(R_f\otimes\QQ_p)\right)$ contains a conjugate of $\textup{SL}_2(\ZZ_p)$. 
\end{itemize}
We will consider the following conditions on $\psi$ and $f$:
\begin{itemize}
\item[\ref{item_Psi_1}] There exists $u \in (\ZZ/NN_\psi\ZZ)^\times$ such that $\epsilon_f\psi^{-1}(u) \not\equiv \pm1\, (\textup{mod}\, \frak{p})$ and $\psi(u)$ is a square modulo $\frak{p}$. 
\item[($\Psi_2^\prime$)] $\epsilon_f\psi^{-1}(p)\neq \pm 1$.
\end{itemize}
\begin{lemma}
\label{lemma_propagate_hypo_psi_2}
Suppose that $\chi$ satisfies the hypothesis \ref{item_Psi_1} and also that
\begin{itemize}
\item[\ref{item_Psi_2}] $\epsilon_f\chi^{-1}(p)\neq \pm1$ and $\phi(N)\phi(N_\chi)$ is coprime to $p$, where $\phi$ is Euler's totient function. 
\end{itemize}
Then the conditions \ref{item_Psi_1} and $(\Psi_2^\prime)$ hold true for any $\psi=\chi\nu$ where $\nu$ is a Dirichlet character of  conductor prime to $NN_\chi$ and $p$-power order. 
\end{lemma}
\begin{proof}
Let $u$ be an integer satisfying \ref{item_Psi_1} with $\psi=\chi$ and that $u\equiv 1$ mod $N_\nu$ (such $u$ exists by {the Chinese remainder theorem}). The chosen $u$ will verify \ref{item_Psi_1} with $\psi=\chi\nu$. We now check $(\Psi_2^\prime)$ for $\psi=\chi\nu$. If it was the contrary, we would then have that $\epsilon_f\chi^{-1}(p)=\pm\nu(p)$. This would mean that either $\epsilon_f\chi^{-1}(p)$ or $-\epsilon_f\chi^{-1}(p)$ is a $p$-power root of unity, contradicting \ref{item_Psi_2}.
\end{proof}
We end this subsection with a general definition.
Recall that $\Lambda_{\cO}(\Gamma)^{\iota}$ is the free $\Lambda_{\cO}(\Gamma)$-module of rank one on which $G_{\QQ}$ acts via the inverse of the canonical character $G_{\QQ} \twoheadrightarrow \Gamma \hookrightarrow \Lambda_{\cO}(\Gamma)^{\times}$.

\begin{defn}
\label{def:canonicalselmerstructure}
Let $K$ be any number field. 
\item[i)] Given an arbitrary free $\cO$-module $M$ of finite rank that is endowed with a continuous $G_K$-action unramified outside a finite set of places of $K$, we let $\mathcal{F}_{\textup{can}}$ denote the canonical Selmer structure on $M$ $($or $M\otimes_{\ZZ_p}\QQ_p$$)$, given as in \cite[Definition 3.2.1]{mr02}. 
\item[ii)] We let $\mathcal{F}_{\textup{can}}^*$ denote the dual Selmer structure on $M^\vee(1)$ $($or on $(M\otimes_{\ZZ_p}\QQ_p)^*(1)$$)$, defined as in Section 2.3 of loc. cit. 
\item[iii)]We write $\mathcal{F}_{\textup{can}}$ for the Selmer structure on $\mathbb{M}:=M\otimes\Lambda_{\cO}(\Gamma)^\iota$, denoted by $\mathcal{F}_\LL$ in Section 5.3 of loc. cit. and let $\mathcal{F}_{\textup{can}}^*$  denote the Selmer structure $\mathcal{F}_\LL^*$ of loc. cit. on the Galois representation $\mathbb{M}^\vee(1)$.
\end{defn}
\begin{remark}
For $K$ and $\mathbb{M}$ as above, we have  
$$H^1_{\mathcal{F}_{\textup{can}}}(K,\mathbb{M})=H^1(K,\mathbb{M})$$ 
by \cite[Lemma 5.3.1]{mr02}. This in turn means that $H^1_{\mathcal{F}_{\textup{can}}^*}(K,\mathbb{M}^\vee(1))$ consists of classes which are locally trivial everywhere.
\end{remark}

\subsection{Twists of the symmetric and the alternating squares}
{In this subsection, we shall introduce various twists of the symmetric and alternating square representations associated to $f$, and study their Galois theoretic properties.}
\begin{defn}
\label{defn_twists_Gal_rep}
 Recall the even Dirichlet character $\psi$, so that the character $\chi_\cyc\psi$ is odd, where $\chi_{\cyc}$ is the $p$-adic {cyclotomic character}. 
\item[i)] We set $T_\psi:=\textup{Sym}^2R_f^*(1)\otimes\psi$, so that $T_\psi^*(1)=\textup{Sym}^2R_f\otimes\psi^{-1}$. \item[ii)] Choose an arbitrary integer $j \in [k+2,2k+2]$ and put
$$T_{\psi,j}:=T_\psi(-j)\otimes\omega^j =\textup{Sym}^2R_f^*(1-j)\otimes\omega^{j}\psi$$ 
(where $\omega$ is the Teichm\"uller character). We remark that the character $\chi_{\cyc}^{1-j}\omega^j\psi$ is always odd. 
\item[iii)] We finally set 
$$X_{\psi,j}:={\bigwedge\!}^2 R_f^*(1-j)\otimes \omega^j\psi\cong \mathcal{O}(k+2-j)\otimes\omega^j\epsilon_f\psi$$
and observe that the character $\chi_{\cyc}^{k+2-j}\omega^j\epsilon_f\psi$ is even.
\end{defn} 
\begin{proposition}
\label{Prop_TauIsGood}
Suppose that $\psi$ satisfies \ref{item_Psi_1} and $(\Psi_2^\prime)$. Then there exists $\tau \in G_\QQ$ with the following properties:
\begin{itemize}
\item $\tau$ acts trivially on $\mu_{p^\infty}$.
\item $T_{\psi,j}/(\tau-1)T_{\psi,j}$ is free of rank one.
\item $\tau-1$ acts invertibly on $X_{\psi,j}^*(1)$.
\end{itemize}
\end{proposition}
\begin{proof}
This is exactly \cite[Proposition 5.2.1]{LZ2}, {where we set $\psi\omega^j$ in place of their $\psi$}. 
\end{proof}

\subsubsection{Selmer groups of the alternating squares}\label{subsubsec_alt_square}
{Our task in Section~\ref{subsubsec_alt_square} is to prove that the Selmer groups for the twists $X_{\psi,j}$ of the alternating square vanishes.}
\begin{proposition}
\label{prop:prop:vanishingofX1}
$H^1_{\mathcal{F}_{\textup{can}}}(\QQ,X_{\psi,j})=0$.
\end{proposition}
\begin{proof}
We first prove the case when $j=k+2$. We remark that already this much will be sufficient for our purposes. In this situation, $X_{\psi,k+2}=\cO(\epsilon_f\psi\omega^{k+2})$ and the conclusion follows from the validity of Leopoldt's Conjecture for abelian number fields and the fact that $\epsilon_f\psi\omega^{k+2}$ is even.

Suppose now that $j\geq k+3$. To ease notation, we set $\eta=\omega^{k+2}\epsilon_f\psi$ and $\rho:=\chi_\cyc^{k+2-j }\omega^{-k-2+j}$. Notice that $\eta$ is an even character and $\rho$ is a character of $\Gamma$. Furthermore, we have an isomorphism 
$$X_{\psi,j} \cong \cO(\eta)\otimes \rho\,.$$ 
which, together with the twisting theorems of \cite[Section 6]{rubin00}, control theorem for the canonical Selmer structure on $X_{\psi,j}$ and the truth of the Main Conjectures for abelian fields, reduces the desired vanishing of the Selmer group to the verification that 
$$L\left(\omega^{(k+1)-j}\eta, k+2-j\right)\neq 0\,.$$
{By the functional equation for Dirichlet $L$-series, this is equivalent to the requirement that 
\begin{itemize}
\item $L\left(\omega^{j-(k+1)}\eta^{-1}, j-k-1\right)\neq 0$, and
\item $\Gamma(s)$ is holomorphic at $s=\dfrac{k+2-j+a}{2}$, where $a=\dfrac{(-1)^{k-j}+1}{2} \in\{0,1\}$\,.
\end{itemize} 
The first of these conditions is clear since $ j-k-1\geq 2$ is in the range of absolute convergence, whereas the second follows since $a$ and $k+2-j$ have opposite parity.}
\end{proof}

\begin{corollary}
\label{cor:vanishingofX}
$H^1(\QQ,X_{\psi,j}\otimes \Lambda_{\cO}(\Gamma)^\iota)=0$.
\end{corollary}
\begin{proof}
This follows from Proposition~\ref{prop:prop:vanishingofX1} and  Nakayama's lemma, since we have an injection 
$$H^1(\QQ,X_{\psi,j}\otimes \Lambda_{\cO}(\Gamma)^\iota)_\Gamma\hookrightarrow H^1_{\mathcal{F}_{\textup{can}}}(\QQ,X_{\psi,j})=0.$$
\end{proof}
\begin{remark}
\label{rem:alternativeproofforXpsij}
One might give a direct proof of Corollary~\ref{cor:vanishingofX}, without relying on the Iwasawa main conjectures $($and using our assumptions \ref{item_Psi_1} and $(\Psi_2^\prime)$ on $\psi$$)$. We first note that since the character $\rho$ above factors through $\Gamma$, it suffices to prove that 
$$H^1(\QQ,\mathcal{O}(\eta)\otimes \Lambda_{\cO}(\Gamma)^\iota)=0$$
for $\eta=\omega^{k+2}\epsilon_f\psi$ also as above. Since $\eta$ is an even character, it follows from the weak Leopoldt conjecture for abelian fields (which we know to hold true) that the $\Lambda_{\cO}(\Gamma)$-module $H^1(\QQ,\mathcal{O}(\eta)\otimes \Lambda_{\cO}(\Gamma)^\iota)$ is torsion. Notice further that the character $\eta$ does not factor through the group $\Gamma$ under our running hypotheses and hence the module $H^1(\QQ,\mathcal{O}(\eta)\otimes \Lambda_{\cO}(\Gamma)^\iota)$ is torsion-free. The proof follows.
\end{remark}

\subsubsection{Selmer groups of the symmetric squares}\label{subsubsec_sym_square}
 Our main objective in this subsection is to prove Corollary~\ref{cor:equivariantSelmersize}, where we determine the ranks of the canonical Selmer groups associated to the twists of the symmetric square representations. The key technical input is provided by Theorem~\ref{thm_horizontalESforSym2}, where we utilize the horizontal Beilinson--Flach Euler system.
We first introduce the twisted Galois representations we shall study.
{\begin{defn}
\label{defn_twists_psi_j_chi_T_V}
For any even Dirichlet character $\psi$ as in Definition~\ref{defn_twists_Gal_rep} and integer $j \in [k+2,2k+2]$, we set $V_{\psi,j}:=T_{\psi,j}\otimes_{\ZZ_p}\QQ_p$. We also put $T_j:=T_{\chi,j}$ and $V_j=V_{\chi,j}$, where $\chi$ is the even Dirichlet character we have fixed in Section~\ref{subsec_main_results}. 
\end{defn} }
{Until the end of Section~\ref{subsubsec_sym_square}, we fix $\psi$ and $j$ as in Definition~\ref{defn_twists_psi_j_chi_T_V}.}

\begin{corollary}
\label{cor:symmetricpowercohomequalsym2cohom}
For each $r\in \mathcal{R}_\psi$ we have 
$$H^1_{\mathcal{F}_{\textup{can}}}(\QQ(r),W_f^*\otimes W_f^*(1-j)\otimes\omega^j\psi)=H^1_{\mathcal{F}_{\textup{can}}}(\QQ(r),V_{\psi,j})\,.$$
\end{corollary}
\begin{proof}
For $W:=M\otimes\QQ_p$ as in Definiton~\ref{def:canonicalselmerstructure}, let us write $\mathcal{F}_{\textup{can}}\big{|}_W$ for the canonical Selmer structure on $W$ to emphasize the dependence on $W$. Set $Y_{\psi_j}:=X_{\psi,j}\otimes\QQ_p$ and observe that we have
$$\mathcal{F}_{\textup{can}}\big{|}_{V_{\psi,j}\oplus Y_{\psi,j}}=\mathcal{F}_{\textup{can}}\big{|}_{V_{\psi,j}}\oplus \mathcal{F}_{\textup{can}}\big{|}_{Y_{\psi,j}}$$
 where the direct sum of Selmer structures on the respective direct sum of Galois representations is defined in the obvious manner. This in turn implies (c.f., \cite[Remark 3.1.4]{mr02}) that 
$$H^1_{\mathcal{F}_{\textup{can}}}(\QQ,W_f^*\otimes W_f^*(1-j)\otimes\omega^j\psi)=H^1_{\mathcal{F}_{\textup{can}}}(\QQ,V_{\psi,j})\oplus H^1_{\mathcal{F}_{\textup{can}}}(\QQ,Y_{\psi,j})$$
The asserted identification follows on applying Proposition~\ref{prop:prop:vanishingofX1} with $\psi$ replaced by $\psi\nu$, as $\nu$ runs through the characters of $\Delta_r$ (note that since $p$ is odd and $\Delta_{r}$ is a $p$-group, all characters $\nu$ on $\Delta_{r}$ are necessarily even, which allows us to apply Proposition~\ref{prop:prop:vanishingofX1}). 
\end{proof}

\begin{defn}
\label{def:Kolyvaginprimesfortwists}
Let $\mathcal{P}_{\chi}$ denote the set of primes $\ell \nmid pNN_{\chi}$ for which we have
\begin{itemize}
\item $\ell \equiv 1\mbox{ mod }p$,
\item $T_{\chi,j}/(\mathrm{Frob}_{\ell} - 1)T_{\chi,j}$ is a free $\cO$-module of rank one,
\item $\mathrm{Frob}_{\ell} - 1$ is bijective on $X_{\chi,j}^*(1)$.
\end{itemize}
We let $\mathcal{N}_{\chi}$ denote the set of square-free products of integers in $\mathcal{P}_{\chi}$.
\end{defn}
\begin{remark}
Since we insist that $\ell \equiv 1\mbox{ mod }p$ in Definition~\ref{def:Kolyvaginprimesfortwists}, the remaining conditions hold true for one $j$ if and only if they hold for every $j$. This justifies our choice to denote this set of primes by $\mathcal{P}_{\chi}$.
\end{remark}

\begin{remark}
Let $\overline{T}_{\chi,j}$ denote the residual representation of ${T}_{\chi,j}$ and let $\QQ(\overline{T}_{\chi,j}, \mu_p)$ denote the number field that is given as the fixed field of $\ker(G_\QQ\ra {\rm Aut}(\overline{T}_{\chi,j}\oplus \mu_p))$. Then any prime $\ell$ whose Frobenius in ${\rm Gal}(\QQ(\overline{T}_{\chi,j}, \mu_p)/\QQ)$ is conjugate to the image of $\tau$ given as in Proposition~\ref{Prop_TauIsGood} verifies the requirements of Definition~\ref{def:Kolyvaginprimesfortwists}. In particular, $\mathcal{P}_\chi$ has infinite cardinality.
\end{remark}
\begin{lemma}
\label{lemma_H0Q(r)iszero}
For each $r\in \mathcal{R}_\chi$ and integer $j$ as above we have $H^0(\QQ(r)\QQ_\infty,T_{\chi,j})=0$.
\end{lemma}
\begin{proof}
If on the contrary $T^{(0)} \subset H^0(\QQ(r)\QQ_\infty,T_{\chi,j})$ were a rank-one $G_\QQ$-stable $\mathcal{O}$-subquotient of $T_{\chi,j}$, then $G_\QQ$ would act on $T^{(0)}$ via $\chi_{\cyc}^s\theta\nu$ where $s\in \ZZ$, $\theta$ is a character of $p$-power conductor and order, and $\nu$ is a character of conductor dividing $r \in \mathcal{R}_\chi$. Since $N_\chi>1$ is prime to $Npr$ by choice and ${\rm Sym}^2R_f^*$ is unramified outside $Np$, a subquotient $T^{(0)}$ with these properties could not exist.
\end{proof}
\begin{theorem}
\label{thm_horizontalESforSym2}
 Let $\chi$ be an even Dirichlet character that satisfies \ref{item_Psi_1} and $(\Psi_2^\prime)$ and let $j\in [k+2,2k+2]$ be an arbitrary integer. 
Then for each $r \in \mathcal{N}_\chi$, there exist  two cohomology classes 
$$d_r^{\alpha,\alpha}, d_r^{\alpha,-\alpha} \in H^1(\QQ(r),V_{\chi,j})$$
with the following properties:
\item[i)] $d_r^{\alpha,\alpha}\,,\,d_r^{\alpha,-\alpha} \in H^1_{\mathcal{F}_{\textup{can}}}(\QQ(r),V_{\chi,j})\,.$
\item[ii)] There exists a constant $D$ (that does not depend on $r$) such that 
$$Dd_r^{\alpha,\alpha}\,,\, Dd_r^{\alpha,-\alpha} \in H^1(\QQ(r),T_{\chi,j})\,.$$
\item[iii)]  For $ r\ell \in \mathcal{P}_\psi$ and $\mu\in \{\alpha,-\alpha\}$ we have 
$$\textup{cor}_{\QQ(r\ell)/\QQ(r)}\left(d_{r\ell}^{\alpha,\mu}\right)=P_\ell(\ell^{-j}\textup{Fr}_\ell^{-1})\cdot d_r^{\alpha,\mu}$$
where $P_\ell(X)$ is the Euler polynomial for $L(\textup{Sym}^2f\otimes\omega^{j}\chi,s)$ at $\ell$ and $\textup{Fr}_\ell$ is the arithmetic Frobenius.  
\item[iv)] The classes $d_1^{\alpha,\alpha}, d_1^{\alpha,-\alpha} \in H^1_{\mathcal{F}_{\textup{can}}}(\QQ,V_{\chi,j})$ are linearly independent.
\end{theorem}

\begin{proof}
This is essentially Theorem 8.1.4 of \cite{LZ1} {(which we combine with ideas from \cite{LZ2})} and we shall only explain why the line of reasoning in loc. cit. is sufficient to validate our theorem.

We first construct classes $c^{\alpha,\pm\alpha}_r\in H^1(\QQ(r),W_f^*\otimes W_f^*(1-j)\otimes\omega^j\chi)$ as in the proof of Theorem 8.1.4 \cite{LZ1} (the twisting with the appropriate characters may be carried out as in Definition~\ref{define:twistedBF} above). We shall construct $d_r^{\alpha,\pm\alpha}$ using $c_r^{\alpha,\pm\alpha}$. To avert any potential confusion, we remark that in place of the twist $1-j$ we consider here, Loeffler and Zerbes in \cite{LZ1} write $-j$. 

Notice that although the Assumption 3.5.6 of op. cit. does not hold in our case of interest, we still have 
\begin{equation}
\label{eqn:notorsionovertamelevelsW} H^0(\QQ(r)\QQ_\infty,W_f^*\otimes W_f^*(1-j)\otimes\omega^j\chi)=0
\end{equation}
thanks to our running hypothesis on $\chi$. Indeed, as we have observed as part of Remark~\ref{rem:alternativeproofforXpsij}, the Dirichlet character $\eta=\omega^{k+2}\epsilon_f\chi$ (that we have defined in the proof of Proposition~\ref{prop:prop:vanishingofX1}) does not factor through $\Gamma$. Since the conductor of $\epsilon_f\chi$ is prime $r$ (and non-trivial), it follows that $H^0(\QQ(r)\QQ_\infty,X_{\psi,j})=0$. Lemma~\ref{Lemma_torsionfree} shows that $H^0(\QQ(r)\QQ_\infty,T_{\chi,j})=0$ as well. These two vanishing results conclude the proof of (\ref{eqn:notorsionovertamelevelsW}). 

Thanks to (\ref{eqn:notorsionovertamelevelsW}), the proof of \cite[Theorem 3.5.9]{LZ1} goes through verbatim and allows one to obtain the interpolated Beilinson--Flach elements along a Coleman family.  The desired classes $c^{\alpha,\pm\alpha}_r$ are obtained on specializing these interpolated classes and modifying them slightly (as in the proof of 8.1.4(iii), that in turn relies on the argument in \cite[\S7.3]{LLZ1}) in order to ensure that they verify the correct Euler system distribution relation). We remark that we work over the fields $\QQ(r)$ (resp., $\QQ_\infty$) here instead of the full cyclotomic fields $\QQ(\mu_r)$ (resp., $\QQ(\mu_{p^\infty})$) as Loeffler and Zerbes in loc. cit. does. This is sufficient for our purposes. 

The classes $c_r^{\alpha,\pm\alpha}$  verify the conclusion of \cite[Theorem~8.1.4(i)]{LZ1}, for the same reason that these classes extend in the cyclotomic direction and therefore Proposition 2.4.4 in op. cit. applies. In other words, we infer that
$$c_r^{\alpha,\pm\alpha} \in H^1_{\mathcal{F}_{\textup{can}}}(\QQ(r),W_f^*\otimes W_f^*(1-j)\otimes\omega^j\chi)\,.$$

We now explain how to define $d_r^{\alpha,\pm\alpha}$ using $c_r^{\alpha,\pm\alpha}$. We follow   \cite[proof of Theorem 5.3.3]{LZ2}. For each prime $\ell \in \mathcal{P}_\chi$ such that $r/\ell \in \mathcal{N}_\chi$, we let $\varphi_\ell \in \Delta_{r}$ denote the unique class that maps to the pair $(\sigma_\ell,1)$ under the canonical isomorphism $\Delta_{r} \cong \Delta_{r/\ell}\times\Delta_\ell$. As in op. cit., notice that $\varphi_\ell$ is congruent to $1$ modulo the radical of the ring $\cO[\Delta_{r}]$ and hence 
$$1-\ell^{k+1-j}\epsilon_f{\chi^{-1}}(\ell)\textup{Fr}_\ell^{-1} \in \cO[\Delta_{r}]^\times\,.$$
We now define $d_r^{\alpha,\pm\alpha}$  to be the image of 
$$\prod_{\ell\mid r}\left(1-\ell^{k+1-j}\epsilon_f{\chi^{-1}}(\ell)\textup{Fr}_\ell^{-1}\right)^{-1}c_r^{\alpha,\pm\alpha} \in  H^1_{\mathcal{F}_{\textup{can}}}(\QQ(r),W_f^*\otimes W_f^*(1-j)\otimes\omega^j\chi)$$ under the identification of Corollary~\ref{cor:symmetricpowercohomequalsym2cohom}. These cohomology classes verify {(i)} by definition. 

In order to check the validity of {(ii)}, we note that Proposition~2.4.7 of \cite{LZ1} applies thanks to (\ref{eqn:notorsionovertamelevelsW}) and as in the proof of Theorem 8.1.4(ii) in op.cit., it yields the desired integrality result.

We now prove (iii). Let $Q_\ell(X)$ denote the Euler polynomial for $L(f\otimes f\otimes\omega^j\chi)$.  For $r\ell \in \mathcal{N}_\chi$, the classes $c_r^{\alpha,\pm\alpha}$ enjoy the distribution property
$$\textup{cor}_{\QQ(r\ell)/\QQ(r)}\left(c_{r\ell}^{\alpha,\pm\alpha}\right)=Q_\ell(\ell^{-j}\textup{Fr}_\ell^{-1})\cdot c_r^{\alpha,\pm\alpha}$$
Since we have
$$Q_\ell(\ell^{-j}\textup{Fr}_\ell^{-1})=(1-\ell^{k+1-j}{\chi^{-1}}(\ell)\textup{Fr}_\ell^{-1})\cdot P_\ell(\ell^{-j}\textup{Fr}_\ell^{-1})$$
 thanks to the decomposition {(\ref{eqn_decompose_duals})}, the proof of (iii) follows by our definition of the classes $d_r^{\alpha,\pm\alpha}$.

We remark that $d_1^{\alpha,\pm\alpha}=c_1^{\alpha,\pm\alpha}$ by definition and (iv) is equivalent to the assertion of Corollary~\ref{cor_linearlyindependenttwistedBF} below.
\end{proof}
\begin{theorem}
\label{THM_weakESargument} Fix $r \in \mathcal{N}_\chi$ and let $\nu$ be a Dirichlet character of conductor $r$. Set $\psi=\chi\nu$. Suppose that $f$ and $\chi$ verify the hypotheses \ref{item_Im}, \ref{item_Psi_1} and \ref{item_Psi_2}. Then,
\label{thm:canselmergroupshaverank2}
$$\dim_{E} H^1_{\mathcal{F}_{\textup{can}}}(\QQ,V_{\psi,j})=2\,,$$ 
$$H^1_{\mathcal{F}_{\textup{can}}^*}(\QQ,V^*_{\psi,j}(1))=0\,.$$
\end{theorem}
\begin{proof}
We start with the observation that $H^0(\QQ_p,T_{\psi,j})=0$
due to weight considerations. Notice further that $\dim_{E} V_{\psi,j}^-=2$. It therefore follows from \cite[Theorem 5.2.15]{mr02} that
$$\dim_{E} H^1_{\mathcal{F}_{\textup{can}}}(\QQ,V_{\psi,j})- \dim_{E} H^1_{\mathcal{F}_{\textup{can}}^*}(\QQ,V^*_{\psi,j}(1))=2\,.$$
As a result, the two assertions in the statement of our theorem are in fact equivalent and the latter follows from Theorem~\ref{thm_horizontalESforSym2} (the existence of a non-trivial \emph{horizontal Euler system}) and \cite[Theorem 2.2.3]{rubin00} (whose assumptions are modified via \cite[\S9.1]{rubin00}, by replacing the condition (ii) in the statement of  \cite[Theorem 2.2.3]{rubin00} with (ii)$^\prime$ in \S9.1 of loc.cit. so as to cover our case).
\end{proof}
\begin{corollary}
\label{cor:basecaseSelmer1}
In the setting of Theorem~\ref{THM_weakESargument}, the $\cO$-module $H^1_{\mathcal{F}_{\textup{can}}}(\QQ,T_{\psi,j})$ is free of rank $2$. 
\end{corollary}
\begin{proof}
After Theorem~\ref{thm:canselmergroupshaverank2}, we only need to prove that $H^1_{\mathcal{F}_{\textup{can}}}(\QQ,T_{\psi,j})$ is torsion-free. This follows from the fact that 
$H^0(\QQ,\overline{T}_{\psi,j})=0$ under our running hypotheses.
\end{proof}

\begin{corollary}
\label{cor:equivariantSelmersize} In the setting of Theorem~\ref{THM_weakESargument}, we have
$$\textup{rank}_{\cO}\, H^1_{\mathcal{F}_{\textup{can}}}(\QQ(r),{T}_j)=2|\Delta_r|\,,$$ 
for each $r \in \mathcal{N}_\chi$. Furthermore, the module $H^1_{\mathcal{F}_{\textup{can}}^*}(\QQ(r),{T}^\vee_j(1))$ has finite cardinality.
\end{corollary}
\begin{proof}
For each character $\nu$ of $\Delta_r$, we infer from Theorem~\ref{thm:canselmergroupshaverank2} (applied with the character $\psi=\chi\nu$) that $\dim H^1_{\mathcal{F}_{\textup{can}}}(\QQ,W_{\chi\nu,j})=2|\Delta_r|$ and the first assertion follows by Shapiro's Lemma. 
The second assertion is an immediate consequence of the first and global duality.
\end{proof}
\subsection{Structure of Iwasawa theoretic Selmer groups} \label{subsec_Iwasawa_Theory}
{We shall rely on results  in Section~\ref{subsubsec_sym_square}  to prove our main result (Theorem~\ref{thm_mainSelmerstructure}) of Section~\ref{sec_selmerstructure}, where we describe the structures of certain Iwasawa theoretic Selmer groups.}

 We recall that $T:=T_{\chi,0}=\Sym^2R_f^*(1+\chi)$. Set $\TT_j:=T_j\otimes \LL_{\cO}(\Gamma)^\iota$ and recall that we have $H^1_{\mathcal{F}_{\textup{can}}}(\QQ,\TT_j)=H^1(\QQ,\TT_j)$ by \cite[Lemma 5.3.1]{mr02}. 
When $j=0$, we shall drop $j$ from the notation and simply write $\TT$ in place of $\TT_0=T\otimes\LL_{\cO}(\Gamma)^\iota$.
\begin{corollary}
\label{cor:equivariantSelmersize1}
In the setting of Theorem~\ref{THM_weakESargument},  the $\LL_{\cO}(\Gamma)$-module $H^1_{\mathcal{F}_{\textup{can}}^*}(\QQ(r),{\TT}_j^\vee(1))$ is cotorsion and the $\LL_{\cO}(\Gamma)$-module $H^1(\QQ(r),{\TT}_j)$ has rank $2|\Delta_r|$ for each $r \in \mathcal{N}_\chi$.
\end{corollary}
\begin{proof}
The first assertion follows from the control theorem \cite[Lemma 3.5.3]{mr02}
$$\left(H^1_{\mathcal{F}_{\textup{can}}^*}(\QQ(r),{\TT}_j^\vee(1))^{\vee}\right)_\Gamma\stackrel{\sim}{\lra} H^1_{\mathcal{F}_{\textup{can}}^*}(\QQ(r),{T}_j^\vee(1))^\vee$$
and Corollary~\ref{cor:equivariantSelmersize}. 

{Let us write $\chi(\TT_j,r)$ for the global Euler--Poincar\'e characteristic for the canonical Selmer structure $\cF_{\rm can}$  over the totally real field $\QQ(r)$. Then,
\begin{align*}
\chi(\TT_j,r)&={\rm rank}_{\LL_{\cO}(\Gamma)}\left( H^1_{\mathcal{F}_{\textup{can}}}(\QQ,\TT_j)\right)-{\rm rank}_{\LL_{\cO}(\Gamma)}\left( H^1_{\mathcal{F}_{\textup{can}}^*}(\QQ(r),{\TT}_j^\vee(1))^{\vee}\right)\\
&={\rm rank}_{\LL_{\cO}(\Gamma)}\left( H^1(\QQ,\TT_j)\right)
\end{align*}
where the first equality follows from the definition of the Euler--Poincar\'e characteristic and the vanishing of $H^0(\QQ(r),\overline{T}_j)$, whereas the second from the first and \cite[Lemma 5.3.1]{mr02}. On the other hand, it follows from the global Euler--Poincar\'e characteristic computation in \cite[Theorem 7.8.6]{nekovar06} that 
$$\chi(\TT_j,r)=[\QQ(r):\QQ]\,{\rm rank}_{\cO}T_j^-,$$
where we rely on the fact that the number field $\QQ(r)$ is totally real. Since ${\rm rank}_{\cO}T_j^-={\rm rank}_{\cO}T^-=2$, the proof of our second assertion follows.}
\end{proof}
\begin{corollary}
\label{cor_v2_Lambda_rank}
In the setting of Theorem~\ref{THM_weakESargument}, the $\LL_{\cO}(\Gamma)$-module $H^1(\QQ(r),{\TT})$ has rank $2|\Delta_r|$  for every $r \in \mathcal{N}_\chi$.
\end{corollary}
\begin{proof}
This follows from Corollary~\ref{cor:equivariantSelmersize1} on noticing that 
\begin{align*}
H^1(\QQ(r),{\TT}) \stackrel{\sim}{\lra} H^1(\QQ(r),\TT_j)\otimes \chi_\cyc^{j}\omega^{-j}.
\end{align*}
\end{proof}

\begin{lemma}
\label{Lemma_torsionfree}
The $\LL_{\cO}(\Gamma)$-module $H^1(\QQ(r),\TT)$ is {torsion-free}  for every $r \in \mathcal{N}_\chi$.
\end{lemma}
\begin{proof}
This is immediate from Lemma~\ref{lemma_H0Q(r)iszero}.
\end{proof}
{We recall that $\LL_r=\cO[[\Delta_r\times\Gamma]]$, where $\Delta_r=\Gal(\QQ(r)/\QQ)$.}
\begin{theorem}
\label{thm_mainSelmerstructure}
In the setting of Theorem~\ref{THM_weakESargument}, the $\LL_r$-module $H^1(\QQ(r),\TT)$ is free of rank $2$  for every $r \in \mathcal{N}_\chi$.
\end{theorem}
\begin{proof}
We have a natural injection
$$H^1(\QQ(r),\TT)/(\mathcal{A}_r,\omega^{-j}\chi_\cyc^{j}(\gamma)\gamma-1) \hookrightarrow H^1_{\mathcal{F}_{\textup{can}}}(\QQ,T_j)\,,$$
where $\mathcal{A}_r \subset \cO[\Delta_r]$ is the augmentation ideal. It follows from Nakayama's lemma and Corollary~\ref{cor:basecaseSelmer1} that the $\LL_r$-module $H^1(\QQ(r),\TT)$ may be generated by at most $2$ elements. Let $\{c_1,c_2\}$ be any set of such generators. To prove our theorem, it suffices to check that $c_1$ and $c_2$ do not admit a non-trivial $\LL_r$-linear relation. 

Assume the contrary and suppose that there is a non-trivial relation 
\begin{equation}\label{eqn:linearrelationalphac}
\alpha_1c_1 + \alpha_2c_2=0,\,\,\, \alpha_1,\alpha_2 \in \LL_r\,.
\end{equation}
Write $\mathscr{B}=\{\delta c_j: \delta \in \Delta_r, j=1,2\}$. Notice that $\mathscr{B}$ generates $H^1(\QQ(r),\TT)$ as a $\LL_{\cO}(\Gamma)$-module and 
\begin{equation}
\label{eqn:dimensioncount}
|\mathscr{B}|=2|\Delta_r|=\dim_{\textup{Frac}(\LL_{\cO}(\Gamma))} \left(H^1(\QQ(r),\TT)\otimes_{\LL_{\cO}(\Gamma)}\textup{Frac}(\LL_{\cO}(\Gamma))\right)\,,
\end{equation}
where the final equality is Corollary~\ref{cor_v2_Lambda_rank}. The equation (\ref{eqn:linearrelationalphac}) may be rewritten as 
{
\begin{equation}
\label{eqn:lamdalinearrelationfordeltac}
\sum_{\delta} (\lambda_{\delta,1} \cdot \delta) c_1 + \sum_{\delta} (\lambda_{\delta,2} \cdot \delta) c_2 = 0
\end{equation}}
with $\lambda_{\delta,j} \in \LL_{\cO}(\Gamma)$. Since $H^1(\QQ(r),\TT)$ is $\LL_{\cO}(\Gamma)$-torsion-free by Lemma~\ref{Lemma_torsionfree}, we have a canonical injection 
$$\iota: H^1(\QQ(r),\TT)\hookrightarrow H^1(\QQ(r),\TT) \otimes_{\LL_{\cO}(\Gamma)}\textup{Frac}(\LL_{\cO}(\Gamma))$$ 
and furthermore, as a $\textup{Frac}(\LL_{\cO}(\Gamma))$-vector space, $H^1(\QQ(r),\TT) \otimes_{\LL_{\cO}(\Gamma)}\textup{Frac}(\LL_{\cO}(\Gamma))$ is generated by the set $\iota(\mathscr{B})$\,. It follows from the relation (\ref{eqn:lamdalinearrelationfordeltac}) that 
$$\dim_{\textup{Frac}(\LL_{\cO}(\Gamma))} \left(H^1(\QQ(r),\TT)\otimes_{\LL_{\cO}(\Gamma)}\textup{Frac}(\LL_{\cO}(\Gamma))\right) \leq |\iota(\mathscr{B})|-1=2|\Delta_r|-1,$$
which contradicts (\ref{eqn:dimensioncount}) and concludes our proof.
\end{proof}


\section{Signed Iwasawa theory} \label{sec:signediwasawatheory}
{In this section, we shall generalize the construction of plus and minus Coleman maps of Kobayashi \cite{kobayashi03} to the representation $W_f^*\otimes W_f^*(1+\chi)$. The kernels of these maps are then served to define local Selmer conditions at $p$, which in turn are used to define the so-called doubly  signed Selmer groups. The local theory we develop here is also used to factorize  the Beilinson--Flach elements in Section~\ref{subsec:plocalBF} into bounded elements and to define bounded $p$-adic $L$-functions, generalizing the work of Pollack  \cite{pollack03} on one single modular form. Conjecture~\ref{conj:signedmainconjecture} relates these $p$-adic $L$-functions to the doubly signed Selmer groups, which is a form of the Iwasawa main conjecture in the spirit of Kobayashi's work on supersingular elliptic curves in \cite{kobayashi03}. We prove one inclusion of the conjecture in Theorem~\ref{thm_signedmainconjecture} using the bounded Beilinson--Flach elements we obtain in \S\ref{S:signedBF}.}

 We recall here that Pollack's plus and minus logarithms are defined as follows. Let $\chi_\cyc$ denote the cyclotomic character on $\Gamma$ and recall that $\gamma$ is a fixed topological generator of $\Gamma_1$.  For an integer $r\ge1$, we define
\begin{align*}
\log^+_{p,r}&=\prod_{j=0}^{r-1}\frac{1}{p}\prod_{n=1}^\infty\frac{\Phi_{p^{2n}}(\chi_\cyc^{-j}(\gamma)\gamma)}{p},\\
\log^-_{p,r}&=\prod_{j=0}^{r-1}\frac{1}{p}\prod_{n=1}^\infty\frac{\Phi_{p^{2n-1}}(\chi_\cyc^{-j}(\gamma)\gamma)}{p}.
\end{align*}
Recall from \cite{pollack03} that $\log^\pm_{p,r}\in \cH_{E,r/2}(\Gamma)$, where $\cH_{E,r/2}(\Gamma)$ denotes the set of $E$-valued tempered distributions of order $r/2$ on $\Gamma$.  (in fact, $\log_{p,r}^\pm\sim O(\log_p^{r/2})$). We shall also write 
\[
\log_{p,r}=\prod_{j=0}^{r-1}\log_{p}(\chi_\cyc^{-j}(\gamma)\gamma)\in \cH_{E,r}(\Gamma).
\]

If $n$ is an integer, we write 
\begin{equation}\label{eq:twist}
\Tw_n:\cH_{E,r}(\Gamma)\rightarrow \cH_{E,r}(\Gamma)
\end{equation}
for the $E$-linear map  induced by $\sigma\mapsto\chi_\cyc^n(\sigma)\sigma $ for all $\sigma\in \Gamma$. {Observe that
\begin{equation}\label{eq:twist-divide}
\textup{Tw}_{-n}\log_{p,r}^{?}={\log_{p,r+n}^{?}}/{\log_{p,n}^{?}}\,,\,\,\,\, ?=\emptyset,\pm.
\end{equation}
As a shorthand, we set
\begin{equation} \label{eqn:log(1)}
\log_{p,r}^{?, (1)}\defeq\textup{Tw}_{-1}\log_{p,r}^{?}={\log_{p,r+1}^{?}}/{\log_{p,1}^{?}}\,,\,\,\,\, ?=\emptyset,\pm.
\end{equation}
It follows from \eqref{eq:twist-divide} that
\begin{equation}\label{eq:twist-divide2}
\textup{Tw}_{-n}\log_{p,r}^{?,(1)}={\log_{p,r+n}^{?,(1)}}/{\log_{p,n}^{?,(1)}}\,,\,\,\,\, ?=\emptyset,\pm.
\end{equation}}

\subsection{Local theory}
{We study  a decomposition of the local representation $R_f\otimes R_f|_{G_{\Qp}}$, which relies crucially on our assumption that $a_p(f)=0$. This decomposition allows us to relate the local representation to the setting studied in \cite{lei09}. This relation will be exploited to define the signed Coleman maps in \S\ref{sec:signedcoleman}. }

Let $D$ be the Dieudonn\'e module of $W_f|_{G_{\Qp}}$. Recall that 
\[
\dim_L\Fil^iD=\begin{cases}
2& i \le 0,\\
1& 1\le i\le k+1,\\
0&i\ge k+2.
\end{cases}
\]
Recall that we have assumed the \emph{Fontaine--Laffaille} condition  {$p > k+1$} holds. 
On combing the Wach module basis in \cite[\S3]{bergerlizhu04} with the construction of integral Dieudonn\'e module in \cite[\S IV]{berger04}, there is an $\cO$-lattice $\Dcris(R_f)$ inside $D$, which is generated by $\omega,p^{-k-1}\vp(\omega)$, where $\omega$ is an $\cO$-basis of $\Fil^1\Dcris(R_f)$, which we fix from now on. (See also \cite[Lemma~3.1]{LLZ-CJM} for a similar basis.)

\begin{lemma}\label{lem:localsplit}
The filtered $\vp$-module $\Sym^2D$ is decomposable into $D_1\oplus D_2$, where $D_i$ is of dimension $i$ for both $i=1,2$.
\end{lemma}
\begin{proof}The filtration of $\Sym^2D$ is given by:
\[
\Fil^i\Sym^2D=\begin{cases}
\Sym^2D& i\le0,\\
\langle \omega\otimes \omega,\omega\otimes \vp(\omega)+\vp(\omega)\otimes\omega\rangle& 1\le i\le k+1,\\
\langle \omega\otimes\omega\rangle&{k+2}\le i\le 2k+2,\\
0&i\ge2k+3.
\end{cases}
\]

We define
\begin{align*}
D_1&=\langle \omega\otimes\vp(\omega)+\vp(\omega)\otimes \omega \rangle;\\
D_2&=\langle \omega\otimes \omega,\vp(\omega)\otimes\vp(\omega)\rangle.
\end{align*}  
The fact that $\vp^2(\omega)=\alpha^2\omega$ implies that both $D_1$ and $D_2$ are stable under $\vp$. Furthermore, they both respect the filtration of $\Sym^2D$. Hence, they are both filtered $\vp$-modules as required.
\end{proof}

\begin{corollary}\label{cor:sym2splits}
The $G_{\Qp}$-representation $\Sym^2W_f|_{G_{\Qp}}$ splits into $W_1\oplus W_2$, where $W_i$ is of dimension $i$.
\end{corollary}
\begin{proof}
This follows from the correspondence between $G_{\Qp}$-representations and filtered $\vp$-modules of Fontaine \cite[Théorème~3.6.5]{fontaine79}. See also \cite[\S 2.2]{bernadette_1998} where a similar decomposition when $W_f$ comes from the Tate module of a $p$-supersingular elliptic curves was studied.
\end{proof}

\begin{remark}We note that this decomposition was exploited in \cite{lei12} in the CM case. In fact, this decomposition holds as $G_\QQ$-representations  (not just $G_{\Qp}$-representations) when $f$ is of CM type. \end{remark}

For $i=1,2$, we define the lattice $R_i=W_i\cap\Sym^2R_f$ inside $W_i$. In particular, we have the decomposition of $G_{\Qp}$-representations
\[
\Sym^2R_f|_{G_{\Qp}}=R_1\oplus R_2.
\]
{We also have the integral Dieudonn\'e modules $\Dcris(R_1)$ and $\Dcris(R_2)$. Furthermore,  $\Dcris(R_1)$ 
 is  generated by $p^{-k-1}(\omega\otimes\vp(\omega)+\vp(\omega)\otimes \omega)$, whereas $\Dcris(R_2)$ is generated by $\omega\otimes\omega$ and $p^{-2k-2}\vp(\omega)\otimes\vp(\omega)$.}

Let $\chi$ be a fixed Dirichlet character as in Section ~\ref{subsec:plocalBF}. In particular,  $\chi$ is unramified at $p$. We 
{have the following isomorphisms of filtered modules
\[
\Dcris(W_f\otimes W_f(\chi^{-1}))\stackrel{\otimes v_{\chi}}{\longrightarrow }\Dcris(W_f\otimes W_f), \quad \Dcris(R_f\otimes R_f(\chi^{-1}))\stackrel{\otimes v_{\chi}}{\longrightarrow }\Dcris(R_f\otimes R_f),
\]
where $\{v_\chi\}$ is an $\cO$-basis of $\Dcris(\cO(\chi))$.
}
 Consequently,  the local representation $\Sym^2W_f(\chi^{-1})|_{G_{\Qp}}$ splits into  $W_{1,\chi}\oplus W_{2,\chi}$ as in Corollary~\ref{cor:sym2splits}. Similarly, we have the integral counterpart 
\[
\Sym^2R_f(\chi^{-1})|_{G_{\Qp}}=R_{1,\chi}\oplus R_{2,\chi}.
\]

We can see from the proof of Lemma~\ref{lem:localsplit} that the Hodge-Tate weight of $W_{1,\chi}$ is $-1-k$, whereas those of $W_{2,\chi}$ are $0$ and $-2-2k$. Furthermore, 
the filtration on $\Dcris(W_{2,\chi})$ is given by
\[
\Fil^i\Dcris(W_{2,\chi})=
\begin{cases}
\Dcris(W_{2,\chi})& i\le 0,\\
\langle\omega\otimes\omega\otimes v_{\chi^{-1}}\rangle&1\le i\le 2k+2,\\
0& i\ge 2k+3.
\end{cases}
\]

By duality, we have the decompositions of $G_{\Qp}$-representations
\begin{equation}\label{eq:localdecomp}
W_f^*\otimes W_f^*(\chi)=W_{0,\chi}^*\oplus W_{1,\chi}^*\oplus W_{2,\chi}^*,\quad R_f^*\otimes R_f^*(\chi)=R_{0,\chi}^*\oplus R_{1,\chi}^*\oplus R_{2,\chi}^*,
\end{equation}
where $W_{0,\chi}^*=\bigwedge^2W_f^*(\chi)$ and $R_{0,\chi}^*=\bigwedge^2R_f^*(\chi)$. Furthermore, 
 $R_{1,\chi}^*$ and $R_{0,\chi}^*$ are rank-one representations of the form $\cO(\psi_i+k+1)$, for some unramified  characters $\psi_i$  on $G_{\Qp}$ sending $p$ to $\pm \epsilon_f\chi(p)$ respectively.
In particular, they both have a single Hodge-Tate weight, namely $k+1$. For the representation $W_{2,\chi}^*$, we have the filtration
\[
\Fil^i\Dcris(W_{2,\chi}^*)=
\begin{cases}
\Dcris(W_{2,\chi}^*)& i\le -2k-2,\\
\langle\omega'\otimes\omega'\otimes v_\chi\rangle&-2k-1\le i\le 0,\\
0& i\ge 1,
\end{cases}
\]
for some basis $\omega'$ that generates the $\cO$-module $\Fil^0\Dcris(R_f^*)$. Note that $\{\omega',\vp(\omega')\}$ is an $\cO$-basis of $\Dcris(R_f^*)$, which implies that $\Dcris(R_{2,\chi}^*)$ is generated by $\omega'\otimes\omega'\otimes v_\chi$ and $\vp(\omega')\otimes\vp(\omega')\otimes v_\chi$.

Let $F/\Qp$ be a finite unramified extension. Given a crystalline $E$-linear representation $W$ of $G_{F}$ whose Hodge-Tate weights are all non-negative, we write
\[
\cL_{W,F}:\HIw(F,W)\rightarrow F\otimes\cH_{E,r}(\Gamma)\otimes\Dcris(W)
\]
for the Perrin-Riou regulator map (c.f., \cite[\S3.1]{leiloefflerzerbes11} and \cite[Appendix B]{LZ-IJNT}). Here, $r$ denotes the largest slope of $\vp$ on $\Dcris(W)$. {We now study the images of the Perrin-Riou maps for  the direct summands in \eqref{eq:localdecomp}.}
\begin{lemma}\label{lem:divisible1}
Let $W=W_{1,\chi}^*$ or $W_{0,\chi}^*$ and $F/\Qp$ a finite unramified extension. For all $z\in \HIw(F,W)$, we have
$$
\cL_{W,F}(z)\in \log_{p,k+1}F\otimes \Lambda_{\cO}(\Gamma)\otimes \Dcris({W}).
$$
Let  $R=R_{1,\chi}^*$ or $ R_{0,\chi}^*$. If  $z\in \HIw(F,R)$, then
$$
\cL_{W,F}(z)\in \log_{p,k+1}\cO_{F}\otimes\Lambda_{\cO}(\Gamma)\otimes \Dcris(R),
$$
where $\Dcris(R)$ is the $\cO$-lattice inside $\Dcris(W)$ as defined in \cite[\S IV]{berger04}.
\end{lemma}
\begin{proof}
Since $k\ge0$, we have the identification 
\[
\HIw(F,R)=\NN_{F}(R)^{\psi=1}
\]
where $\NN_{F}(R)=\cO_F\otimes_{\Zp}\NN_{\Qp}(R)$ denotes the Wach module of $R$ over $F$ for $R=R_{1,\chi}^*$ or $ R_{0,\chi}^*$ (see \cite[\S2.7]{LZ-IJNT}). Similarly, we may identify $\HIw(F,W)$ with $\NN_F(W)^{\psi=1}$.

We recall that the construction of $\cL_{W,F}$ can be realized as
\[
1-\vp:\NN_F(W)^{\psi=1}\rightarrow \left(\vp^*\NN_F(W)\right)^{\psi=0}.
\]
The right-hand side is contained inside $F\otimes(\Brig)^{\psi=0}\otimes \Dcris(W)$, which in turn is isomorphic to  $F\otimes\cup_r\cH_{\Qp,r}(\Gamma)\otimes \Dcris(W)$ via the Mellin {transforms}. 

The image of $\NN_F(R)^{\psi=1}$ under $1-\vp$ lies inside $\left(\vp^*\NN_{\Qp}(R)\right)^{\psi=0}$, which is a free $\cO_F\otimes\Lambda_{\cO}(\Gamma)$-module  generated by $(1+\pi)\vp(n)$ for some $\cO\otimes\AQp$-basis $n$ of $\NN_{\Qp}(R)$ by \cite[Theorem~3.5]{leiloefflerzerbes10}.  By \cite[proof of Proposition V.2.3]{berger04}, $v: =n\mod \pi$  is a basis of $\Dcris(R)$. Furthermore, \cite[Proposition~III.2.1]{berger04} tells us that $n$ and $(t/\pi)^{k+1}v$ agree up to a unit in $\Brig$. But since both $v$ and $n$ are defined integrally, the aforementioned unit is in fact defined over $\AQp$. Consequently, if $x\in \NN_F(R)^{\psi=1}$, we have $(1-\vp)(x)\in \vp(t/\pi)^{r}\cO_F\otimes(\AQp)^{\psi=0}v$. On taking Mellin transform, this lies inside $\log_{p,k+1}\cO_F\otimes\Lambda_\cO(\Gamma)v$ by \cite[Theorem~5.4]{leiloefflerzerbes10} and \cite[Theorem~2.1]{LLZ-CJM}. 
\end{proof}

\begin{lemma}\label{lem:divisible2}
Let $F/\Qp$ be a finite unramified extension. {There exist $\LL_\cO(\Gamma)$-homomorphisms
\[
\cL_{\pm,F}:\HIw(F,W_{2,\chi}^*)\rightarrow  \log^\pm_{p,2k+2}F\otimes \Lambda_{\cO}(\Gamma)
\]
 such that for all $z\in \HIw(F,W_{2,\chi}^*)$,}
\[
\cL_{W_{2,\chi}^*,F}(z)=\cL_{+,F}(z)\omega'\otimes\omega'\otimes v_\chi+\cL_{-,F}(z)\vp(\omega')\otimes\vp(\omega')\otimes v_\chi.
\] 
 Furthermore, if $z\in\HIw(F,R_{2,\chi}^*)$, then $\cL_{\pm,F}(z)\in \log^\pm_{p,2k+2}\cO_{F}\otimes\Lambda_{\cO}(\Gamma)$. 
\end{lemma}
\begin{proof}
Recall that $\omega'\otimes\omega'\otimes v_\chi$ is a basis of $\Fil^0\Dcris(W_{2,\chi}^*)$ and its image under $\vp$ is, up to a unit, equal to $\vp(\omega')\otimes\vp(\omega')\otimes v_\chi$. These two elements of $\Dcris(W_{2,\chi}^*)$ give rise to two maps 
$$\cL_{\pm,F}:\HIw(F,W_{2,\chi}^*)\lra F\otimes\cH_{E,k+1}(\Gamma) $$
as given by \cite[(18)]{leiloefflerzerbes10} (see also the construction in \cite{lei09}, \S3.2). These two maps then decompose  $\cL_{W_{2,\chi}^*,F}$ in a manner as explained in \cite[\S5A]{leiloefflerzerbes11} (note that our maps here differ from those given in op. cit. by units). This proves the first assertion of the lemma. 

 For the integrality statement, we may argue as in Lemma~\ref{lem:divisible1} on using the Wach module basis of $\NN_{\Qp}(R_{2,\chi}^*)$ as given in \cite[\S A]{berger04} and \cite[\S5.2]{leiloefflerzerbes10}.
\end{proof}

Let $\omega'$ be a fixed basis of $\Fil^0\Dcris(R_f^*)$ as above. The eigenvalues  of $\vp$ on $\Dcris(W_f^*)$ are $\pm\frac{1}{\alpha}$ and we have the eigenvectors
\begin{equation}
v_\lambda=\vp(\omega')+\frac{1}{\lambda}\omega',\label{eq:evecs}
\end{equation}
for $\lambda=\pm{\alpha}$. Then, the four { $\vp$-eigenvectors $v_\lambda\otimes v_\mu\otimes v_\chi$ decompose $\Dcris(W_f^*\otimes W_f^*(\chi))$  into a direct sum of one-dimensional subspaces. This in turn allows us to decompose the Perrin-Riou map as follows.}

\begin{defn}
\label{def:thePRmapsforthesymmetricproduct}
For any finite unramified extension $F$ of $\Qp$, we define $\cL_{\pm\alpha,\pm\alpha,F},\cL_{\circ,F}$ and $\cL_{\bullet,F}$   to be the unique $\LL_\cO(\Gamma)$-morphisms from $\HIw(F,W_f^*\otimes W_f^*(\chi))$ to $F\otimes\cH_{E,k+1} (\Gamma)$ satisfying the equation
\begin{align*}
\cL_{W_f^*\otimes W_f^*(\chi),F}(z)&=\sum_{\lambda,\mu\in\{\alpha,-\alpha\}}\cL_{\lambda,\mu,F}(z)v_\lambda\otimes v_\mu\otimes v_\chi\\
&=\cL_{\circ,F}(z)v_{\circ}+\cL_{\bullet,F}(z)v_{\bullet}+\cL_{+,F}(z)\omega'\otimes\omega'\otimes v_\chi+\cL_{-,F}(z)\vp(\omega')\otimes\vp(\omega')\otimes v_\chi,
\end{align*}
for all $z\in \HIw(F,W_f^*\otimes W_f^*(\chi))$, 
where 
\begin{align*}
v_{\circ}&=\omega'\otimes\vp(\omega')\otimes v_\chi-\vp(\omega')\otimes\omega'\otimes v_\chi\in \Dcris(R_{0,\chi}^*),\\
v_{\bullet}&=\omega'\otimes\vp(\omega')\otimes v_\chi+\vp(\omega')\otimes\omega'\otimes v_\chi\in\Dcris(R_{1,\chi}^*),
\end{align*}
 and $\cL_{\pm,F}$ are defined as in Lemma~\ref{lem:divisible2}.
\end{defn}

\begin{lemma}\label{lem:factorisation}
Let $F/\Qp$ be a finite unramified extension. For all $z\in \HIw(F,R_f^*\otimes R_f^*(\chi))$, we have
\[
\begin{pmatrix}
1&1&1&1\\
\alpha^2&\alpha^2 &-\alpha^2&-\alpha^2\\
2\alpha&-2\alpha&0&0\\
0&0&-2\alpha&2\alpha
\end{pmatrix}
\begin{pmatrix}
\cL_{\alpha,\alpha,F}(z)\\
\cL_{-\alpha,-\alpha,F}(z)\\
\cL_{\alpha,-\alpha,F}(z)\\
\cL_{-\alpha,\alpha,F}(z)
\end{pmatrix}\in \begin{pmatrix}
\log_{p,2k+2}^-\\ \log_{p,2k+2}^+\\ \log_{p,k+1}\\\log_{p,k+1}
\end{pmatrix}\\ \cO_{F}\otimes\Lambda_{\cO}(\Gamma).
\]
\end{lemma}
\begin{proof}
We have the change of basis matrix \[\begin{pmatrix}
\vp(\omega')\otimes\vp(\omega')\\
\omega'\otimes\omega'\\
\vp(\omega')\otimes\omega'+\omega'\otimes \vp(\omega')\\
\vp(\omega')\otimes\omega'-\omega'\otimes \vp(\omega')
\end{pmatrix}
=1/4\begin{pmatrix}
1&1&1&1\\
\alpha^2&\alpha^2 &-\alpha^2&-\alpha^2\\
2\alpha&-2\alpha&0&0\\
0&0&-2\alpha&2\alpha
\end{pmatrix}
\begin{pmatrix}
v_\alpha\otimes v_\alpha\\
v_{-\alpha}\otimes v_{-\alpha}\\
v_{\alpha}\otimes v_{-\alpha}\\
v_{-\alpha}\otimes v_\alpha
\end{pmatrix}.
\] Hence, our result follows from Lemmas~\ref{lem:divisible1} and~\ref{lem:divisible2}.
\end{proof}

{We finish this subsection by introducing certain projection maps on global cohomology groups based on the local maps that we have defined above. }

\begin{defn}
\label{define_projectionviaotsukifunctionals}
For any number field $K$ which is unramified at all primes above $p$ and given an element 
$$z=z_1\wedge z_2\in {\bigwedge\!}^2 H^1(K,W_f^*\otimes W_f^*(\chi)\otimes \Lambda_{\cO}(\Gamma)^\iota)$$ as well as $\lambda,\mu\in\{\pm \alpha\}$, we define 
\begin{align*}
\pr_{\lambda,\mu}(z)&=\cL_{\lambda,\mu,K}(\res_p (z_1))z_2-\cL_{\lambda,\mu,K}(\res_p (z_2))z_1\\
&\in H^1_\Iw(K,W_f^*\otimes W_f^*(\chi))\widehat{\otimes} \cH_{E,k+1}(\Gamma),
\end{align*}
where $\res_p:H^1(K,W_f^*\otimes W_f^*(\chi)\otimes \Lambda_{\cO}(\Gamma)^\iota)\rightarrow \bigoplus_{v|p}H^1(K_v,W_f^*\otimes W_f^*(\chi)\otimes \Lambda_{\cO}(\Gamma)^\iota)$  is defined by the local restriction maps and $\cL_{\lambda,\mu,K}$ is the shorthand for the sum ${\displaystyle \Sigma_{v|p}\cL_{\lambda,\mu,K_v}}$.
\end{defn}
\begin{proposition}
\label{prop:rankreductionviaotsukisfunctionals}
For $K$ and $z\in\bigwedge^2 H^1(K,R_f^*\otimes R_f^*(\chi)\otimes \Lambda_{\cO}(\Gamma)^\iota) $ as in Definition~\ref{define_projectionviaotsukifunctionals}, we have
$$\begin{pmatrix}
1&1&1&1\\
\alpha^2&\alpha^2 &-\alpha^2&-\alpha^2\\
2\alpha&-2\alpha&0&0\\
0&0&-2\alpha&2\alpha
\end{pmatrix}
\begin{pmatrix}
\pr_{\alpha,\alpha}(z)\\
\pr_{-\alpha,-\alpha}(z)\\
\pr_{\alpha,-\alpha}(z)\\
\pr_{-\alpha,\alpha}(z)
\end{pmatrix} \in \begin{pmatrix}
\log_{p,2k+2}^-\\ \log_{p,2k+2}^+\\ \log_{p,k+1}\\\log_{p,k+1}
\end{pmatrix}H^1(K,R_f^*\otimes R_f^*(\chi)\otimes  \Lambda_{\cO}(\Gamma)^\iota)\,.
$$
\end{proposition}
\begin{proof}
This follows immediately from Lemma~\ref{lem:factorisation}.
\end{proof}

\begin{remark}\label{rk:twistmaps}
We may define similar maps   on the Tate twists of $W_f^*\otimes W_f^*(\chi)$ as follows. We have the local maps 
$$\HIw(F, W_f^*\otimes W_f^*(\chi+j))\lra  \HIw(F, W_f^*\otimes W_f^*(\chi))\stackrel{\cL_{\lambda,\mu,F}}{\lra} \cH_{E,k+1}(\Gamma)\stackrel{\Tw_{j}}{\lra}\cH_{E,k+1}(\Gamma).$$
We can then define the semi-local map
\[
 \HIw(\QQ(m), W_f^*\otimes W_f^*(\chi+j))\lra \cH_{E,k+1}(\Gamma)
\]
 and the projection map
\[
{\bigwedge\!}^2 \HIw(\QQ(m), W_f^*\otimes W_f^*(\chi+j))\lra
\HIw(\QQ(m), W_f^*\otimes W_f^*(\chi+j))\widehat{\otimes}\cH_{E,k+1}(\Gamma)
\]
for every integer $j$ and $m\in \mathcal{N}_\chi$ as in Definition~\ref{define_projectionviaotsukifunctionals}. {We shall  denote the resulting maps by $\cL_{\lambda,\mu,m}^{(j)}$ and $\pr_{\lambda,\mu}^{(j)}$ respectively. When the dependence on $j$ is clear from the context, we shall drop the superscript $(j)$ from the notation.} 
\end{remark}

\subsection{Signed Coleman maps and Selmer conditions} \label{sec:signedcoleman}

Let $F$ be a finite unramified extension of $\Qp$. Recall that  we have set $T:=\textup{Sym}^2R_f^*(1)\otimes\chi$ and $V:=\textup{Sym}^2W_f^*(1)\otimes\chi$. 
As in Remark~\ref{rk:twistmaps}, we define the twisted version of the maps $\cL_{\pm,F}$ in Lemma~\ref{lem:divisible2} as well as $\cL_{\bullet,F}$ and $\cL_{\circ,F}$ in Definition~\ref{def:thePRmapsforthesymmetricproduct}. {The twisted maps  are denoted by  $\cL_{?,F}^{(1)}$.} Given any element $z\in \HIw(F, V)$, we have
 \[
\cL_{\pm,F}^{(1)}(z)\in \log_{p,2k+2}^{\pm,(1)} F\otimes \Lambda_{\cO}(\Gamma) \mbox{ and } \cL^{(1)}_{\bullet, F}(z)\in \log_{p, k+1}^{(1)} F\otimes \Lambda_{\cO}(\Gamma).
\] 
Furthermore, if $z\in \HIw(F, T)$, then
 \[
\cL_{\pm,F}^{(1)}(z)\in \log_{p,2k+2}^{\pm,(1)}\cO_F\otimes \Lambda_{\cO}(\Gamma) \mbox{ and } \cL_{\bullet, F}^{(1)}(z)\in \log_{p, k+1}^{(1)} \cO_F\otimes \Lambda_{\cO}(\Gamma)
\] 
(thanks to Lemmas~\ref{lem:divisible1} and ~\ref{lem:divisible2}). We now define the signed Coleman maps as follows.
\begin{defn} \label{defn:signedcolemanmaps}
For $\clubsuit \in \{ +,-,\bullet \}$, we let $\log_{p}^{\clubsuit}$ denote $\log_{p, 2k+2}^{+,(1)}$, $\log_{p, 2k+2}^{-,(1)}$ and $\log_{p, k+1}^{(1)}$ respectively. We define the signed Coleman maps $\col_{F}^\clubsuit$ by setting
\begin{align*}
\col_F^\clubsuit:\HIw(F,T)&\lra \cO_F\otimes \Lambda_{\cO}(\Gamma)\\
z&\longmapsto \cL_{\clubsuit,F}^{(1)}(z)/\log_{p}^\clubsuit.
\end{align*}
\end{defn} 
Let $\eta$ be a character on $\Gamma_{\rm tors}$. We may identify $e_\eta\Lambda_{\cO}(\Gamma)$ with the power series ring $\cO[[X]]$, where $X$ is given by $\gamma-1$ and $\gamma$ is a topological generator of $\Gamma_1$. The images of the plus and minus Coleman maps for $\QQ_p$ can be described as follows:
\begin{proposition}
For $\eta$ as above, $e_\eta\image(\col^\pm_{\QQ_{p}})$ is pseudo-isomorphic to $\prod_{j\in S_\eta^\pm}(X-\chi_\cyc^j(\gamma)-1)\mathbb{Z}_{p}\otimes\cO[[X]]$, where $S_{\eta}^\pm$ is some subset of $\{1,2,\ldots, 2k+2\}$.
\end{proposition}
\begin{proof}
This follows from \cite[Corollary~4.15]{leiloefflerzerbes11}. 
\end{proof}

\begin{defn}
Following \cite{kobayashi03, lei09}, we define the signed Selmer conditions
\[
H^1_{\Iw,\clubsuit}(F,T)=\ker(\col_F^\clubsuit).
\]
for $\clubsuit \in \{+,-,\bullet \}$. Further, if $j$ is any integer, we define  $H^1_{\Iw,\clubsuit}(F,T(j))$ to be the natural image of $H^1_{\Iw,{\clubsuit}}(F,T)$ under the twisting morphism $H^1_{\Iw}(F,T)\rightarrow H^1_{\Iw}(F,T(j))$. 
\end{defn}

Fix an integer $m \in \mathcal{N}_\chi$. We may combine the signed Coleman maps $\col^{\clubsuit}_{\QQ(m)_v}$ for primes $v$ of $\QQ(m)$ above $p$ to obtain
\[ \col^{\clubsuit}_{m} =\oplus_{v|p}\col^\clubsuit_{\QQ(m)_v}\]
on
$ H^{1}(\QQ(m)_p, \TT){=\bigoplus_{v|p}H^1(\QQ(m)_v,\TT)}$ (recall that $\TT=T\otimes \LL_{\cO}(\Gamma)^\iota$).
When $m=1$, we will write $\col^{\clubsuit}$ in place of $\col^{\clubsuit}_m$.

For $\clubsuit \in \{+, -, \bullet \}$, we define the (compact) signed Selmer group $H^{1}_{\mathcal{F}_{\clubsuit}}(\QQ(m), \mathbb{T})$ by setting

 \[H^1_{\mathcal{F}_{\clubsuit}}(\QQ(m),\mathbb{T}) := \ker\left(H^1(\QQ(m),\mathbb{T})\lra  \frac{H^1(\QQ(m)_{p},\mathbb{T})}{\ker(\col^{\clubsuit}_m)} \right).\]

\subsection{Signed Beilinson--Flach classes}\label{S:signedBF}
We now give the proof of  Theorem~\ref{thm_theoremA} stated in the introduction { modulo Theorem~\ref{thm_positionofBFintheanalyticselmer}, which we shall prove in \S\ref{sec:proofoftheorem}. We note that Theorem~\ref{thm_positionofBFintheanalyticselmer} below is the key technical ingredient in the construction of signed Beilinson--Flach classes attached to symmetric squares in the non-ordinary setting. It crucially supplements earlier ideas in this direction by addressing the lack of a ``sufficiently large range of interpolation'' to compensate the growth of denominators.}

For $\lambda,\mu\in\{\pm\alpha\}$ and $m\in \mathcal{N}_\chi$, let 
$\BF_{m,\chi}^{\lambda,\mu}$ be the  Beilinson--Flach element  from \S\ref{subsec:plocalBF}. Via \cite[Proposition~2.4.5]{LZ1}, we make the following identification:
\begin{equation}\label{eq:identificatio_H}
 H^1(\QQ(m),W_f^*\otimes W_f^*(1+\chi)\otimes \cH_{E,k+1}(\Gamma)^\iota)= \cH_{E,k+1}(\Gamma)\, \widehat\otimes\, \HIw(\QQ(m),W_f^*\otimes W_f^*(1+\chi)) .
\end{equation}
Let us write 
\[
\BF_{m,\chi}^{\lambda,\mu}=\sum F^{\lambda,\mu}_i z_i,
\]
where $\sum F^{\lambda,\mu}_i \in \cH_{E,k+1}(\Gamma)$ and $\{z_i\}$ is some fixed $\LL_\cO(\Gamma)$-basis of $ H^1_\Iw(\QQ(m),W_f^*\otimes W_f^*(1+\chi))$. We recall from \cite[\S3.1]{BL16b} that  if $1\le j\le k+1$ and $\theta$ is a Dirichlet character of conductor $p^n>1$, then
\[
F_i^{\lambda,\mu}(\chi^j\theta)=(\lambda\mu)^{-n}c_{n,i,j}
\]
for some constant $c_{n,i,j}$ that is independent of  $\lambda$ and $\mu$. This property is crucially used in the proof of \cite[Theorem~5.4.1]{BLLV}, which can be  recast in our current setting in the following explicit manner:

\begin{lemma}\label{lem:firsttwists}
There exists $\widetilde{\BF}_{m,\chi}^{\lambda,\mu}\in H^1(\QQ(m),W_f^*\otimes W_f^*(1+\chi)\otimes \cH_{E,k+1}(\Gamma)^\iota)$, $\lambda,\mu\in\{\pm\alpha\}$ such that
\[
\begin{pmatrix}
1&1&1&1\\
\alpha^2&\alpha^2 &-\alpha^2&-\alpha^2\\
2\alpha&-2\alpha&0&0\\
0&0&-2\alpha&2\alpha
\end{pmatrix}
\begin{pmatrix}
\BF_{m,\chi}^{\alpha,\alpha}\\
\BF_{m,\chi}^{-\alpha,-\alpha}\\
\BF_{m,\chi}^{\alpha,-\alpha}\\
\BF_{m,\chi}^{-\alpha,\alpha}\end{pmatrix}= 
\begin{pmatrix}
\log_{p,k+1}^{+,(1)}\widetilde{\BF}_{m,\chi}^{\alpha,\alpha}\\ \log_{p,k+1}^{-,(1)}\widetilde{\BF}_{m,\chi}^{-\alpha,-\alpha}\\ \log_{p,k+1}^{(1)}\widetilde{\BF}_{m,\chi}^{\alpha,-\alpha}\\ \log_{p,k+1}^{(1)}\widetilde{\BF}_{m,\chi}^{-\alpha,\alpha}
\end{pmatrix}.
\] 
\end{lemma}

\begin{remark}\label{rk:denominators}
 On comparing denominators, we see that 
\begin{align*}
\widetilde{\BF}_{m,\chi}^{\alpha,\alpha}, \widetilde{\BF}_{m,\chi}^{-\alpha,-\alpha}&\in \cH_{E,(k+1)/2}(\Gamma)\,  \widehat\otimes\, \HIw(\QQ(m),W_f^*\otimes W_f^*(1+\chi)),    \\
\widetilde{\BF}_{m,\chi}^{\alpha,-\alpha}, \widetilde{\BF}_{m,\chi}^{-\alpha,\alpha} &\in \HIw(\QQ(m),W_f^*\otimes W_f^*(1+\chi))
\end{align*}
under the identification \eqref{eq:identificatio_H}.
{We shall show in Theorem~\ref{thm_positionofBFintheanalyticselmer}  that the elements $\widetilde{\BF}_{m,\chi}^{\alpha,\alpha}$ and $\widetilde{\BF}_{m,\chi}^{\alpha,\alpha}$ are further divisible by $\Tw_{-k-1}\log_{p,k+1}^{+,(1)}$ and $\Tw_{-k-1}\log_{p,k+1}^{-,(1)}$, respectively.  Recall from \eqref{eq:twist-divide2} that  we have the equality $\log_{p,k+1}^{\pm,(1)}\Tw_{-k-1}\log_{p,k+1}^{\pm,(1)}=\log_{p,2k+2}^{\pm,(1)}$. This} in turn allows us to define the bounded Beilinson--Flach classes described in Theorem~\ref{thm_theoremA} in the introduction.
\end{remark}

\begin{theorem}
\label{thm_positionofBFintheanalyticselmer}
Suppose that the Dirichlet character $\chi$ verifies the conditions \ref{item_Psi_1} and \ref{item_Psi_2}. Assume also that \ref{item_NV} and \ref{item_Im} hold true. {Let $\cH=\cup_{r\ge0}\cH_{E,r}(\Gamma_1)$}. For $m\in \mathcal{N}_\chi$ and $\eta\in\widehat{\Delta}_m$, we write $e_\eta$ for the corresponding idempotent. For all four choices of $\lambda,\mu\in\{\pm\alpha\}$, there exist $c_m\in \QQ_p[\Delta_m]\otimes \textup{Frac}(\mathcal{H})$ and $z_m\in \bigwedge^2 H^1_\Iw(\QQ(m),W_f^*\otimes W_f^*(1+\chi))$ $($both of which depend only on $m$ and not on the choice of the pair $\lambda,\mu$$)$ satisfying the following properties.
\begin{itemize}
\item[i)]$\BF_{m,\chi}^{\lambda,\mu}= \delta c_m\times\pr_{\lambda,\mu}(z_m)$, where $\delta\in\{\pm\}$ {is} determined according to $\lambda\mu=\delta\alpha^2$.
\item[ii)]  For each $\eta \in \widehat{\Delta}_m$, the element $c_\eta:=e_\eta c_m \in \textup{Frac}(\cH)$ is non-zero.
\item[iii)] For each $\eta \in \widehat{\Delta}_m$, we write $c_\eta=d_\eta/h_\eta$, where $d_\eta,h_\eta\in \cH$ are coprime. Then, $h_\eta$ is coprime to ${\displaystyle\frac{\log_{p,2k+2}^{(1)}}{\log_{p,k+1}^{(1)}}}$\,.
\end{itemize}
\end{theorem}
The proof of this theorem requires the theory of $(\vp,\Gamma)$-modules and Selmer complexes. It will be presented in Section~\ref{sec:proofoftheorem}.

\begin{corollary} \label{cor_BFfactorisation}
In the setting of Theorem~\ref{thm_positionofBFintheanalyticselmer}, there exist 
$$\BF_{m,\chi}^{+},\BF_{m,\chi}^{-},\BF_{m,\chi}^{\bullet}, \BF_{m,\chi}^{\circ}\in H^1_\Iw(\QQ(m),W_f^*\otimes W_f^*( 1+\chi))$$ such that
\begin{equation}\label{eq:BFfactorisation}
\begin{pmatrix}
1&1&1&1\\
\alpha^2&\alpha^2 &-\alpha^2&-\alpha^2\\
2\alpha&-2\alpha&0&0\\
0&0&-2\alpha&2\alpha
\end{pmatrix}
\begin{pmatrix}
\BF^{\alpha,\alpha}_{m,\chi}\\
\BF_{m,\chi}^{-\alpha,-\alpha}\\
\BF_{m,\chi}^{\alpha,-\alpha}\\
\BF_{m,\chi}^{-\alpha,\alpha}
\end{pmatrix}= \begin{pmatrix}
\log_{p,2k+2}^{+,(1)}\BF_{m,\chi}^{+}\\ \log_{p,2k+2}^{-,(1)}\BF_{m,\chi}^{-}\\ \log_{p,k+1}^{(1)}\BF_{m,\chi}^{\bullet}\\\log_{p,k+1}^{(1)}\BF_{m,\chi}^{\circ}
\end{pmatrix}.
\end{equation}
\end{corollary}
\begin{proof}
{The assertion concerning the bottom two rows of \eqref{eq:BFfactorisation}  is a direct consequence of Lemma~\ref{lem:firsttwists} (see also the discussion in Remark~\ref{rk:denominators}).}
We shall prove the divisibility on  the first row; that for the second row can be proved in a similar fashion. 

It follows from {the second row of the factorization given in} Proposition~\ref{prop:rankreductionviaotsukisfunctionals} and Theorem~\ref{thm_positionofBFintheanalyticselmer}(iii) that 
\[
\left(\pr_{\alpha,\alpha}+\pr_{-\alpha,-\alpha}-\pr_{\alpha,-\alpha}-\pr_{-\alpha,\alpha}\right)(c_mz_m)\in{\frac{\log_{p,2k+2}^{+,(1)}}{\log_{p,k+1}^{+,(1)}}}\cH_{E,(k+1)/2}(\Gamma)\,\widehat\otimes\,\HIw(\QQ(m), W_f^*\otimes W_f^*(1+\chi)).
\]
Therefore, if we write 
$$\BF^{\alpha,\alpha}_{m,\chi}+\BF^{-\alpha,-\alpha}_{m,\chi}+\BF^{\alpha,-\alpha}_{m,\chi}+\BF^{-\alpha,\alpha}_{m,\chi}=\sum F_iz_i,$$
where $F_i\in\cH_{E,k+1}(\Gamma)$ and $\{z_i\}$ is  a $\Lambda_\cO(\Gamma)$-basis of $H^1_\Iw(\QQ(m),W_f^*\otimes W_f^*(1+\chi))$, then 
 Theorem~\ref{thm_positionofBFintheanalyticselmer}(i) tells us that
each $F_i$ is divisible by ${\displaystyle\frac{\log_{p,2k+2}^{+,(1)}}{\log_{p,k+1}^{+,(1)}}}$.
Furthermore, Lemma~\ref{lem:firsttwists} says that all the $F_i$'s are also divisible by $\log^{+,(1)}_{p,k+1}$. 
The conclusion follows from growth order considerations.
\end{proof}

\begin{proposition} \label{prop:latticeBF}
In the setting of Theorem~\ref{thm_positionofBFintheanalyticselmer}, there exists an integer $C$ independent of $m$ such that 
\[
C\times\BF_{m,\chi}^{\clubsuit}\in H^1_\Iw(\QQ(m),R_f^*\otimes R_f^*(1+\chi))
\]
for all four choices of $\clubsuit\in\{+,-,\bullet,\circ\}$.
\end{proposition}
\begin{proof}
Let $\lambda,\mu\in\{\pm\alpha\}$ and fix $m$.
Note that $\frac{\lambda\mu}{p^{k+1}}$ is a $p$-adic unit given that $v_p(\lambda)=v_p(\mu)=(k+1)/2$. Write $x_{r}^{\lambda,\mu} \in H^1(\QQ(m)(\mu_{p^r}),W_f^*\otimes W_f^*(1+\chi))$ for the image of the Iwasawa theoretic Beilinson--Flach class $\BF_{m,\chi}^{\lambda,\mu}$. Then by \cite[Theorem 8.1.4(ii)]{LZ1}
\[
C_0\times p^{(k+1)r}x_r^{\lambda,\mu}\in H^1(\QQ(m)(\mu_{p^r}),R_f^*\otimes R_f^*(1+\chi))
\]
for some integer $C_0$ that is independent of $r$, $m$, $\lambda$ and $\mu$.

Let $\cBF^\clubsuit$ be any one of the four linear combinations of Belinson-Flach classes on the left-hand side of \eqref{eq:BFfactorisation} and expand $\cBF^\clubsuit$ with respect to the basis  $\{z_i\}$ of $\HIw(\QQ(m),W_f^*\otimes W_f^*(1+\chi))$, say $\cBF^\clubsuit=\sum F_iz_i$. Let $\log^\clubsuit$ be the corresponding logarithm on the right-hand side of \eqref{eq:BFfactorisation}. Then,
\begin{itemize}
\item $\log^\clubsuit\mid F_i$ for all $i$;
\item $F_i=O(\log_p^{k+1})$;
\item For all $r\ge1$, we have $p^{-(k+1)r}||F_i||_{\rho_r}$ is bounded independently of $i$, $r$, $m$ and the choice of $\clubsuit$.
\end{itemize}
Here $\rho_r=p^{-1/p^{r-1}(p-1)}$ and $||\bullet||_{\rho_r}$ is the sup-norm on power series as defined  in \cite[\S2.1]{BL16b}. Consequently $||F_i/\log^\clubsuit||_{\rho_r}$ is bounded independently of $i$, $r$, $m$ and the choice of $\clubsuit$. Hence, $\cBF^\clubsuit/\log^\clubsuit$ are bounded classes as required.
\end{proof}
 Recall that $T := \Sym^2 R_f^{*}(1 + \chi)$ and $\mathbb{T} := T \otimes \LL_\cO(\Gamma)^{\iota}$.  

\begin{corollary} \label{cor:signedBFinsymmsquare}
In the setting of Theorem~\ref{thm_positionofBFintheanalyticselmer} and for $\clubsuit \in \{ +,-,\bullet \}$, we have
\[ C\times\BF_{m,\chi}^{\clubsuit}\in H^1(\QQ({m}),\TT),\qquad \BF_{m,\chi}^{\circ} = 0, \]
where we consider $H^1(\QQ({m}),\TT)$ as a subgroup of $\HIw(\QQ(m),R_f^*\otimes R_f^*(1+\chi))$ via the decomposition of $G_{\QQ(m)}$-representations $R_f^*\otimes R_f^*(1+\chi)=T\oplus \bigwedge^2 R_f^*(1+\chi)$.
\end{corollary}
\begin{proof}
The first part of the corollary follows from Proposition \ref{prop:thedichotomy}, Corollary \ref{cor_BFfactorisation} and Proposition \ref{prop:latticeBF}. For the second part of the corollary, note that
\[ \BF_{m,\chi}^{\circ} = (\BF_{m,\chi}^{-\alpha,\alpha} - \BF_{m,\chi}^{\alpha,-\alpha})/2\alpha \] 
Under our assumption on the parity of the Dirichlet character $\chi$, we show that 
\[ \BF_{m,\chi}^{\lambda,\mu} = \BF_{m,\chi}^{\mu,\lambda} \]
for any choices of $\lambda,\mu \in \{\pm\alpha\}$ in Section~\ref{sec:proofoftheorem}. In particular, see Remark~\ref{rem:sym2orsymmetricproductdoesntmatter}. Thus, on taking $\lambda=-\mu=\alpha$, this shows that $\BF_{m,\chi}^{\circ} = 0$, as required.
\end{proof}

{We now show that the bounded Beilinson--Flach classes  satisfy the Selmer conditions we defined in \S\ref{sec:signedcoleman}. This allows us to apply the Euler system machinery to obtain one inclusion of Conjecture~\ref{conj:signedmainconjecture} (see Theorem~\ref{thm_signedmainconjecture}). For the rest of the section,}
we assume that the character $\chi$ verifies \ref{item_Psi_1} and \ref{item_Psi_2}. Suppose also that (\textbf{Im}) holds true. We also fix an integer $m\in \mathcal{N}_\chi$.
Let us  recall the following ``geometric'' property of the unbounded Beilinson--Flach classes.

\begin{proposition}\label{prop:BFgeom}
For an integer $j\in [-k,0]$ and $\lambda,\mu\in\{\pm\alpha\}$, the natural image of $\loc_p\left(\BF_{m,\chi}^{\lambda,\mu,(j)}\right)$ in $H^1(\QQ(m)(\mu_{p^r})_p,W_f^*\otimes W_f^*(1+j+\chi))$ belongs to the Bloch-Kato subgroup $H^1_{\rm f}(\QQ(m)(\mu_{p^r})_p,W_f^*\otimes W_f^*(1+j+\chi))$.
\end{proposition}
\begin{proof}
This is \cite[Proposition~3.3.3]{KLZ2}, since $H^1_{\rm f}=H^1_{\rm g}$ in this case (see \cite[Proposition 8.1.3]{LZ1}).
\end{proof}

\begin{proposition}
\label{PROP_PositionofSignedBF}Let $\clubsuit \in \{+, -, \bullet \}$ and $v$ a prime of $\QQ(m)$ above $p$. Then
$\res_v\left(\BF_{m,\chi}^\clubsuit\right)\in \ker\col_{\QQ(m)_v}^\clubsuit$, where $\res_v$ denotes the localization map 
\[
H^1_{\Iw}(\QQ(m),W_f^*\otimes W_f^*(1+\chi))\rightarrow H^1_{\Iw}(\QQ(m)_v,W_f^*\otimes W_f^*(1+\chi)).
\]
\end{proposition}
\begin{proof}
We shall only consider the ``{$\clubsuit = +$}" case. {The other two cases can be proved similarly.}
Let us set $F=\QQ(m)_v$  and write 
\[
z:=\res_v\left(\BF_{m,\chi}^{\alpha,\alpha}+\BF_{m,\chi}^{-\alpha,-\alpha}+\BF_{m,\chi}^{\alpha,-\alpha}+\BF_{m,\chi}^{-\alpha,\alpha}\right).
\]
Proposition~\ref{prop:BFgeom} tells us that the image of $z$ in $H^1(F(\mu_{p^r}),W_f^*\otimes W_f^*(1+\chi))$ belongs to the Bloch-Kato subspace $H^1_{\rm f}(F(\mu_{p^r}),W_f^*\otimes W_f^*(1+\chi))$ for all $r\ge0$.
By the interpolative properties of Perrin-Riou's map, this implies that 
$$\cL_{W_f^*\otimes W_f^*(1+\chi),F}(z)\in F\otimes \cH_E\otimes \Dcris(W_f^*\otimes W_f^*(1+\chi))$$ 
vanishes at all finite characters on $\Gamma$. Let $\cL^{(1)}_{\pm,F}$ be the morphism given in \S\ref{sec:signedcoleman}. Then, both $\cL_{+,F}^{(1)}(z)$ and $\cL_{-,F}^{(1)}(z)$ vanish at all finite characters of $\Gamma$. 

By an abuse of notation, we shall denote $\cL_{+,F}^{(1)}$ (respectively $\col_F^+$) composed with the projection map $W_f^*\otimes W_f^*(\chi)\lra W_{2}^*(1+\chi)$ by the same symbol.
Note that 
\[
\cL_{+,F}^{(1)}(z)=\left(\log_{p,2k+2}^{+,(1)}\right)^2\col_F^+\circ\loc_p\left(\BF_{m,\chi}^+\right).
\]
Therefore, $\col_F^+\circ\loc_p\left(\BF_{m,\chi}^+\right)$ vanishes at infinitely many finite characters of $\Gamma$ (the ones that do not vanish at $\log_{p,2k+2}^{+,(1)}$). This forces $\col_F^+\circ\loc_p\left(\BF_{m,\chi}^+\right)$ to vanish, as required.
\end{proof}

{
\begin{corollary}
We have $C \times \BF_{m, \chi}^{\clubsuit} \in H^1_{\mathcal{F}_{\clubsuit}}(\QQ(\mu_m),\mathbb{T})$ for $\clubsuit \in \{+, -, \bullet \}$.
\end{corollary}
\begin{proof}
This follows immediately from Proposition~\ref{PROP_PositionofSignedBF} and Corollary~\ref{cor:signedBFinsymmsquare}. 
\end{proof} 
}

\subsection{Doubly signed main conjectures}
\label{sec:nonordIwasawaESargument}

Recall that $T:=\textup{Sym}^2R_f^*(1 + \chi)$ and $\mathbb{T} = T \otimes \Lambda^{\iota}$. We now define doubly signed compact and discrete Selmer groups as well as doubly signed $p$-adic $L$-functions in the spirit of \cite{BLLV}. 
\begin{defn}
Let $\mathcal{S}$ denote the set of pairs $\{(+, -), (+, \bullet), (-, \bullet) \}$. For $\mathfrak{S} = (\clubsuit, \spadesuit) \in \mathcal{S}$, we define the following objects
\begin{itemize}
 \item A compact Selmer group $H^1_{{\mathfrak{S}}}(\QQ,\mathbb{T})$, given by
 $$H^1_{{\mathfrak{S}}}(\QQ,\mathbb{T}) := \ker\left(H^1(\QQ,\mathbb{T})\longrightarrow\frac{H^1(\Qp,\mathbb{T})}{\ker\left(\col^{\clubsuit}\right) \cap \ker\left(\col^{\spadesuit}\right)}\right)\,.$$
\item A discrete Selmer group $\Sel_{\mathfrak{S}}(T^\vee(1)/\QQ(\mu_{p^\infty}))$, given by the kernel of the restriction map
 \[
  H^1(\QQ(\mu_{p^\infty}),T^\vee(1))\longrightarrow \prod_{v|p}\frac{H^1(\QQ(\mu_{p^\infty})_v,T^\vee(1))}
  {H^1_{\mathfrak{S}}(\QQ(\mu_{p^\infty})_v,T^\vee(1))} \times \prod_{v \nmid p} \frac{H^1(\QQ(\mu_{p^\infty})_v,T^\vee(1))}{H^1_{\rm f}(\QQ(\mu_{p^\infty})_v,T^\vee(1))},
 \]
 where $v$ runs through all primes of $\QQ(\mu_{p^\infty})$, and for $v \mid p$ the local condition $H^1_{\mathfrak{S}}(\QQ(\mu_{p^\infty})_v,T^\vee(1))$ is the orthogonal complement of $\ker\left(\col^{\clubsuit}\right)\cap\ker\left(\col^{\spadesuit}\right)$ under the local Tate pairing. 
 \item In the setting of Theorem~\ref{thm_positionofBFintheanalyticselmer}, we define the \emph{doubly-signed $p$-adic $L$-function} by setting
$$\mathcal{L}_{\mathfrak{S}}:= \col^{\clubsuit}\circ\res_p\left(\BF^{\spadesuit}_{1, \chi}\right) \in C^{-1}\LL_\cO(\Gamma).$$
\end{itemize}
\end{defn} 
\begin{remark} Interchanging the roles of $\clubsuit$ and $\spadesuit$ has the effect of multiplying $\mathcal{L}_{\mathfrak{S}}$ by $-1$. This is the content of \cite[Proposition 5.3.4]{BLLV} (see Proposition~\ref{prop_antisymofcolBF} below for its incarnation in our setting). Since we are only interested in the ideal generated by $\mathcal{L}_{\mathfrak{S}}$, the ambiguity of sign is not an issue for us.
\end{remark}

We are now in a position to formulate a doubly-signed Iwasawa main conjecture for the symmetric square representation of a non-$p$-ordinary eigenform.
\begin{conjecture} \label{conj:signedmainconjecture}
For every $\mathfrak{S} \in \mathcal{S}$ and every character $\eta$ of $\Gamma_{\mathrm{tors}}$, the module $e_{\eta}\Sel_{\mathfrak{S}}(T^\vee(1)/\QQ(\mu_{p^\infty}))$ is $\Lambda(\Gamma_{1})$-cotorsion and
\[ \mathrm{char}_{\Lambda_\cO(\Gamma_{1})}\left( e_{\eta}\Sel_{\mathfrak{S}}(T^\vee(1)/\QQ(\mu_{p^\infty}))^{\vee}\right) = (e_{\eta}\mathcal{L}_{\mathfrak{S}}) \]
as ideals of $\Lambda(\Gamma_{1})\otimes\QQ_p$, with equality away from the support of $\mathrm{coker}(\col^{\clubsuit})$ and $\mathrm{coker}(\col^{\spadesuit})$.
\end{conjecture}

\begin{proposition} \label{prop:nonzerop-adicLfunction}
Suppose we are in the setting of Theorem~\ref{thm_positionofBFintheanalyticselmer}. Then there exists a choice of $\mathfrak{S} \in S$ such that $e_{\omega^{j}}\mathcal{L}_{\mathfrak{S}} \neq 0$.
\end{proposition}
\begin{proof}
By Lemma ~\ref{lem:factorisation} and Definition ~\ref{defn:signedcolemanmaps}, we have 
\[
\begin{pmatrix}
1&1&1&1\\
\alpha^2&\alpha^2 &-\alpha^2&-\alpha^2\\
2\alpha&-2\alpha&0&0\\
0&0&-2\alpha&2\alpha
\end{pmatrix}
\begin{pmatrix}
\cL_{\alpha,\alpha,F}^{(1)}\\
\cL_{-\alpha,-\alpha,F}^{(1)}\\
\cL_{\alpha,-\alpha,F}^{(1)}\\
\cL_{-\alpha,\alpha,F}^{(1)}
\end{pmatrix} = \begin{pmatrix}
\log_{p,2k+2}^{-,(1)}\col^{-}\\ \log_{p,2k+2}^{+,(1)}\col^{+}(z)\\ \log_{p,k+1}^{(1)}\col^{\bullet}(z)\\\log_{p,k+1}^{(1)}\col^{\circ}
\end{pmatrix}.
\]
In particular,
\[ \cL_{\alpha,\alpha}^{(1)}(z) = \frac{\log_{p,2k+2}^{-,(1)}}{4}\col^{-}(z)+\frac{\log_{p,2k+2}^{+,(1)}}{4\alpha^2}\col^{+}(z)+\frac{\log_{p,k+1}^{(1)}}{4\alpha}\col^{\bullet}(z)\,. \]
Similarly,  Corollary ~\ref{cor_BFfactorisation} gives
\[ \BF_{1,\chi}^{\alpha,-\alpha} = \frac{\log_{p,2k+2}^{+,(1)}}{4}\BF_{1,\chi}^{+}-\frac{\log_{p,2k+2}^{-,(1)}}{4\alpha^2}\BF_{1,\chi}^{-}-\frac{\log_{p,k+1}^{(1)}}{4\alpha}\BF_{1,\chi}^{\circ}\,. \]
By Corollary~\ref{cor:signedBFinsymmsquare}, we know that $\BF_{1,\chi}^{\circ} = 0$. Hence, $\cL_{\alpha,\alpha}^{(1)}(\BF_{1,\chi}^{\alpha,-\alpha})$ is an {$\mathcal{H}_{E,k+1}(\Gamma)$}-linear combination of the terms $\col^{\clubsuit}\circ\mathrm{res}_{p}(\BF_{1,\chi}^{\spadesuit})$ for $(\clubsuit, \spadesuit) \in \mathcal{S}$. By (\ref{eqn:antisymmetryrelation}), we know that  $\cL_{\alpha,\alpha}^{(1)}(\BF_{1,\chi}^{\alpha,-\alpha})$ is a non-zero multiple of the \emph{geometric $p$-adic $L$-function} and hence is non-zero. We conclude that there exists at least one $\mathfrak{S} = (\clubsuit, \spadesuit) \in \mathcal{S}$ such that $\col^{\clubsuit}\circ\mathrm{res}_{p}(\BF_{1,\chi}^{\spadesuit})$ is non-zero.
\end{proof}

We can now give the proof of Theorem~\ref{thm:intro}.

\begin{theorem} \label{thm_signedmainconjecture} 
Suppose that the hypotheses \ref{item_Psi_1}--\ref{item_Psi_3}, \ref{item_NV} and \ref{item_Im}  hold true. Then for every $j \in \{k+2,\ldots,2k+2 \}$ even and $\mathfrak{S} \in \mathcal{S}$ that validates the conclusion of Proposition~\ref{prop:nonzerop-adicLfunction}, the $\omega^{j}$-isotypic component $e_{\omega^{j}}\Sel_{\mathfrak{S}}(T^\vee(1)/\QQ(\mu_{p^\infty}))$ is $\Lambda_\cO(\Gamma_{1})$-cotorsion and we have
\[\mathrm{char}_{\Lambda_\cO(\Gamma_{1})}\left( e_{\omega^{j}}\Sel_{\mathfrak{S}}(T^\vee(1)/\QQ(\mu_{p^\infty}))^{\vee}\right)\,\,\big{|}\,\, (e_{\omega^{j}}\mathcal{L}_{\mathfrak{S}}) \]
as ideals of $\Lambda_\cO(\Gamma_{1}) \otimes \QQ_{p}$.
\end{theorem}
\begin{proof}
This theorem  follows from the same proof of \cite[Theorem 6.2.4]{BLLV}, using the (rank-one) locally restricted Euler system machinery we have defined  above. The quadruple sign $\mathfrak{S} = \{(\bigtriangleup,\square),(\bullet,\circ)\}$ used therein corresponds to our double sign $\mathfrak{S} = (\clubsuit, \spadesuit) \in \mathcal{S}$.  The additional hypothesis \ref{item_Psi_3}  ensures the big image condition on $T$ in order to apply the Euler system machinery holds (c.f. \cite[\S5.2]{LZ2}).
\end{proof}

\section{$(\vp,\Gamma)$-modules and analytic main conjectures} \label{sec:analyticmainconjectures}
Our main goal in this section is to give proofs of Theorem~\ref{thm_horizontalESforSym2}(iv), Theorem~\ref{thm_positionofBFintheanalyticselmer} and Corollary~\ref{cor:signedBFinsymmsquare}. These results are crucial for the construction of bounded Beilinson--Flach classes as well as to translate our results on the signed Iwasawa main conjectures into the \emph{analytic} language of Pottharst~\cite{jayanalyticselmer} and Benois~\cite{Benois2015}.

Fix once and for all an integer $r \in \mathcal{N}_\chi$ as in Definition~\ref{def:Kolyvaginprimesfortwists}, where the Dirichlet character $\chi$ is given as in Section~\ref{sec_selmerstructure} verifying \ref{item_Psi_1} and \ref{item_Psi_2}. Fix also  a character $\nu \in \widehat{\Delta}_r$ and set $\psi:=\chi\nu$. As in Section~\ref{sec_selmerstructure}, we also fix an integer $j$ in the interval $[k+2,2k+2]$. Recall that  $V_{\psi,j}$ denotes $\textup{Sym}^2W_f^*(1-j)\otimes\omega^j\psi$,
which sits inside
$$W_{\psi,j}=W_f^*\otimes W_f^*(1-j) \otimes \omega^j\psi\,=\left({\bigwedge\!}^2 W_f^*(1-j)\otimes\omega^j\psi\right)\oplus V_{\psi,j}.$$
For notational simplicity, we write $\mathscr{V}_{\psi}=V_{\psi,j}\otimes \omega^{-j}$ and $\mathscr{W}_{\psi}=W_{\psi,j}\otimes \omega^{-j}$ throughout this section.

We shall make use of the identification  (which arises from the inflation-restriction sequence)
$$H^1_\Iw(K^{\Delta},W_{\psi,j})\stackrel{\sim}{\lra} H^1_\Iw(K,\mathscr{W}_{\psi})^{\omega^{-j}}$$
(and likewise, for the representations $\mathscr{V}_{\psi}$ and $V_{\psi,j}$) for any abelian extension $K$ of $\QQ$ that contains $\mu_{p}$, where $\Delta=\textup{Gal}(\QQ(\mu_p)/\QQ)$ and $K^\Delta$ denotes the fixed field of $\Delta$.
\subsection{Local preparation}
\label{subsec_Localanalysis}
{In this subsection, we introduce the $p$-adic Hodge-theoretic objects that we shall rely on in the proofs of Theorem~\ref{thm_horizontalESforSym2}(iv) and Theorem~\ref{thm_positionofBFintheanalyticselmer}. We also prove Lemma~\ref{lemma_twobfclasssesareLI}, which is a statement on local classes and will serve as a key ingredient in the proof of Theorem~\ref{thm_horizontalESforSym2}(iv).}

We fix a generator $\{\varepsilon_{(n)}\}_{n}$ of $\varprojlim \mu_{p^n}=:\ZZ_p(1)$.
Recall from  Corollary~\ref{cor:sym2splits} the $G_{\QQ_p}$-subrepresentations $W_1$ and $W_2$ of $\textup{Sym}^2W_f$. We define for $i=1,2$ the Dieudonn\'e module
$$D_i:=\Dcris\left(W_i^*(1-j)\otimes\psi\right).$$
We also set 
$$D_0:=\Dcris\left({\bigwedge\!}^2W_f^*(1-j)\otimes\psi\right).$$
In particular, this gives the decompositions
$$\Dcris(\mathscr{W}_{\psi})=D_0\oplus D_1\oplus D_2,\quad \Dcris(\mathscr{V}_{\psi})=D_1\oplus D_2\,.$$
The crystalline Frobenius $\varphi$ acts on $D_1$ by { $\alpha_{\psi,j}:=p^{j-1}\psi(p)/\alpha^2=-p^{j-k-2}\epsilon_f^{-1}\psi(p)$, whereas it acts on $D_0$ by $-\alpha_{\psi,j}=p^{j-k-2}\epsilon_f^{-1}\psi(p)$}. Recall from $\eqref{eq:evecs}$ the $\varphi$-eigenvectors $v_{\pm\alpha} \in \Dcris(W_f^*)$. If we fix a non-zero vector $v_{j,\psi} \in \Dcris(E(1-j)\otimes\psi)$, we have the following $\vp$-eigenvectors in  $\Dcris(\mathscr{W}_{\psi})$: $\omega_{\pm,\pm}:=v_{\pm\alpha}\otimes v_{\pm\alpha} \otimes v_{j,\psi}$. {We can check} that 
\begin{itemize}
\item $D_0 =\textup{span}\{\omega_{+-}-\omega_{-+}\}$,
\item $D_1 =\textup{span}\{\omega_{++}-\omega_{--}\}$,
\item $D_2 =\textup{span}\{\omega_{++}+\omega_{--}, \omega_{+-}+\omega_{-+}\}$.
\end{itemize}

\begin{defn}
Given $\lambda \in \{\pm \alpha\}$, we let $\delta_\lambda:\QQ_{p}^\times\ra E^\times$ denote the character that is given by $\delta_\lambda(p)=1/\lambda$ and $\delta(u)=1$ for $u\in \ZZ_{p}^\times$. We also write $\delta_{\psi,j}:\QQ_{p}^\times\ra E^\times$ for the character which is given by $\delta_{\psi,j}(x)=|x|^{1-j}_px^{1-j}\delta_\psi(x)$ where $\delta_{\psi}(p)=\psi^{-1}(p)$ and $\delta_{\psi}(u)=1$.
\end{defn}

For each $\varphi$-eigenspace  $\Dcris(W_f^*)^{\varphi=1/\lambda}$ of  $\Dcris(W_f^*)$ (where $\lambda=\pm\alpha$ as above), there is a unique rank-one $(\varphi,\Gamma)$-submodule $\mathbb{D}_\lambda\subset\mathbb{D}_f:=D_{\textup{rig}}^\dagger(W_f^*)$, which is of the form $\mathcal{R}_{E}(\delta_\lambda)$, where $\mathcal{R}_{E}$ is the Robba ring over $E$ (see \cite[\S2.2]{benois14}). More precisely, $\mathbb{D}_\lambda$ is the free $\mathcal{R}_{E}-$module generated by an element $e_\lambda \in \mathbb{D}_\lambda$ for which we have
$$\varphi(e_\lambda)=\delta_{\lambda}(p)\cdot e_\lambda\,\,\,, \,\,\, \tau(e_\lambda)=\delta_\lambda(\chi_{\cyc}(\tau))\cdot e_\lambda\,\, (\forall \tau \in \Gamma)\,.$$
We also set $\mathbb{D}_{\psi,j}$ to be $\mathcal{R}_{E}(\delta_{\psi,j})$ and define 
$$\mathbb{D}_{f,j}:=\mathbb{D}_f\otimes\mathbb{D}_{\psi,j} \cong D_{\textup{rig}}^\dagger(W_f^*(1-j)\otimes\psi)\,\,,\,\,\mathbb{D}_{\lambda,j}:=\mathbb{D}_{\lambda}\otimes \mathbb{D}_{\psi,j}$$

In what follows, the following $(\varphi,\Gamma)$-subquotients (all of which are necessarily crystalline) of $D_{\textup{rig}}^\dagger(\mathscr{W}_{\psi})$ will play a crucial role. Let $\lambda,\mu \in \{\pm\alpha\}$.
\begin{itemize}
\item $D_{\textup{rig}}^\dagger(\mathscr{W}_{\psi})^{\lambda,\mu}:=\mathbb{D}_\lambda \otimes \mathbb{D}_{\mu,j}$.
\item $D_{\textup{rig}}^\dagger(\mathscr{W}_{\psi})^{\lambda,\circ}:=\mathbb{D}_\lambda \otimes \mathbb{D}_{f,j}$.
\item $D_{\textup{rig}}^\dagger(\mathscr{W}_{\psi})_{/\lambda,\circ}:=D_{\textup{rig}}^\dagger(\mathscr{W}_{\psi})/D_{\textup{rig}}^\dagger(\mathscr{W}_{\psi})^{\lambda,\circ}$.
\item $D_{\textup{rig}}^\dagger(\mathscr{W}_{\psi})_{/\lambda,\mu}:=(\mathbb{D}_{f}/\mathbb{D}_\lambda)\otimes \mathbb{D}_{\mu,j}\subset  D_{\textup{rig}}^\dagger(\mathscr{W}_{\psi})_{/\lambda,\circ}$.
\item $D_{\textup{rig}}^\dagger(\mathscr{W}_{\psi})^{\lambda,\mu,+}:=\left(\mathbb{D}_\lambda\otimes\mathbb{D}_f+\mathbb{D}_f\otimes\mathbb{D}_\mu\right)\otimes \mathbb{D}_{\psi,j}$.
\item  $D_{\textup{rig}}^\dagger(\mathscr{W}_{\psi})^{\lambda,\mu,-}:=D_{\textup{rig}}^\dagger(\mathscr{W}_{\psi})/D_{\textup{rig}}^\dagger(\mathscr{W}_{\psi})^{\lambda,\mu,+}$.
\end{itemize}

Let  $D=D_{\textup{rig}}^\dagger(\mathscr{W}_{\psi})_{?}$ or  $D_{\textup{rig}}^\dagger(\mathscr{W}_{\psi})^{?}$ be one of the $(\vp,\Gamma)$-modules above. We write $\Dcris(\mathscr{W}_{\psi})_?$ (or $\Dcris(\mathscr{W}_{\psi})^?$) for the corresponding Dieudonn\'e module $\Dcris(D)$, {where $\Dcris(-)$ denotes the functor defined in \cite[\S2, P.341]{nakamura17}}.
We write $\frak{d}_{/\lambda,\circ}$ for  the natural projection map
$$\frak{d}_{/\lambda,\circ}:\Dcris(\mathscr{W}_{\psi})\lra \Dcris(\mathscr{W}_{\psi})_{/\lambda,\circ}\,.$$ 
We also have the following canonical short exact sequence:
\begin{equation}
\label{eqn:pmcirccrissequence}
0\lra \Dcris(\mathscr{W}_{\psi})_{/\lambda,\mu}\stackrel{\frak{d}^{*}}{\lra}  \Dcris(\mathscr{W}_{\psi})_{/\lambda,\circ}\stackrel{\frak{d}^-_{\lambda,\mu}}{\lra}  \Dcris(\mathscr{W}_{\psi})^{\lambda,\mu,-}\lra 0.
\end{equation}

\begin{lemma}
\label{lemma:locateclassbyphiaction}
$ \ker(\frak{d}^-_{\lambda,\mu}) = \left(\Dcris(\mathscr{W}_{\psi})_{/\lambda,\circ}\right)^{\varphi=\frac{p^{j-1}\psi^{-1}(p)}{-\lambda\mu}}$
\end{lemma}
\begin{proof}
Note that $\Dcris(\mathscr{W}_{\psi})_{/\lambda,\mu}$ is one dimensional over $E$. By comparing the action of  $\varphi$, we see that it is isomorphic to $\Dcris(\mathbb{D}_{-\lambda}\otimes\mathbb{D}_{\mu,j})$ as $\vp$-modules. The proof now follows from the exact sequence (\ref{eqn:pmcirccrissequence}), as the image of $\frak{d}^*$ may be identified with the (one dimensional) $\displaystyle\frac{{p^{j-1}}\psi^{-1}(p)}{-\lambda\mu}$-eigenspace for the $\varphi$-action on $\Dcris(\mathscr{W}_{\psi})_{/\lambda,\circ}$.
\end{proof}
Given an integer $m\in \{1,\cdots,p-1\}$, we let $[m]\in \QQ_p^\times$ denote its Teichm\"uller lift. We also let $\textup{Tr}_j: \QQ_p(\mu_p)^{\omega^{-j}}\ra \QQ_p$ denote the twisted trace map induced by 
$$\frac{1}{p-1}\sum_{r=1}^{p-1}\, [r]^j\varepsilon_{(1)}^{r} \mapsto 1\,.$$

\begin{defn}
\label{def:theexpstarevaluated}
For any crystalline $G_{\QQ_p}$-representation or a $(\varphi,\Gamma)$-module $D$, we denote the composition of the arrows
\begin{equation}\label{eqn:compositumofexpwithidempotent}
H^1(\QQ_p(\mu_p),D)^{\omega^{-j}}\stackrel{\exp^*}{\lra}  \QQ_p(\mu_p)^{\omega^{-j}} \otimes \Dcris(D)\stackrel{\textup{Tr}_j\otimes 1}{\lra}  \Dcris(D)
\end{equation}
 by $\omega^{-j}\,\circ\, \exp^*$.  {Here, $\exp^*$ is the dual exponential map given in \cite[\S3, P.360]{nakamura17}.}\footnote{Note that once we identify $\QQ_p(\mu_p)$ with $\QQ_p[\Delta]$ via the generator $\varepsilon_{(1)}$ of $\mu_p$,  the map $\textup{Tr}_j$ agrees with the map $\omega^{-j}:\,\QQ_p[\Delta]^{\omega^{-j}}\ra \QQ_p$, which justifies the notation we have chosen for the composition (\ref{eqn:compositumofexpwithidempotent}).} More generally, if $\theta$ is a character of $\Gamma_n:=\Gamma/\Gamma^{p^n}$ that does not factor through $\Gamma_{n-1}$ (where $n$ is a positive integer), we may define a map
 $$\omega^{-j}\circ\,\exp^*:\,H^1(\QQ_p(\mu_{p^n}),D)^{\omega^{-j}\theta}\lra \Dcris(D)$$
 starting off with the map $\omega^{-j}\theta:\,\QQ_p[\Delta\times\Gamma_n]^{\omega^{-j}\theta}\ra \QQ_p$ and identifying it $($via the generator $\varepsilon_{(n)}$ of $\mu_{p^n}$) with a twisted trace map $\QQ_p(\mu_{p^n})\ra \QQ_p$.
 \end{defn}
 \begin{defn}\label{defn:singular}
Given a finite extension $K$ of $\QQ_p$ and a crystalline $G_K$-representation $V$, we set 
$$H^1_s(K,V):=H^1(K,V)/H^1_\textup{f}(K,V)$$
and call it the singular quotient of $H^1(K,V)$. For each positive integer $n$, we further set 
$$\res_{/\textup{f}}: H^1(\QQ(\mu_{p^n}),V)\lra H^1_s(\QQ_p(\mu_{p^n}),V) $$
(the singular projection) to denote the composition of the arrows 
$$H^1(\QQ(\mu_{p^n}),V)\stackrel{\res_p}{\lra}H^1(\QQ_p(\mu_{p^n}),V)\stackrel{\frak{s}}{\lra} H^1_s(\QQ_p(\mu_{p^n}),V)\,,$$
where $\frak s$ is the natural projection map.
\end{defn}
\begin{lemma}
\label{lemma_twobfclasssesareLI}
Let $x,y \in H^1(\QQ_p(\mu_p),\mathscr{W}_{\psi})^{\omega^{-j}}$ be two classes with non-trivial singular projection (meaning that their images in $H^1_s(\QQ_p(\mu_p),\mathscr{W}_{\psi})$ under the map $\frak{s}$ given in  Definition~\ref{defn:singular} are non-trivial) such that
\begin{itemize}
\item $\frak{d}_{/\lambda,\circ}\circ\omega^{-j}\,\circ\, \exp^*(x) \in \left(\Dcris(\mathscr{W}_{\psi})_{/\lambda,\circ}\right)^{\varphi=\frac{p^{j-1}\psi^{-1}(p)}{\alpha^2}}$,
\item $\frak{d}_{/\lambda,\circ}\circ\omega^{-j}\,\circ\, \exp^*(y) \in \left(\Dcris(\mathscr{W}_{\psi})_{/\lambda,\circ}\right)^{\varphi=-\frac{p^{j-1}\psi^{-1}(p)}{\alpha^2}}\,.$
\end{itemize}
Then the classes $\frak{s}(x)$ and $\frak{s}(y)$ are linearly independent {over $E$} in $H^1_s(\QQ_p(\mu_p),\mathscr{W}_{\psi})$. 
\end{lemma}
\begin{proof}
This is clear, as the images of $x$ and $y$ under the dual exponential map (that factors through the singular quotient) composed with the map $\frak{d}_{/\lambda,\circ}\circ\omega^{-j}$ fall within different eigenspaces.   
\end{proof}
\subsection{Linear independence of Beilinson--Flach classes} 
{We are now ready to complete the proof of Theorem~\ref{thm_horizontalESforSym2}{(iv)}, which follows as an immediate consequence of Corollary~\ref{cor_linearlyindependenttwistedBF} to Theorem~\ref{thm_linearlyindependent} below. Theorem~\ref{thm_linearlyindependent}  concerns the $p$-local images of the Beilinson--Flach elements and its proof relies on the criterion on linear independence  established in Lemma~\ref{lemma_twobfclasssesareLI} above.}

For $\lambda,\mu\in\{\pm\alpha\}$, the Beilinson--Flach classes $\BF_{r,\chi}^{\lambda,\mu}\in H^1(\QQ(r),W_f^*\otimes W_f^*(1+\chi)\otimes \cH_{E,k+1}(\Gamma)^\iota)$ from \S\ref{subsec:plocalBF} give rise to the  classes in $\BF_{\psi}^{\lambda,\mu,(j)} \in H^1_\Iw(\QQ(\mu_p),\mathscr{W}_{\psi}\otimes  \cH_{E,k+1}(\Gamma)^\iota)$. 
We let $\mathbf{bf}^{\lambda,\mu} \in H^1(\QQ(\mu_p),\mathscr{W}_{\psi})^{\omega^{-j}}$ denote the images of $\BF_{\psi}^{\lambda,\mu,(j)}$ under the composition 
$$H^1_\Iw(\QQ,\mathscr{W}_{\psi}\otimes \cH_{E,k+1}(\Gamma)^\iota)\lra H^1(\QQ(\mu_p),\mathscr{W}_{\psi})\lra H^1(\QQ(\mu_p),\mathscr{W}_{\psi})^{\omega^{-j}}\,.$$

\begin{remark}\label{rk:comapre_classes}
The classes $c_1^{\lambda,\mu} \in H^1(\QQ,W_{\psi,j})$ that we have considered in the proof of Theorem~\ref{thm_horizontalESforSym2} maps to the class $\mathbf{bf}^{\lambda,\mu}$ under the canonical isomorphism 
$$H^1(\QQ,W_{\psi,j})\lra H^1(\QQ(\mu_p),\mathscr{W}_{\psi})^{\omega^{-j}}.$$ 
\end{remark}

\begin{theorem}
\label{thm_linearlyindependent}
The classes $\res_{/\textup{f}}(\mathbf{bf}^{\lambda,\mu})$ and $\res_{/\textup{f}}(\mathbf{bf}^{\lambda,-\mu})$ are linearly independent in $H^1_s(\QQ_p(\mu_p),\mathscr{W}_{\psi})^{\omega^{-j}}$.
\end{theorem}
\begin{proof}
By Lemmas~\ref{lemma:locateclassbyphiaction} and \ref{lemma_twobfclasssesareLI} , the theorem will follow once we verify the following two properties.
\begin{itemize}
\item[(i)] $\frak{d}_{/\lambda,\circ}\circ\omega^{-j}\,\circ\, \exp^*_{\mathscr{W}_{\psi}^*(1)}\circ\, \res_p(\mathbf{bf}^{\lambda,\mu}) \in \ker(\frak{d}_{\lambda,\mu}^-)$ for $\mu\in\{\alpha,-\alpha\}$.
\item[(ii)] $\res_p(\mathbf{bf}^{\lambda,\mu})$ and $\res_p(\mathbf{bf}^{\lambda,-\mu})$ are non-trivial.
\end{itemize} 
The  property (i) is immediate by the commutativity of the following diagram
$$\xymatrix{H^1_\Iw(\QQ_p,D_{\textup{rig}}^\dagger(\mathscr{W}_{\psi}))^{\omega^{-j}}\ar[r]\ar[d]&H^1_\Iw(\QQ_p,D_{\textup{rig}}^\dagger(\mathscr{W}_{\psi})^{\lambda,\mu,-})^{\omega^{-j}}\ar[d]\\
H^1(\QQ_p(\mu_p),D_{\textup{rig}}^\dagger(\mathscr{W}_{\psi}))^{\omega^{-j}}\ar[r]\ar[d]_{\omega^{-j}\circ\, \exp^*} &H^1(\QQ_p(\mu_p),D_{\textup{rig}}^\dagger(\mathscr{W}_{\psi})^{\lambda,\mu,-})^{\omega^{-j}}\ar[d]^{\omega^{-j}\circ\, \exp^*}\\
\Dcris(\mathscr{W}_{\psi})\ar[r]^{\frak{d}_{/\lambda,\mu,-}}&\Dcris(D_{\textup{rig}}^\dagger(\mathscr{W}_{\psi})^{\lambda,\mu,-})
}$$
together with the fact that $\BF_{\psi}^{\lambda,\mu,(j)}$ belongs to the kernel of the top horizontal arrow by \cite[Theorem 7.1.2]{LZ1}. We now prove property (ii) by arguing as in the proof of Theorem 8.2.1(v) in \cite{LZ1}. 

To ease notation, we let  $H^1_\Iw(X)$ (resp., $H^1(X)$) denote $H^1_\Iw(\QQ_p,X)$ (resp., $H^1(\QQ_p(\mu_p),X)$) in the following commutative diagram:
$$\xymatrixcolsep{1.4pc}\xymatrix{H^1_{\textup{Iw}}(\mathscr{W}_{\psi}\otimes\cH_{E,k+1}(\Gamma)^\iota)^{\omega^{-j}}\ar[r]^(0.55){\frak{d}_{/\lambda,\circ}}\ar[d]_{\frak{a}}&H^1_{\textup{Iw}}(D_{\textup{rig}}^\dagger(\mathscr{W}_{\psi})_{/\lambda,\circ})^{\omega^{-j}}\ar[d]_{\frak{a}}& H^1_{\textup{Iw}}(D_{\textup{rig}}^\dagger(\mathscr{W}_{\psi})_{/\alpha,\mu})^{\omega^{-j}}\ar@{_{(}->}[l] \ar[d]^{\frak{a}}\\
 H^1(\mathscr{W}_{\psi})^{\omega^{-j}}\ar[r]^(.42){\frak{d}_{/\lambda,\circ}}\ar[d]_{\omega^{-j}\circ\,\exp^*}&H^1(D_{\textup{rig}}^\dagger(\mathscr{W}_{\psi})_{/\lambda,\circ})^{\omega^{-j}}\ar[d]_{\omega^{-j}\circ\,\exp^*}&H^1(D_{\textup{rig}}^\dagger(\mathscr{W}_{\psi})_{/\lambda,\mu})^{\omega^{-j}}\ar@{_{(}->}[l]\ar[d]^{\omega^{-j}\circ\,\exp^*}\\
\Dcris(\mathscr{W}_{\psi})\ar[r]&\Dcris(\mathscr{W}_{\psi})_{/\lambda,\circ}&\Dcris(\mathscr{W}_{\psi})_{/\alpha,\mu}\ar@{_{(}->}[l]
}$$
We remind the reader that for the \'etale $(\varphi,\Gamma)$-module $D_{\textup{rig}}^\dagger(\mathscr{W}_{\psi})$, we have identified its cohomology with the cohomology of $\mathscr{W}_{\psi}$ in order to define the horizontal arrows on the left.
 
It follows from Theorem 7.1.2 in op.cit. that the image 
$$\frak{d}_{/\lambda,\circ}\circ\,\res_p\left(\BF_{\psi}^{\lambda,\mu,(j)}\right) \in H^1_{\textup{Iw}}(\QQ_p, D_{\textup{rig}}^\dagger(\mathscr{W}_{\psi})_{/\lambda,\circ})^{\omega^{-j}}$$ 
of the Iwasawa theoretic Beilinson--Flach class in fact falls in the image of 
$$H^1_{\textup{Iw}}(\QQ_{p},D_{\textup{rig}}^\dagger(\mathscr{W}_{\psi})_{/\lambda,\mu})^{\omega^{-j}}\hookrightarrow H^1_{\textup{Iw}}(\QQ_{p},D_{\textup{rig}}^\dagger(\mathscr{W}_{\psi})_{/\lambda,\circ})^{\omega^{-j}}{.}$$ 
We let $d_\Iw \in H^1_{\textup{Iw}}(\QQ_{p},D_{\textup{rig}}^\dagger(\mathscr{W}_{\psi})_{/\lambda,\mu})^{\omega^{-j}}$ denote the unique element that maps to $\frak{d}_{/\lambda,\circ}\circ\,\res_p(\BF_{\psi}^{\lambda,\mu,(j)})$. The commutative diagram above shows that 
$$\frak{d}_{/\lambda,\circ}\circ\,\res_p\left(\mathbf{bf}^{\lambda,\mu}\right)=\frak{a}(d_\Iw)$$ 
It therefore suffices to prove that $\frak{a}(d_\Iw)$ is non-trivial. Theorem 7.1.5 of op. cit. reduces this to verifying that $L_p(f,f,\omega^{-j},j) \neq 0$. Observe that in place of the variable $j$ in op. cit., we have used $j-1$ and furthermore, we have projected to the $\omega^{-j}$-eigenspaces. As a result, the relevant $p$-adic $L$-value in the current work is $L{_{p}}(f,f,,\omega^{-j},j)$ in place of $L_p(f,f,1+j)$ in op. cit.
{The desired non-vanishing of the $p$-adic $L$-value now follows from the hypothesis (\textbf{NV}). }
\end{proof}
\begin{remark}
\label{rem:candowithothertwists}
The attentive reader will realize that we could have in fact worked over the fields $\QQ(\mu_{p^n})$ in place of $\QQ(\mu_p)$ above and considered the $\omega^{-j}\theta$ invariants (for characters of $\Gamma_n:=\Gamma/\Gamma^{p^n}$, where $n$ is an arbitrary non-negative integer). The same proof would apply and prove for $\mu\in \{\alpha,-\alpha\}$ that
\[
\frak{d}_{/\lambda,\circ}\circ\omega^{-j}\theta\,\circ\, \exp^*\circ\,\res_p \left(\mathbf{bf}_{n}^{\lambda,\mu}\right)\in \left(\Dcris(\mathscr{W}_{\psi})_{/\lambda,\circ}\right)^{\varphi=-\frac{p^{j-1}\psi^{-1}(p)}{\lambda\mu}},
\]
where $\mathbf{bf}_{n}^{\lambda,\mu} \in H^1_{\mathcal{F}_{\textup{can}}}(\QQ(\mu_{p^n}),\mathscr{W}_{\psi})^{\omega^{-j}\theta}$ is the image of $\BF_{\psi}^{\lambda,\mu,(j)}$ and the morphism $\omega^{-j}\theta\,\circ\, \exp^*$ is given as in Definition~\ref{def:theexpstarevaluated}. This allows us to conclude that the classes $\res_{/\textup{f}}(\mathbf{bf}_{n}^{\lambda,\mu})$ and $\res_{/\textup{f}}(\mathbf{bf}_{n}^{\lambda,-\mu})$ are linearly independent.
\end{remark}
{The following corollary follows immediately  from Remark~\ref{rk:comapre_classes} and Theorem~\ref{thm_linearlyindependent}.}
\begin{corollary}
\label{cor_linearlyindependenttwistedBF}
The classes $c_1^{\lambda,\alpha}, c_1^{\lambda,-\alpha} \in H^1_{\mathcal{F}_{\textup{can}}}(\QQ,W_{\psi,j})$ are linearly independent.
\end{corollary}
\subsection{Analytic Selmer groups}
{Based on the local analysis in Section~\ref{subsec_Localanalysis}, we introduce in this subsection the $(\varphi,\Gamma)$-modules associated to the twists of the symmetric square representations, as well as their triangulations. With these objects in hand, we then define the Pottharst-style Selmer groups in our current setting. Proposition~\ref{prop_globallocalseqandcomparisonwithBK} below explains the relation of these  Selmer groups to Bloch--Kato Selmer groups as well as determines their size, relying chiefly on Theorem~\ref{thm_linearlyindependent}.}

Recall from \S\ref{subsec_Localanalysis} the decompositions
$$D_0\oplus D_1\oplus D_2=\Dcris(\mathscr{W}_{\psi})\supset \Dcris(\mathscr{V}_{\psi})=D_1\oplus D_2$$
of filtered $\varphi$-modules. We fix throughout this section a choice of $\lambda,\mu \in \{\pm\alpha\}$.
We define the $(\varphi,\Gamma)$-modules $\mathbb{D}_{\psi}^{+\circ}\subset\mathbb{D}_{\psi}^+$ by setting
\begin{align*}
\mathbb{D}_{\psi}^+&:=D_{\textup{rig}}^\dagger(\mathscr{W}_{\psi})^{\lambda,\mu,+}\cap D_{\textup{rig}}^\dagger(\mathscr{V}_{\psi})\,, \\
\mathbb{D}_{\psi}^{+\circ}&:=\mathbb{D}_\lambda\otimes\mathbb{D}_\lambda\otimes\mathbb{D}_{\psi,j}.
\end{align*}
\begin{lemma}
The crystalline $(\varphi,\Gamma)$-submodule $\mathbb{D}_{\psi}^+\subset D_{\textup{rig}}^\dagger(\mathscr{V}_{\psi})$ is a saturated  $(\varphi,\Gamma)$-submodule of rank $2$. Likewise, the submodule $\mathbb{D}_{\psi}^{+,\circ}\subset \mathbb{D}_{\psi}^+$ is  saturated of rank one.
\end{lemma}
\begin{proof}
Notice that 
$$\mathbb{D}_{\psi}^+=\ker\left(D_{\textup{rig}}^\dagger(\mathscr{V}_{\psi})\hookrightarrow D_{\textup{rig}}^\dagger(\mathscr{W}_{\psi}) \lra D_{\textup{rig}}^\dagger(\mathscr{W}_{\psi})^{\lambda,\mu,-}\right).$$
The first claim follows since 
$$e_\lambda\otimes e_{-\lambda}-e_{\mu}\otimes e_{-\mu} \not\in \ker\left(\mathbb{D}_f\otimes\mathbb{D}_f\lra \mathbb{D}_f/\mathbb{D}_\lambda \otimes \mathbb{D}_f/\mathbb{D}_\mu\right)$$
if $\lambda=-\mu$, whereas
$$e_\lambda\otimes e_{\lambda}-e_{-\lambda}\otimes e_{-\lambda} \not\in \ker\left(\mathbb{D}_f\otimes\mathbb{D}_f\lra \mathbb{D}_f/\mathbb{D}_\lambda \otimes \mathbb{D}_f/\mathbb{D}_\mu\right)$$
if $\lambda=\mu$, so that $\mathbb{D}_{\psi}^+\subsetneq D_{\textup{rig}}^\dagger(\mathscr{V}_{\psi})$. The second part follows from the exactness of the following sequence of $(\varphi,\Gamma)$-modules:
\begin{equation}
\label{en:themannerpluscircsitsinplus}
0\lra \mathbb{D}_{\psi}^{+\circ}\lra \mathbb{D}_{\psi}^+\lra \left(\mathbb{D}_f/\mathbb{D}_\lambda \otimes \mathbb{D}_f/\mathbb{D}_\lambda\right)\otimes \mathbb{D}_{\psi,j}\lra 0\,.
\end{equation}
\end{proof}
\begin{lemma}
\label{lemma:regularitystepone}
$\mathbb{D}_{\psi}^+\cap D_{\textup{rig}}^\dagger(D_1)=
\begin{cases} 
D_{\textup{rig}}^\dagger(D_1) & \hbox{ if } \mu = -\lambda\\
\,\,\,\,\,\,\,0 & \hbox{ if } \mu=\lambda
\end{cases}$.
\end{lemma}
\begin{proof}
We have the identification of $(\varphi,\Gamma)$-modules
$$D_{\textup{rig}}^\dagger(D_1)=\mathcal{R}_{E}(e_\alpha\otimes e_{\alpha}-e_{-\alpha}\otimes e_{-\alpha})\otimes \mathbb{D}_{\psi,j}.$$
 When $\lambda=\mu$, the conclusion follows from noting that 
$$e_\alpha\otimes e_\alpha-e_{-\alpha}\otimes e_{-\alpha} \not\in \ker\left(\mathbb{D}_f\otimes\mathbb{D}_f\lra \mathbb{D}_f/\mathbb{D}_\lambda \otimes \mathbb{D}_f/\mathbb{D}_\mu\right).$$
 When $\lambda=-\mu$, it follows from
$$e_\alpha\otimes e_\alpha-e_{-\alpha}\otimes e_{-\alpha} \in \ker\left(\mathbb{D}_f\otimes\mathbb{D}_f\lra \mathbb{D}_f/\mathbb{D}_\lambda \otimes \mathbb{D}_f/\mathbb{D}_\mu\right).$$
\end{proof}
\begin{corollary}
\label{cor_regularity_dcris}
$\Dcris(\mathbb{D}_{\psi}^{+\circ})\cap \textup{Fil}^0\Dcris(\mathscr{V}_{\psi})=0$.
\end{corollary}
\begin{proof}
As $\Dcris(\mathbb{D}_{\psi}^{+\circ})$ is one-dimensional over $E$, its intersection with $\textup{Fil}^0\Dcris(\mathscr{V}_{\psi})$
is either trivial or  $\Dcris(\mathbb{D}_{\psi}^{+\circ})$. Suppose that the latter holds. Then, $\Dcris(\mathbb{D}_{\psi}^{+\circ})$ is a  $\varphi$-stable subspace of $\subset \textup{Fil}^0\Dcris(\mathscr{V}_{\psi})$. But the unique $\varphi$-stable subspace of 
$$\textup{Fil}^0\Dcris(\mathscr{V}_{\psi})=D_1\oplus\textup{Fil}^0\Dcris(D_2)$$
is $D_1$. Thus, $\Dcris(\mathbb{D}_{\psi}^{+\circ})=D_1\,.$ When $\lambda=\mu$, this  contradicts Lemma~\ref{lemma:regularitystepone}. Therefore, the intersection is $0$ as required. When $\lambda=-\mu$, the same conclusion  follows from the fact that $D_{\textup{rig}}^\dagger(D_1)$ does not fit in the exact sequence  \eqref{en:themannerpluscircsitsinplus}). 
\end{proof}

Corollary~\ref{cor_regularity_dcris} tells us  that the submodule $\mathbb{D}_{\psi}^{+\circ}\subset D_{\textup{rig}}^\dagger(\mathscr{V}_{\psi})$ is regular in the sense of Benois and Perrin-Riou, c.f. \cite[\S2.1]{Benois2015} for an elaboration on this  property.

\begin{defn}
For  each natural number $n$ and  $\frak{D}=\mathbb{D}_{\psi}^{+\circ}$ or $\mathbb{D}_{\psi}^{+}$, we let $S^{\bullet}(\QQ(\mu_{p^n}),\mathscr{V}_{\psi},\frak{D})$ denote the $($analytic$)$ Selmer complex, given as in \cite[Definition 2.4]{benoisbuyukboduk} (with the base field taken as $\QQ(\mu_{p^n})$ in place of $\QQ$) and let $\mathbf{R}\Gamma(\QQ(\mu_{p^n}),\mathscr{V}_{\psi},\frak{D})$  denote the corresponding  class  in the derived category of $E$-vector spaces. 

We also define, following \cite[\S2.3]{Benois2015},  the Iwasawa theoretic $($analytic$)$ Selmer complex $S^{\bullet}_\Iw(\QQ(\mu_{p^n}),\mathscr{V}_{\psi},\frak{D})$ and the corresponding class $\mathbf{R}\Gamma_\Iw(\QQ(\mu_{p^n}),\mathscr{V}_{\psi},\frak{D})$ in the derived category of $\cH_E(\Gamma^{(n)})$-modules $($here, $\cH_E(\Gamma^{(n)})$ stands for $\cup_{r>0}\cH_{E,r}(\Gamma^{(n)})$ and $\Gamma^{(n)}:=\Gal(\QQ(\mu_{p^n})/\QQ))$. 

{For each natural number $n$} we set 
$$\widetilde{H}^i_?(\QQ(\mu_{p^n}),\mathscr{V}_{\psi},\frak{D}):=\mathbf{R}\Gamma_?^i(\QQ(\mu_{p^n}),\mathscr{V}_{\psi},\frak{D}) \hbox{ for } ?=\emptyset,\Iw$$ 
and call them analytic Selmer groups. 
\end{defn}
We have the following control theorem in the context of analytic Selmer complexes:
\begin{equation}
\label{eqn:controltheoremforSelmercomplex}
{\mathbf{R}\Gamma_\Iw(\QQ(\mu_{p^n}),\mathscr{V}_{\psi},\frak{D}) \otimes^{\mathbb{L}}_{\mathcal{H}_E(\Gamma^{(n)})}E \stackrel{\sim}{\lra} \mathbf{R}\Gamma(\QQ(\mu_{p^n}),\mathscr{V}_{\psi},\frak{D})\,,}
\end{equation}
{where $\otimes^{\mathbb{L}}$ denotes  the derived tensor product.} 

\begin{lemma}
\label{lem:structureoflocalcohomologygroups}
Let $D$ be a $(\varphi,\Gamma)$-module over $\mathcal{R}_{E}$ such that $H^0(\QQ_p,D)=H^2(\QQ_p,D)=0$. Then $H^1(\QQ_p,D)$ is an $E$-vector space of dimension $\textup{rank}_{\mathcal{R}_E}D$.
\end{lemma}
\begin{proof}
This is an immediate consequence of Liu's local Euler characteristic formula proved in \cite{LiuIMRN2008}. 
\end{proof}
\begin{defn}
For any Dirichlet character $\eta:G_\QQ\ra E^\times$, we let $E_\eta$ denote the one dimensional $E$-vector space on which $G_\QQ$ acts via $\eta$.
\end{defn}
\begin{proposition}
\label{prop_globallocalseqandcomparisonwithBK}
Fix a non-negative integer $n$ and let $\theta$ be a character of the quotient group $\Gamma_n$. Suppose at least one of the following two conditions holds:
\begin{itemize}
\item[a)] $\omega^{-j}\theta$ is non-trivial.
\item[b)] If $j=k+2$, then $\epsilon_f\psi^{-1}(p)\neq \pm 1$.
\end{itemize}
Then:
\item[i)] We have the following exact sequence:
\[0\lra \widetilde{H}^1(\QQ(\mu_{p^n}),\mathscr{V}_{\psi},\frak{D})^{\omega^{-j}\theta}\lra H^1_{\mathcal{F}_\textup{can}}(\QQ(\mu_{p^n}),\mathscr{V}_{\psi})^{\omega^{-j}\theta}
\lra {H^1(\QQ_p(\mu_{p^n}),D_{\rm rig}^{\dagger}(\mathscr{V}_{\psi})/\frak{D})}^{\omega^{-j}\theta}
\]
where $\frak{D}=\mathbb{D}_{\psi}^{+\circ}$ or $\mathbb{D}_{\psi}^{+}$.
\item[ii)]  The Selmer group $\widetilde{H}^1(\QQ(\mu_{p^n}),\mathscr{V}_{\psi},\mathbb{D}_{\psi}^{+\circ})^{\omega^{-j}\theta}$ is canonically isomorphic to the Bloch-Kato Selmer group $H^1_\textup{f}(\QQ(\mu_{p^n}),\mathscr{V}_{\psi})^{\omega^{-j}\theta}$, whereas  $\widetilde{H}^2(\QQ(\mu_{p^n}),\mathscr{V}_{\psi},\mathbb{D}_{\psi}^{+\circ})^{\omega^{-j}\theta}$ is isomorphic to $H^1_\textup{f}(\QQ(\mu_{p^n}),\mathscr{V}_{\psi}^*(1))^{\omega^{-j}\theta}$.
\item[iii)]  The Selmer group $H^1_\textup{f}(\QQ(\mu_{p^n}),\mathscr{V}_{\psi})^{\omega^{-j}\theta}$ is trivial, whereas $\widetilde{H}^1(\QQ(\mu_{p^n}),\mathscr{V}_{\psi},\mathbb{D}_{\psi}^{+})^{\omega^{-j}\theta}$ is one dimensional over $E$.
\end{proposition}

\begin{remark}
{We have checked in Lemma~\ref{lemma_propagate_hypo_psi_2} that Condition $($\textup{b}$)$ in the statement of Proposition~\ref{prop_globallocalseqandcomparisonwithBK} holds true when we assume the validity of \ref{item_Psi_2}. We recall that the hypothesis \ref{item_Psi_2} is required to avoid exceptional zeros, see Corollary~\ref{coro_trivialzeroes}.}
\end{remark}

\begin{proof}[Proof of Proposition~\ref{prop_globallocalseqandcomparisonwithBK}]
\item[i)] This portion follows from the definition of the Selmer complex as a mapping cone, once we verify that $H^0(\QQ_p(\mu_{p^n}), D_{\rm rig}^{\dagger}(\mathscr{V}_{\psi})/\frak{D})^{\omega^{-j}\theta}=0$. {When $\omega^{-j}\theta$ is non-trivial, it is a Dirichlet character ramified at $p$. Since the $G_{\QQ_p}$-representation $\mathscr{V}_\psi$ is crystalline, the desired vanishing of $H^0(\QQ_p(\mu_{p^n}), D_{\rm rig}^{\dagger}(\mathscr{V}_{\psi})/\frak{D})^{
\omega^{-j}\theta}$ follows.} Thus, we are reduced to checking that $H^0(\QQ_p, D_{\rm rig}^{\dagger}(\mathscr{V}_{\psi})/\frak{D})=0$ assuming (b). Note that if $H^0(\QQ_p,D_{\rm rig}^{\dagger}(\mathscr{V}_{\psi})/\frak{D})$ is non-zero, then we necessarily have $p^{j-1}\psi(p)=\pm \alpha^{2}$, which can only hold if $j=k+2$ and $\epsilon_f\psi^{-1}(p)=\pm 1$ (recall that $\alpha^{2}=-\epsilon_f(p)p^{k+1}$). Therefore,  the condition (b) implies the vanishing $H^0(\QQ_p(\mu_{p^n}), D_{\rm rig}^{\dagger}(\mathscr{V}_{\psi})/\frak{D})^{
\omega^{-j}\theta}=0$, as required.
 \item[ii)]{The first half of  this portion is immediate by \cite[Proposition 3.7(3)]{jayanalyticselmer}. The second assertion follows from  global duality (Theorem 1.15 of op.cit.). Note that the conditions of  \cite[Proposition 3.7(3)]{jayanalyticselmer} are valid thanks to Lemma~\ref{lemma:regularitystepone} and our hypotheses in this proposition. }
\item[iii)] {We follow the proof of \cite[Theorem 8.2.1]{LZ1} to prove the first assertion in (iii). We start off with the following exact sequence:
$$0\ra H^1_\textup{f}(\QQ(\mu_{p^n}),\mathscr{V}_{\psi})^{\omega^{-j}\theta} \lra H^1_{\mathcal{F}_\textup{can}}(\QQ(\mu_{p^n}),\mathscr{V}_{\psi})^{\omega^{-j}\theta}\stackrel{\res_{/\textup{f}}}{\lra} H^1_s(\QQ(\mu_{p^n}),\mathscr{V}_{\psi})^{\omega^{-j}\theta}\,.
$$
We have seen that the canonical Selmer group $H^1_{\mathcal{F}_\textup{can}}(\QQ(\mu_{p^n}),\mathscr{V}_{\psi})^{\omega^{-j}\theta}$ is of dimension $2$ and moreover, the image of $\res_{/\textup{f}}$ is also two-dimensional thanks to Theorem~\ref{thm_linearlyindependent} and Remark~\ref{rem:candowithothertwists}. This proves that $H^1_\textup{f}(\QQ(\mu_{p^n}),\mathscr{V}_{\psi})^{\omega^{-j}\theta} =0$.}

{We now prove the second assertion of (iii). On taking  $\frak{D}=\mathbb{D}^{+}_\psi$ in the exact sequence of part (i), we see that 
$$\dim \widetilde{H}^1(\QQ(\mu_{p^n}),\mathscr{V}_{\psi},\frak{D})^{\omega^{-j}\theta} +{H^1(\QQ_p(\mu_{p^n}),D_{\rm rig}^{\dagger}(\mathscr{V}_{\psi})/\frak{D})}^{\omega^{-j}\theta}\geq 
\dim H^1_{\mathcal{F}_\textup{can}}(\QQ(\mu_{p^n}),\mathscr{V}_{\psi})^{\omega^{-j}\theta}=2\,.$$
Lemma~\ref{lem:structureoflocalcohomologygroups} (applied with $D=\left(D_{\textup{rig}}^\dagger(\mathscr{V}_{\psi})/\mathbb{D}^+_\psi\right)\otimes D_{\textup{rig}}^\dagger(E_{\omega^j\theta^{-1}})$ and $K=\QQ_p$) tells us that $$\dim {H^1(\QQ_p(\mu_{p^n}),D_{\rm rig}^{\dagger}(\mathscr{V}_{\psi})/\frak{D})}^{\omega^{-j}\theta}=\dim H^1(\QQ_p,D)=1\,.$$ 
Therefore,
\begin{equation}\label{eq:ge1}
\dim\widetilde{H}^1(\QQ(\mu_{p^n}),\mathscr{V}_{\psi},\mathbb{D}_{\psi}^{+})^{\omega^{-j}\theta}\ge1.
\end{equation} 

Now, consider the exact sequence
\[
0\lra  \widetilde{H}^1(\QQ(\mu_{p^n}),\mathscr{V}_{\psi},\mathbb{D}_{\psi}^{+\circ})^{\omega^{-j}\theta}\lra \widetilde{H}^1(\QQ(\mu_{p^n}), \mathscr{V}_{\psi},\mathbb{D}_{\psi}^{+})^{\omega^{-j}\theta}
\lra H^1(\QQ_p(\mu_{p^n}),\mathbb{D}_{\psi}^{+}/\mathbb{D}_{\psi}^{+\circ})^{\omega^{-j}\theta}
\]
(which follows from the definitions of extended Selmer groups) shows that 
\begin{align*}
    \dim \widetilde{H}^1(\QQ(\mu_{p^n}), \mathscr{V}_{\psi},\mathbb{D}_{\psi}^{+})^{\omega^{-j}\theta} &\leq \dim \widetilde{H}^1(\QQ(\mu_{p^n}),\mathscr{V}_{\psi},\mathbb{D}_{\psi}^{+\circ})^{\omega^{-j}\theta} + H^1(\QQ_p(\mu_{p^n}),\mathbb{D}_{\psi}^{+}/\mathbb{D}_{\psi}^{+\circ})^{\omega^{-j}\theta}\\
    &=H^1(\QQ_p(\mu_{p^n}),\mathbb{D}_{\psi}^{+}/\mathbb{D}_{\psi}^{+\circ})^{\omega^{-j}\theta}\\
    &=1\,,
\end{align*}
where the first equality follows from (ii) combined with the vanishing of the Bloch-Kato Selmer group $H^1_\textup{f}(\QQ(\mu_{p^n}),\mathscr{V}_{\psi})^{\omega^{-j}\theta}$ (which is the first assertion of (iii)), and the second equality follows from  Lemma~\ref{lem:structureoflocalcohomologygroups} applied with $D=\left(\mathbb{D}_{\psi}^{+}/\mathbb{D}_{\psi}^{+\circ}\right)\otimes D_{\textup{rig}}^\dagger(E_{\omega^j\theta^{-1}})$. Combining this with \eqref{eq:ge1}, the second assertion of (iii) follows.}
\end{proof}
\begin{corollary}
\label{for:trivialityofthecorrectselmersubquotient}
In the setting of Proposition~\ref{prop_globallocalseqandcomparisonwithBK}, we have 
$$\widetilde{H}^2(\QQ(\mu_{p^n}),\mathscr{V}_{\psi},\mathbb{D}_{\psi}^{+})^{\omega^{-j}\theta}=0\,.$$ 
\end{corollary}
\begin{proof}
The definition of the Selmer complex as a mapping cone gives rise to a canonical surjection
$$\widetilde{H}^2(\QQ(\mu_{p^n}),\mathscr{V}_{\psi},\mathbb{D}_{\psi}^{+\circ})^{\omega^{-j}\theta}\twoheadrightarrow \widetilde{H}^2(\QQ(\mu_{p^n}),\mathscr{V}_{\psi},\mathbb{D}_{\psi}^{+})^{\omega^{-j}\theta}$$
and the corollary follows from Proposition~\ref{prop_globallocalseqandcomparisonwithBK}(ii)-(iii).
\end{proof}
\subsection{Zeros of characteristic ideals}
\label{subsec_controltheanalyticselmercharideals}
{The main goal of this subsection is to determine a locus where the generators of the characteristic ideals of the Pottharst-style Iwasawa theoretic Selmer groups do not have a zero. In particular, we prove Theorem~\ref{thm_controlthecharidealofanalyticselmer}. Corollary~\ref{cor_nondivisibilityofthecorrecttwist} to this theorem plays a fundamental role in the proof of Theorem~\ref{thm_positionofBFintheanalyticselmer}.}

From now on, we assume that $\epsilon_f\psi^{-1}(p)\neq \pm 1$. {We recall from Lemma~\ref{lemma_propagate_hypo_psi_2} that this assumption holds for all $\psi=\chi\nu$ (where $\nu$ is Dirichlet character of $p$-power order and prime-to-$p$ conductor) under the hypothesis \ref{item_Psi_2}.}
Throughout this section, we continue to work with our fixed choice of  $\lambda,\mu\in \{\pm\alpha\}$. 
\begin{lemma}
\label{lem:structureoflocalcohomologygroupsIwasawa}
Let $D$ be a $(\varphi,\Gamma)$-module of rank $d$ over $\mathcal{R}_{E}$ such that $H^0(\QQ_p(\mu_{p^\infty}),D)=H^2(\QQ_p(\mu_{p^\infty}),D)=0$. Then $H^1_\Iw(\QQ_p,D)$ is a projective {$\mathcal{H}_E(\Gamma)$}-module of rank $d$.
\end{lemma}
\begin{proof}
Only in this proof, we let $\psi$ denote the  left inverse for the Frobenius operator $\varphi$, and not a Dirichlet character $\psi$ that is unramified at $p$. The proof of the lemma follows from the fact that the complex $C^{\bullet}_\psi(D)$ is a perfect complex of $\mathcal{H}_E(\Gamma)$-modules, which may be represented by a single projective module concentrated in degree $1$ thanks to our running hypotheses.
\end{proof}
{We recall that $\cH:=\cup_{r>0} \cH_{E,r}(\Gamma_1)$.}
\begin{proposition}
\label{prop_saturated_sub_selmer}
Let $n$ be a positive integer and $\theta$ a character of $\Gamma_n$. Then $\widetilde{H}^1_\Iw(\QQ(\mu_{p^n}),\mathscr{V}_{\psi},\mathbb{D}_{\psi}^{+})^{\omega^{-j}\theta}$ is a saturated rank one  $\cH$-submodule of 
$$H^1_\Iw(\QQ(\mu_{p^n}),\mathscr{V}_{\psi})^{\omega^{-j}\theta}\otimes_{ \Lambda_\cO(\Gamma_1)} \cH \cong H^1_\Iw(\QQ,V_{\psi,j}\otimes\theta^{-1})\otimes_{ \Lambda_\cO(\Gamma_1)}\cH\,.$$ 
\end{proposition}
We remark that the isomorphism $H^1_\Iw(\QQ(\mu_{p^n}),\mathscr{V}_{\psi})^{\omega^{-j}\theta}\otimes \cH \cong H^1_\Iw(\QQ,V_{\psi,j}\otimes\theta^{-1})\otimes\cH$ {follows as a consequence of a formal twisting argument (c.f. \cite{rubin00}, \S6).}
\begin{proof}[Proof of Proposition~\ref{prop_saturated_sub_selmer}]
The definition of the Selmer complex as a mapping cone (and the fact that Iwasawa cohomology classes are unramified) yields the exact sequence
\begin{align*}
0\lra \widetilde{H}^1_\Iw(&\QQ(\mu_{p^n}), \mathscr{V}_{\psi},\mathbb{D}_{\psi}^{+})^{\omega^{-j}\theta} \lra H^1_\Iw(\QQ(\mu_{p^n}),\mathscr{V}_{\psi})^{\omega^{-j}\theta}\otimes \cH  \\
&\lra H^1_\Iw(\QQ_p(\mu_{p^n}),D_{\textup{rig}}^\dagger(\mathscr{V}_{\psi})/\mathbb{D}_{\psi}^{+})^{\omega^{-j}\theta} \lra \widetilde{H}^2_\Iw(\QQ(\mu_{p^n}),\mathscr{V}_{\psi},\mathbb{D}_{\psi}^{+})^{\omega^{-j}\theta}\,.
\end{align*}
Since the $\cH$-module 
$$H^1_\Iw(\QQ(\mu_{p^n}),\mathscr{V}_{\psi})^{\omega^{-j}\theta}\otimes \cH \cong H^1_\Iw(\QQ,V_{\psi,j}\otimes\theta^{-1})\otimes\cH$$
is free of rank $2$ thanks to Theorem~\ref{thm_mainSelmerstructure}, it suffices to verify that
\begin{itemize}
\item[(i)] The $\cH$-module 
$$H^1_\Iw(\QQ_p(\mu_{p^n}),D_{\textup{rig}}^\dagger(\mathscr{V}_{\psi})/\mathbb{D}_{\psi}^{+})^{\omega^{-j}\theta}\cong  H^1_\Iw\left(\QQ_p,(D_{\textup{rig}}^\dagger(\mathscr{V}_{\psi})/\mathbb{D}_{\psi}^{+})\otimes D_{\textup{rig}}^\dagger(E_{\omega^j\theta^{-1}})\right)$$
is projective of rank one, where the isomorphism follows from the version of Shapiro's Lemma in \cite[Lemma 2.3.5]{KPX} in the context of $(\varphi,\Gamma)$-modules and their cohomology;\\
\item[(ii)] The $\cH$-module  $ \widetilde{H}^2_\Iw(\QQ(\mu_{p^n}),\mathscr{V}_{\psi},\mathbb{D}_{\psi}^{+})^{\omega^{-j}\theta}$ is torsion.
\end{itemize}
The assertion (i) follows from Lemma~\ref{lem:structureoflocalcohomologygroupsIwasawa} thanks to our running hypotheses on $\psi$, so it remains to verify the assertion (ii). To do so, we first consider the exact sequence
\begin{align}
\notag 0\lra  \widetilde{H}^1(\QQ&(\mu_{p^n}),\mathscr{V}_{\psi},\mathbb{D}_{\psi}^{+\circ})^{\omega^{-j}\theta}\lra \widetilde{H}^1(\QQ(\mu_{p^n}),\mathscr{V}_{\psi},\mathbb{D}_{\psi}^{+})^{\omega^{-j}\theta}\lra H^1(\QQ_p(\mu_{p^n}),\mathbb{D}_{\psi}^{+}/\mathbb{D}_{\psi}^{+\circ})^{\omega^{-j}\theta}
\\
\label{eqn_exact_seq_new_Prop542}&\lra \widetilde{H}^2(\QQ(\mu_{p^n}),\mathscr{V}_{\psi},\mathbb{D}_{\psi}^{+\circ})^{\omega^{-j}\theta}\lra \widetilde{H}^2(\QQ(\mu_{p^n}),\mathscr{V}_{\psi},\mathbb{D}_{\psi}^{+})^{\omega^{-j}\theta}\lra 0
\end{align}
of $E$-vector spaces. By Proposition~\ref{prop_globallocalseqandcomparisonwithBK}(ii)-(iii) and global duality, it follows that 
$$\widetilde{H}^2(\QQ(\mu_{p^n}),\mathscr{V}_{\psi},\mathbb{D}_{\psi}^{+\circ})^{\omega^{-j}\theta}\cong H^1_{\textup{f}}(\QQ(\mu_{p^n}),\mathscr{V}_{\psi}^*(1))^{\omega^{-j}\theta}=0.$$ Hence, we deduce from \eqref{eqn_exact_seq_new_Prop542} that 
$$\widetilde{H}^2(\QQ(\mu_{p^n}),\mathscr{V}_{\psi},\mathbb{D}_{\psi}^{+})^{\omega^{-j}\theta}=0\,.$$

The control theorem for Selmer complexes (\ref{eqn:controltheoremforSelmercomplex}) yields an injection 
$$\left(\widetilde{H}^2_\Iw(\QQ(\mu_{p^n}),\mathscr{V}_{\psi},\mathbb{D}_{\psi}^{+})^{\omega^{-j}\theta}\right)_{{\Gamma^{(n)}}}\hookrightarrow \widetilde{H}^2(\QQ(\mu_{p^n}),\mathscr{V}_{\psi},\mathbb{D}_{\psi}^{+})^{\omega^{-j}\theta},$$
which shows {(thanks to the structure theory of coadmissible $\mathcal{H}$-modules {in the sense of \cite{ST}}, given as in \cite[Proposition 3.6]{Benois2015})} that $\widetilde{H}^2_\Iw(\QQ(\mu_{p^n}),\mathscr{V}_{\psi},\mathbb{D}_{\psi}^{+})^{\omega^{-j}\theta}$ does not have positive $\mathcal{H}$-rank and this concludes our proof.
\end{proof}
{If $M$ is a coadmissible torsion $\cH$-module in the sense of \cite{ST}, we may define its characteristic ideal $\Char_\cH(M)$ as in \cite[\S7.2.1]{HP}. Pottharst has shown that the torsion $\cH$-module $\widetilde{H}^2_\Iw(\QQ(\mu_{p}),\mathscr{V}_{\psi},\mathbb{D}_{\psi}^{+})^{\omega^{-j}}$ is coadmissible (see \cite[Proposition 3.10]{Benois2015}) and we shall study its characteristic ideal.} Recall  that $\chi_\cyc$ denotes the cyclotomic character on $\Gamma$, the $p^n$-cyclotomic polynomial is denoted by $\Phi_{p^n}$  and $\Tw_n$ is a twisting map defined as in \eqref{eq:twist}. We define $\Tw_{\langle n \rangle}$ to be the twisting map that acts as $\Tw_n$ on elements in $\Gamma_1$ and the identity on $\Gamma_{\mathrm{tors}}$.
\begin{theorem}
\label{thm_controlthecharidealofanalyticselmer}
For any $m \in [k+1,2k+1]$ and any positive integer $n$, the element $\textup{Tw}_{j-m-1} \Phi_{p^{n}}(\gamma)$ does not divide in $\cH$ the ideal $\Char_\cH\left(\widetilde{H}^2_\Iw(\QQ(\mu_{p}),\mathscr{V}_{\psi},\mathbb{D}_{\psi}^{+})^{\omega^{-j}}\right)$. 
\end{theorem}

\begin{proof}
In order to ease our notation, we shall adopt the following convention in this proof.
\\\\
\emph{Convention}. Let $X$ be a twist of $\mathscr{V}_{\psi}$ by a power of $\chi_\cyc$. We denote the $(\varphi,\Gamma)$-submodule of $D_{\textup{rig}}^\dagger(X)$ that gives rise as a twist of $\mathbb{D}_\psi^+$ also by $\mathbb{D}_\psi^+$.

It follows by the twisting formalism (c.f., \cite[\S 6]{rubin00}) that 
$$\Char_\cH\left(\widetilde{H}^2_\Iw(\QQ(\mu_{p}),\mathscr{V}_{\psi},\mathbb{D}_{\psi}^{+})^{\omega^{-j}}\right)=\textup{Tw}_{\langle j-1\rangle}\Char_\cH\left(\widetilde{H}^2_\Iw(\QQ(\mu_{p}),\mathscr{V}_{\psi}(j-1),\mathbb{D}_{\psi}^{+})^{\omega^{-1}}\right).$$
Our assertion is therefore equivalent to the claim that $\Char_\cH\widetilde{H}^2_\Iw(\QQ(\mu_{p}),\mathscr{V}_{\psi}(j-1),\mathbb{D}_{\psi}^{+})^{\omega^{-1}}$ is not divisible by any linear factor of the product
$$\Phi_{p^{n}}(\chi_\cyc^{-m}(\gamma)\gamma):=\prod_{\eta}\left(\chi_\cyc^{-m}(\gamma)\eta^{-1}(\gamma)\gamma-1\right)$$
for any $r\in [k+1,2k+2]$ and any positive integer $n$, where $\eta$ runs through primitive characters of $\Gamma_{n-1}$.

Assume the contrary, so that there exists a positive integer $n$ and an $E$-valued primitive character $\theta$ of $\Gamma_{n-1}$ (after enlarging $E$ if necessary) such that 
\begin{equation}
\label{eqn_div_1_543}
(\chi_\cyc^{-m}(\gamma)\theta^{-1}(\gamma)\gamma-1)\,\, \big{|}\,\, \Char_\cH\left(\widetilde{H}^2_\Iw(\QQ(\mu_{p}),\mathscr{V}_{\psi}(j-1),\mathbb{D}_{\psi}^{+})^{\omega^{-1}}\right)
\end{equation} 
for some $m \in [k+1,2k+1]$. 

Set $\mathscr{V}_m:=\mathscr{V}_{\psi}(j-m-1)\cong \textup{Sym}^2W_f^*(-m)\otimes\psi$. Since 
$$\textup{Tw}_{\langle m \rangle}\left(\chi_\cyc^{-m}(\gamma)\theta^{-1}(\gamma)\gamma-1\right)=\theta^{-1}(\gamma)\gamma-1,$$ 
observe that the divisibility \eqref{eqn_div_1_543} is equivalent to the divisibility
\begin{align*}
\left(\theta^{-1}(\gamma)\gamma-1\right)\,\, &\big{|}\,\, \textup{Tw}_{\langle m \rangle}\left(\Char_\cH\,\widetilde{H}^2_\Iw(\QQ(\mu_{p}),\mathscr{V}_{\psi}(j-1),\mathbb{D}_{\psi}^{+})^{\omega^{-1}}\right)\\
&=\Char_\cH\left(\widetilde{H}^2_\Iw(\QQ(\mu_{p}),\mathscr{V}_{\psi}(j-1),\mathbb{D}_{\psi}^{+})^{\omega^{-1}}\otimes \langle \chi_\cyc\rangle^{-m}\right)\\
&=\Char_\cH\left(\widetilde{H}^2_\Iw(\QQ(\mu_{p}),\mathscr{V}_{m},\mathbb{D}_{\psi}^{+})^{\omega^{-1-m}}\right)\,,
\end{align*}
where the equality on the second line is immediate from the definition of ${\rm Tw}_{\langle m \rangle}$ and the one on the third line is a direct consequence of the twisting formalism. By the structure theory of admissible $\cH$-modules, this is equivalent to the requirement that 
the quotient 
$$\widetilde{H}^2_\Iw(\QQ(\mu_{p}),\mathscr{V}_m,\mathbb{D}_{\psi}^{+})^{\omega^{-1-m}}\big{/}\left(\theta^{-1}(\gamma)\gamma-1\right)\widetilde{H}^2_\Iw(\QQ(\mu_{p}),\mathscr{V}_m,\mathbb{D}_{\psi}^{+})^{\omega^{-1-m}}$$
is a non-trivial $E$-vector space. Since the Iwasawa theoretic Selmer complex has no cohomology in degree $3$, we have the following canonical isomorphism thanks to the control theorem for Selmer complexes:
\begin{equation}
\label{en:firststepinreductionofiwasawaselmer}
\widetilde{H}^2_\Iw(\QQ(\mu_{p}),\mathscr{V}_m,\mathbb{D}_{\psi}^{+})^{\omega^{-1-m}}\Big{/}\left(\gamma^{p^{n-1}}-1\right)\widetilde{H}^2_\Iw(\QQ(\mu_{p}),\mathscr{V}_m,\mathbb{D}_{\psi}^{+})^{\omega^{-1-m}}\stackrel{\sim}{\lra} \widetilde{H}^2(\QQ(\mu_{p^{n}}),\mathscr{V}_m,\mathbb{D}_{\psi}^{+})^{\omega^{-1-m}}\,.
\end{equation}

Furthermore, since the element $\gamma^{p^{n-1}}-1$ belongs to the ideal of $\cH$ generated by $\theta^{-1}(\gamma)\gamma-1$, the natural surjection 
$$\widetilde{H}^2_\Iw(\QQ(\mu_{p}),\mathscr{V}_m,\mathbb{D}_{\psi}^{+})^{\omega^{-1-m}}\lra \widetilde{H}^2_\Iw(\QQ(\mu_{p}),\mathscr{V}_m,\mathbb{D}_{\psi}^{+})^{\omega^{-1-m}}\Big{/}\left(\theta^{-1}(\gamma)\gamma-1\right)\widetilde{H}^2_\Iw(\QQ(\mu_{p}),\mathscr{V}_m,\mathbb{D}_{\psi}^{+})^{\omega^{-1-m}}$$
factors as 
$$\xymatrixcolsep{1pc}\xymatrixrowsep{1pc}\xymatrix{\widetilde{H}^2_\Iw(\QQ(\mu_{p}),\mathscr{V}_m,\mathbb{D}_{\psi}^{+})^{\omega^{-1-m}}\ar@{->>}[rr]\ar[rd]_(.42){\eqref{en:firststepinreductionofiwasawaselmer}}&&\frac{\widetilde{H}^2_\Iw(\QQ(\mu_{p}),\mathscr{V}_m,\mathbb{D}_{\psi}^{+})^{\omega^{-1-m}}}{\left(\theta^{-1}(\gamma)\gamma-1\right)}.\\
&\frac{\widetilde{H}^2(\QQ(\mu_{p^{n}}),\mathscr{V}_m,\mathbb{D}_{\psi}^{+})^{\omega^{-1-m}}}{\left(\theta^{-1}(\gamma)\gamma-1\right)}\ar[ur]&
}$$
This shows that the finite-dimensional $E$-vector space ${\widetilde{H}^2(\QQ(\mu_{p^{n}}),\mathscr{V}_m,\mathbb{D}_{\psi}^{+})^{\omega^{-1-m}}}\big{/}{\left(\theta^{-1}(\gamma)\gamma-1\right)}$ is non-zero. On the other hand, the exactness of the sequence
\begin{align*}
0\lra \widetilde{H}^2(\QQ(\mu_{p^{n}}),\mathscr{V}_m,\mathbb{D}_{\psi}^{+})^{\omega^{-1-m}\theta}&\lra \widetilde{H}^2(\QQ(\mu_{p^{n}}),\mathscr{V}_m,\mathbb{D}_{\psi}^{+})^{\omega^{-1-m}}\xrightarrow{\times (\theta^{-1}(\gamma)-1)}\\
&\widetilde{H}^2(\QQ(\mu_{p^{n}}),\mathscr{V}_m,\mathbb{D}_{\psi}^{+})^{\omega^{-1-m}} \lra {\widetilde{H}^2(\QQ(\mu_{p^{n}}),\mathscr{V}_m,\mathbb{D}_{\psi}^{+})^{\omega^{-1-m}}}\big{/}{\left(\theta^{-1}(\gamma)\gamma-1\right)}
\end{align*}
shows that $\widetilde{H}^2(\QQ(\mu_{p^{n}}),\mathscr{V}_m,\mathbb{D}_{\psi}^{+})^{\omega^{-1-m}\theta}$ is non-trivial as well. This contradicts Corollary~\ref{for:trivialityofthecorrectselmersubquotient}, which we apply with the choice $j=m+1$. This shows that the divisibility \eqref{eqn_div_1_543} is false and completes the proof.
\end{proof}
Recall that we have set $V_\psi:=\textup{Sym}^2W_f^*(1+\psi)$.

\begin{corollary}
\label{cor_nondivisibilityofthecorrecttwist}
The characteristic ideal of the $\cH$-module $\widetilde{H}^2_\Iw(\QQ,V_\psi,\mathbb{D}_{\psi}^{+})$ is prime to ${\log_{2k+3}^{(1)}/\log_{k+2}^{(1)}}$\,.
\end{corollary}
\begin{proof}
Note that we have
\begin{align*}\Char_\cH\,\widetilde{H}^2_\Iw(\QQ,V_\psi)&=\textup{Tw}_{\langle -j\rangle}\Char_\cH\widetilde{H}^2_\Iw(\QQ,V_\psi\otimes\langle \chi_\cyc\rangle^{-j})\\
&=\textup{Tw}_{\langle-j\rangle}\Char_\cH\widetilde{H}^2_\Iw(\QQ,\mathscr{V}_{\psi}\otimes \omega^{j})\\
&=\textup{Tw}_{\langle-j\rangle}\Char_\cH\widetilde{H}^2_\Iw(\QQ(\mu_p),\mathscr{V}_{\psi})^{\omega^{-j}},
\end{align*}
where the final isomorphism is deduced from the version of Shapiro's Lemma in \cite[Lemma 2.3.5]{KPX} in the context of $(\varphi,\Gamma)$-modules and their cohomology. The assertion in the corollary follows from Theorem~\ref{thm_controlthecharidealofanalyticselmer}.
\end{proof}
Recall the positive integer $r\in \mathcal{R}_\chi$ and $\eta\in \widehat{\Delta}_r$ we have fixed at the start of this section (so that $\psi=\chi\eta$). 
\begin{defn}
\label{def:thedenominators}
We denote by $h_\eta^{\lambda,\mu}\in \Char_\cH\left(\widetilde{H}^2_\Iw(\QQ,V_\psi,\mathbb{D}^+_\psi)\right)$ any fixed generator. We also set 
$$h^{\lambda,\mu}:=\sum_{\eta\in \widehat{\Delta}_r}e_{\eta}h_{\eta^{-1}}^{\lambda,\mu} \in \cH[\Delta_r],$$ 
where $e_{\eta}$ is the idempotent associated to $\eta$. 
\end{defn}
\subsection{Proofs of Theorem~\ref{thm_positionofBFintheanalyticselmer} and Corollary~\ref{cor:signedBFinsymmsquare}} \label{sec:proofoftheorem}
\label{subsec:theproofoftheorem39}
{Before we go into the technical details, we outline the key ideas in the proofs of these two results. As explained in \cite[Proposition 5.3.2]{BLLV}, the proof of the factorization in Corollary~\ref{cor_BFfactorisation} would have been straightforward if the Beilinson--Flach elements $\BF_{r,\chi}^{\lambda,\mu}$ belonged to the image of the Perrin-Riou projectors ${\rm pr}_{\lambda,\mu}$ (which were introduced in Definition~\ref{define_projectionviaotsukifunctionals}). We unfortunately do not know if that is indeed the case. However, Proposition~\ref{prop_imageofperrinrioumapontheglobalcohomology} quantifies the potential failure of this property, in terms of the characteristic ideal of a certain Pottharst-style Selmer group, which we have already studied in Section~\ref{subsec_controltheanalyticselmercharideals}. Using our result on the support of this ideal  (Corollary~\ref{cor_nondivisibilityofthecorrecttwist} above), we can then define the sought after multiplier $c_m \in  \QQ_p[\Delta_r]\otimes  \textup{Frac}(\cH)$ as in Definition~\ref{def_multiplier_c_433}, whose denominator is a generator of the characteristic ideal of the said Pottharst-style Selmer group.}

Recall that for our fixed Dirichlet character $\chi$, we have set $W:=W_{f}^*\otimes W_{f}^*(1+\chi)$, $V:=\Sym^2W_f^*(1+\chi)$ and $V_\psi:=\textup{Sym}^2W_f^*(1+\psi)$.
We  remark that all the Hodge-Tate weights of $V$ are positive. This fact is crucial for our purposes. 
\begin{defn}
Let $r\in \mathcal{R}_\chi$ and let $\eta\in \widehat{\Delta}_r$ be a character. We denote by $\{F_{r,i}\}_i$ the set of completions of $\QQ(r)$ at primes above $p$.  We let 
$$\mathscr{L}_{\lambda,\mu,r}^{(1)}: H^1_\Iw(\QQ(r)_p,V)\lra \QQ(r)\otimes\cH\otimes \Dcris(V)$$
denote the Perrin-Riou map whose restriction to $H^1_\Iw(F_{r,i},V_\psi)$ is the corresponding twist of the morphism $\cL_{\lambda,\mu,F_{r,i}}$  introduced in Definition~\ref{def:thePRmapsforthesymmetricproduct}. For each $\eta\in \widehat{\Delta}_r$ and for $\psi=\chi\eta$ we write
$$\mathscr{L}_{\lambda,\mu,\eta}^{(1)}: H^1(\QQ_p,V_\psi)\lra \cH\otimes \Dcris(V_\psi)$$
for the $\eta^{-1}$-component of $\mathscr{L}_{\lambda,\mu,r}^{(1)}$.
\end{defn}

As in \S\ref{S:signedBF}, we have the Beilinson--Flach element
\begin{equation}\label{eq:BFelementsatlevelrtwistedbycyclochar}
\BF_{r,\chi}^{\lambda,\mu}\in H^1_\Iw(\QQ(r),W)\otimes\cH\stackrel{(\ast)}{=}H^1_\Iw(\QQ(r),V)\otimes\cH
\end{equation}
for each positive integer $r\in \mathcal{R}_\chi$.

\begin{remark}
\label{rem:sym2orsymmetricproductdoesntmatter}
The equality $(\ast)$ in \eqref{eq:BFelementsatlevelrtwistedbycyclochar} follows from Corollary~\ref{cor:vanishingofX} applied with $\psi=\eta\chi$ (where $\eta$ runs through characters of $\Delta_r$) and twisted by the character $\omega^{-j}\chi_\cyc^{j}$ of $\Gamma$. We also have 
$$\mathscr{L}^{(1)}_{\lambda,\mu,r}\Big{|}_{H^1_\Iw(\QQ(r)_p,V)}=\mathscr{L}^{(1)}_{\mu,\lambda,r}\Big{|}_{H^1_\Iw(\QQ(r)_p,V)}$$
for the restriction of the Perrin-Riou maps to the semi-local cohomology for the symmetric square. This combined with $(\ast)$ in turn implies that
$$\mathscr{L}^{(1)}_{\lambda,\mu,r}\circ\, \res_p = \mathscr{L}^{(1)}_{\mu,\lambda,r}\circ\, \res_p\,.$$
Moreover, it follows from Proposition~\ref{prop:thedichotomy} combined with Corollary~\ref{cor:vanishingofX} (which amounts to the vanishing for the Iwasawa cohomology for the odd twists of the alternating square) that 
$$\BF_{r,\chi}^{\lambda,\mu}=\BF_{r,\chi}^{\mu,\lambda}\,.$$
Based on these remarks, we may easily go back and forth between the cohomological  invariants of $V$ and $W$.
\end{remark}
We recall from Definition~\ref{define_projectionviaotsukifunctionals} the projectors
$$\xymatrix{\bigwedge^2 H^1_\Iw(\QQ(r),W)\ar[r]^(.55){\pr_{{\lambda},{\mu}}}\ar@{=}[d]& H^1(\QQ(r,W)\ar@{=}[d]\\
\bigwedge^2 H^1_\Iw(\QQ(r),V)\ar[r]&H^1(\QQ(r,V)
}$$
\begin{defn}
For $h_{\eta}^{\lambda,\mu}$ as in Definition~\ref{def:thedenominators}, we set  $h_\eta:=\prod_{\lambda,\mu}h_\eta^{\lambda,\mu} \in \cH$ and define 
$$h_r:=\sum_{\eta\in \widehat{\Delta}_r}e_\eta h_{\eta^{-1}} \in \cH[\Delta_r].$$
\end{defn}
Note that $h_\eta$ and $\log_{2k+2}^{(1)}/\log_{k+1}^{(1)}$ have no common factor thanks to Corollary~\ref{cor_nondivisibilityofthecorrecttwist}.
\begin{proposition}
\label{prop_imageofperrinrioumapontheglobalcohomology}
For any choice of $\lambda,\mu\in \{\alpha,-\alpha\}$, we have  
$$h_r\cH[\Delta_r]\subset \mathscr{L}^{(1)}_{\lambda,\mu,r}\circ\, \res_p(H^1_\Iw(\QQ(r),V))$$ 
for the image of $H^1_\Iw(\QQ,V)$ under the Perrin-Riou map.
\end{proposition}
\begin{proof}
It suffices to prove this for each isotypic component. Namely, once we verify that 
$$h_\eta\cH\subset \mathscr{L}^{(1)}_{\lambda,\mu,\eta}\circ\, \res_p(H^1_\Iw(\QQ,V_\psi))$$
for each character $\eta \in \widehat{\Delta}_r$ with $\psi=\chi\eta$, the proof will follow. 

By the definition of the Selmer complex as a mapping cone, we have the following exact sequence:
\begin{equation}
\label{eqn:theselmercomlexseqforPRmaps}
0\lra \frac{H^1_\Iw(\QQ,V_\psi)\otimes\cH}{\widetilde{H}^1_\Iw(\QQ,V_\psi,\mathbb{D}^+_\psi)}\stackrel{\res_p}{\lra} H^1_\Iw(\QQ_p,D_{\textup{rig}}^\dagger(V_\psi)/\mathbb{D}_\psi^+)\lra \widetilde{H}^2_\Iw(\QQ,V_\psi,\mathbb{D}^+_\psi)
\end{equation}
where we recall our convention that for twists $\mathscr{V}$ of $\mathscr{V}_{\psi}$ by a character of $\Gamma$ (such as our representation $V_\psi$ here), we denote the $(\varphi,\Gamma)$-submodule of $D_{\textup{rig}}^\dagger(\mathscr{V})$ corresponding to $\mathbb{D}_\psi^+$ also by $\mathbb{D}_\psi^+$. Recall also that the $(\varphi,\Gamma)$-submodule $\mathbb{D}_\psi^+$ depends on our choice of the pair $\lambda,\mu$. Observe further that the map $ \mathscr{L}^{(1)}_{\lambda,\mu,\eta}$ factors as
\begin{equation}
\label{eqn:factorthePRmap}
\xymatrixcolsep{1pc}\xymatrixrowsep{1pc}\xymatrix{H^1_\Iw(\QQ_p,V_\psi)\otimes\cH\ar[rd]\ar[rr]^(.6){ \mathscr{L}^{(1)}_{\lambda,\mu,\eta}}&& \cH\\
&H^1_\Iw(\QQ_p,D_{\textup{rig}}^\dagger(V_\psi)/\mathbb{D}_\psi^+)\ar[ur]&
}\end{equation}
by its very definition. As the Perrin-Riou map
$$ \mathscr{L}^{(1)}_{\lambda,\mu,\eta}: H^1_\Iw(\QQ_p,V_\psi)\otimes\cH\lra \cH$$
is surjective, the proof follows from the exact sequence (\ref{eqn:theselmercomlexseqforPRmaps}) and the choice of $h_\eta$.
\end{proof}
\begin{theorem}
\label{thm_analyticselmervsrankreductionmap}
$h_r\BF^{\lambda,\mu}_{r,\chi} \in \pr_{\lambda,\mu}\left(H^1_\Iw(\QQ(r),V)\otimes\cH\right)$\,.
\end{theorem}
\begin{proof}
We may once again prove this one character at a time: For each $\eta\in \widehat{\Delta}_r$ and $\psi=\chi\eta$, we shall verify that 
$$h_\eta\BF^{\lambda,\mu}_{\eta}\in \pr_{\lambda,\mu}\left(H^1_\Iw(\QQ,V_\psi)\otimes\cH\right)\,.$$
Here, $\BF^{\lambda,\mu}_{\eta} \in H^1_\Iw(\QQ,V_\psi)\otimes\cH$ is the image of the class $\BF^{\lambda,\mu}_{r,\chi}$, on projection to the $\eta^{-1}$-isotypic component.

Recall that the $\cH$-module $H^1_\Iw(\QQ,V_\psi)\otimes\cH$ is free of rank $2$ thanks to Theorem~\ref{thm_mainSelmerstructure}. We fix a basis $\{\mathcal{Y}_1,\mathcal{Y}_2\}$ of this module and observe that 
$$ \pr_{\lambda,\mu}\left(H^1_\Iw(\QQ(r),V)\otimes\cH\right)=\textup{span}_\cH{}\left(\mathscr{L}^{(1)}_{\lambda,\mu,\eta}(\mathcal{Y}_1)\mathcal{Y}_2-\mathscr{L}^{(1)}_{\lambda,\mu,\eta}(\mathcal{Y}_2)\mathcal{Y}_1\right)\,.$$

The fact that 
$${H^1_\Iw(\QQ,V_\psi)\otimes\cH}\big{/}{\widetilde{H}^1_\Iw(\QQ,V_\psi,\mathbb{D}^+_\psi)}\hookrightarrow H^1_\Iw(\QQ_p,D_{\textup{rig}}^\dagger(V_\psi)/\mathbb{D}_\psi^+)$$
is torsion free implies that
$$\cH\cdot\mathcal{Y}_i\cap {\widetilde{H}^1_\Iw(\QQ,V_\psi,\mathbb{D}^+_\psi)}=0$$  for some $i\in \{1,2\}$. The exact sequence (\ref{eqn:theselmercomlexseqforPRmaps}) and factorization (\ref{eqn:factorthePRmap}) yields the following containments:
\begin{align}
\notag\textup{span}_\cH\left(\mathscr{L}^{(1)}_{\lambda,\mu,\eta}\circ\,\res_p(\mathcal{Y}_1)\mathcal{Y}_2-\mathscr{L}^{(1)}_{\lambda,\mu,\eta}\circ\,\res_p(\mathcal{Y}_2)\mathcal{Y}_1\right)&=\pr_{\lambda,\mu}\left(H^1_\Iw(\QQ(r),V)\otimes\cH\right)\\
\label{eqn:containmentsforanalyticselmerotsukireductions}
&\subseteq \widetilde{H}^1_\Iw(\QQ,V_\psi,\mathbb{D}^+_\psi)\\
\notag&\hookrightarrow \left(H^1_\Iw(\QQ,V_\psi)\otimes\cH\right)/\cH\cdot \mathcal{Y}_i.
\end{align}
\emph{Case 1}. $\mathcal{Y}_1 \in \widetilde{H}^1_\Iw(\QQ,V_\psi,\mathbb{D}^+_\psi)$. In this case, it follows from (\ref{eqn:containmentsforanalyticselmerotsukireductions}) applied with $i=2$ that 
$$\mathscr{L}_{\lambda,\mu,\eta}^{(1)}\circ\,\res_p(\mathcal{Y}_2) c\in \pr_{\lambda,\mu}\left(H^1_\Iw(\QQ(r),V)\otimes\cH\right)$$
for every class $c\in  \widetilde{H}^1_\Iw(\QQ,V_\psi,\mathbb{D}^+_\psi)$; in particular, this holds true with the choice $c=\BF^{\lambda,\mu}_{\eta}$. The proof in this case is complete on noticing that 
$$\cH\cdot\mathscr{L}^{(1)}_{\lambda,\mu,\eta}\circ\,\res_p(\mathcal{Y}_2)=\mathscr{L}^{(1)}_{\lambda,\mu,\eta}\circ\,\res_p\left(H^1_\Iw(\QQ,V_\psi)\otimes\cH\right)\stackrel{\textup{Prop. {\ref{prop_imageofperrinrioumapontheglobalcohomology}}}}{\ni} h_\eta$$
since $\mathscr{L}_{\lambda,\mu,\eta}^{(1)}\circ\,\res_p(\mathcal{Y}_1)=0$.
\\\\
\emph{Case 2}. $\mathcal{Y}_2 \in \widetilde{H}^1_\Iw(\QQ,V_\psi,\mathbb{D}^+_\psi)$. The proof of Case 1 carries over. 
\\\\
\emph{Case 3}. $\mathcal{Y}_1,\mathcal{Y}_2 \not\in \widetilde{H}^1_\Iw(\QQ,V_\psi,\mathbb{D}^+_\psi)$. In this case, it follows from (\ref{eqn:containmentsforanalyticselmerotsukireductions}) (applied with both choices of $i\in \{1,2\}$) that
$$\left(r_1\mathscr{L}^{(1)}_{\lambda,\mu,\eta}\circ\,\res_p(\mathcal{Y}_1)+r_2\mathscr{L}^{(1)}_{\lambda,\mu,\eta}\circ\,\res_p(\mathcal{Y}_2)\right)c \in \pr_{\lambda,\mu}\left(H^1_\Iw(\QQ(r),V)\otimes\cH\right)$$
for any $r_1,r_2 \in \cH$ and any $c \in  \widetilde{H}^1_\Iw(\QQ,V_\psi,\mathbb{D}^+_\psi)$. Since we have 
$$\mathscr{L}^{(1)}_{\lambda,\mu,\eta}\circ\,\res_p\left(H^1_\Iw(\QQ,V_\psi)\otimes\cH\right)=\textup{span}_\cH\{\mathscr{L}^{(1)}_{\lambda,\mu,\eta}\circ\,\res_p(\mathcal{Y}_1)\,,\,\mathscr{L}^{(1)}_{\lambda,\mu,\eta}\circ\,\res_p(\mathcal{Y}_2)\},$$ 
this amounts to saying that
$$\mathscr{L}_{\lambda,\mu,\eta}^{(1)}\circ\,\res_p\left(H^1_\Iw(\QQ,V_\psi)\otimes\cH\right)c \subset  \pr_{\lambda,\mu}\left(H^1_\Iw(\QQ(r),V)\otimes\cH\right).$$
By Proposition~\ref{prop_imageofperrinrioumapontheglobalcohomology},
$$h_\eta c \in \mathscr{L}^{(1)}_{\lambda,\mu,\eta}\circ\,\res_p\left(H^1_\Iw(\QQ,V_\psi)\otimes\cH\right)c.$$ 
On taking  $c$ to be $\BF^{\lambda,\mu}_{\eta}$, the proof follows.
\end{proof}
\begin{defn}
For a fixed positive integer $r \in \mathcal{R}_\chi$ and each character $\eta \in \widehat{\Delta}_r$, let $\{\mathcal{Y}_1,\mathcal{Y}_2\}$ be a basis of $H^1_\Iw(\QQ,V_\psi)$. Let $d_{\eta}^{\lambda,\mu}\in \cH$ be the unique element with the property that
$$h_\eta\BF^{\lambda,\mu}_{\eta}=d_{\eta}^{\lambda,\mu}\cdot \left(\mathscr{L}^{(1)}_{\lambda,\mu,\eta}\circ\,\res_p(\mathcal{Y}_1)\mathcal{Y}_2-\mathscr{L}^{(1)}_{\lambda,\mu,\eta}\circ\,\res_p(\mathcal{Y}_2)\mathcal{Y}_1\right)\,.$$
Note that the existence of $d_{\eta}^{\lambda,\mu}$ is guaranteed by Theorem~\ref{thm_analyticselmervsrankreductionmap} and the description of $\pr_{\lambda,\mu}\left(H^1_\Iw(\QQ(r),V)\otimes\cH\right)$ in (\ref{eqn:containmentsforanalyticselmerotsukireductions}).
\end{defn}

\begin{proposition}
\label{prop:loeffler}
The elements $d_\eta^{\lambda,\mu}\in \cH$ are independent of $\lambda,\mu$.
\end{proposition}
\begin{proof}
To ease notation, we fix $\eta$ and drop it from the notation we use for the Perrin-Riou maps. With a slight abuse, we shall also write $\mathscr{L}^{(1)}_{\lambda,\mu}$ in place of $\mathscr{L}^{(1)}_{\lambda,\mu}\circ\,\res_p$ to ease our notation here.

Notice that we have
\begin{equation}
\label{eqn:Loefflerstep1}
h_\eta\mathscr{L}^{(1)}_{\lambda,\mu^*}\left({\BF^{\lambda,\mu}_{\eta}}\right)=-h_\eta\mathscr{L}^{(1)}_{\lambda,\mu}\left({\BF^{\lambda,\mu^*}_{\eta}}\right)
\end{equation}
by the explicit reciprocity law for Beilinson--Flach elements. On the other hand,
\begin{align}
\label{eqn:Loefflerstep2}
 h_\eta\mathscr{L}^{(1)}_{\lambda,\mu^*}\left({\BF^{\lambda,\mu}_{\eta}}\right)=d_{\eta}^{\lambda,\mu}\cdot \left(\mathscr{L}^{(1)}_{\lambda,\mu}(\mathcal{Y}_1)\mathscr{L}^{(1)}_{\lambda,\mu^*}(\mathcal{Y}_2)-\mathscr{L}^{(1)}_{\lambda,\mu}(\mathcal{Y}_2)\mathscr{L}^{(1)}_{\lambda,\mu^*}(\mathcal{Y}_1)\right)
\end{align}
and
\begin{equation}
\label{eqn:Loefflerstep3}
h_\eta\mathscr{L}^{(1)}_{\lambda,\mu}\left({\BF^{\lambda,\mu^*}_{\eta}}\right)=d_{\eta}^{\lambda,\mu^*}\cdot \left(\mathscr{L}^{(1)}_{\lambda,\mu^*}(\mathcal{Y}_1)\mathscr{L}^{(1)}_{\lambda,\mu}(\mathcal{Y}_2)-\mathscr{L}^{(1)}_{\lambda,\mu^*}(\mathcal{Y}_2)\mathscr{L}^{(1)}_{\lambda,\mu}(\mathcal{Y}_1)\right)
\end{equation}
On comparing (\ref{eqn:Loefflerstep1}), (\ref{eqn:Loefflerstep2}) and (\ref{eqn:Loefflerstep3}), we conclude that $d_{\eta}^{\lambda,\mu}=d_{\eta}^{\lambda,\mu^*}\,.$ The proof follows using in addition the fact that $d_{\eta}^{\lambda,\mu}=d_{\eta}^{\mu,\lambda}$, which we have thanks to Proposition~\ref{prop:thedichotomy} and Remark~\ref{rem:sym2orsymmetricproductdoesntmatter}.
\end{proof}
From now on, we let $d_\eta \in \cH$ denote $d_\eta^{\lambda,\mu}$ (which we have just seen is independent of  $\lambda$ and $\mu$).
\begin{defn}
\label{def_multiplier_c_433}
We set $c_\eta:=d_\eta/h_\eta \in \textup{Frac}(\cH)$ and 
$$c_r=\sum_{\eta\in \widehat{\Delta}_r}e_\eta c_{\eta^{-1}}\in \QQ_p[\Delta_r]\otimes  \textup{Frac}(\cH).$$
\end{defn}

In the statement of Theorem~\ref{thm_positionofBFintheanalyticselmer}, we may take $c_m$ to be  the element given in Definition~\ref{def_multiplier_c_433}. This satisfies properties (i)-(iii) and Theorem~\ref{thm_positionofBFintheanalyticselmer} follows.

\subsection{Analytic main conjectures with $p$-adic $L$-functions} 
\label{subsec_AnaylticMainConj}
{We prove in this subsection results towards Pottharst-style Iwasawa main conjectures (Conjecture~\ref{conj_analyticmainconjecture}). Our main result is a divisibility statement (Theorem~\ref{thm_analyticmainconjecture}) in these Iwasawa main conjectures, which is based on the divisibility in \eqref{eqn_ESmachinery}. This divisibility is deduced using  the Euler system of integral (doubly-signed) Beilinson--Flach elements (that we have constructed in Corollary~\ref{cor:signedBFinsymmsquare} of our main technical result Theorem~\ref{thm_positionofBFintheanalyticselmer}). Using global duality and the reciprocity laws for Beilinson--Flach elements, we give a bound on the Pottharst-style Selmer groups in terms of  $p$-adic Rankin--Selberg $L$-functions} (see Theorem~\ref{thm_analyticmainconjecture}).

Recall that $V:=\textup{Sym}^2W_f^*(1+\chi)$ and $\lambda,\mu \in \{\pm \alpha \}$\,. 
\begin{defn}
Let $\delta_{\chi} :\QQ_p^{\times}\ra E^{\times}$ be the character defined by $\delta_{\chi}(p):=p\chi^{-1}(p)$ and $\delta_{\chi}(u):=u$ for $u \in \ZZ_p^{\times}$. Let $\mathbb{D}_{\chi}$ denote the rank one $(\varphi,\Gamma)$-module $\mathcal{R}_{E}(\delta_{\chi})$. We set
\begin{itemize}
\item $\mathbb{D}^{\lambda,\mu}_{\chi} := (\mathbb{D}_{\lambda}\otimes \mathbb{D}_{f} + \mathbb{D}_{f}\otimes\mathbb{D}_{\mu}) \otimes \mathbb{D}_{\chi} \cap D^{\dagger}_{\mathrm{rig}}(V)$,
\item $\mathbb{D}^{\lambda}_{\chi} := \mathbb{D}_{\lambda}\otimes \mathbb{D}_{\lambda} \otimes \mathbb{D}_{\chi}\,.$
\end{itemize}
\end{defn}

\begin{conjecture}[Analytic Iwasawa main conjecture] \label{conj_analyticmainconjecture}  For $j \in \{k+2,\ldots,2k+2 \}$ even, the $\mathcal{H}$-module $e_{\omega^j}\widetilde{H}^2_{\Iw}(\QQ,V,\mathbb{D}_{\chi}^{\lambda})$ is torsion and \[ \mathrm{char}_{\mathcal{H}}\,e_{\omega^j}\widetilde{H}^2_{\Iw}(\QQ,V,\mathbb{D}_{\chi}^{\lambda}) = e_{\omega^j}L_{p}^{\mathrm{geom}}(\Sym^{2}f_{\lambda}\otimes\chi^{-1})\cdot\mathcal{H}\,. \]
\end{conjecture}

We will explain how our results in \S\ref{sec:nonordIwasawaESargument} on the signed Iwasawa main conjectures lead to partial results towards Conjecture~\ref{conj_analyticmainconjecture}. To this end, we assume until the end of this article that the hypotheses of Theorem~\ref{thm_signedmainconjecture} hold. Fix also an even integer $j \in \{k+2,\ldots,2k+2 \}$ and $\mathfrak{S} = (\clubsuit, \spadesuit) \in \mathcal{S}$  as in Proposition~\ref{prop:nonzerop-adicLfunction}.

\begin{proposition}
\label{prop_antisymofcolBF}
$e_{\omega^{j}}\col^{\clubsuit}\circ\mathrm{res}_{p}(\BF^{\spadesuit}_{1,\chi}) = - e_{\omega^j}\col^{\spadesuit}\circ\mathrm{res}_{p}(\BF^{\clubsuit}_{1,\chi}) \neq 0\,. $
\end{proposition}
\begin{proof}
The proof of the asserted equality is identical to the proof of \cite[Proposition 5.3.4]{BLLV}, on replacing reference to Theorem 3.9.1 in op. cit. by Theorem~\ref{thm_analyticselmervsrankreductionmap} here (with $r=1$). The non-vanishing follows immediately from our choice of $\frak{S}$.
\end{proof}

Recall that the  $\Lambda_\cO(\Gamma_1)$-module $e_{\omega^{j}}H^{1}(\QQ, \mathbb{T})$ is free of rank two by Theorem~\ref{thm_mainSelmerstructure}. We fix from now on a $\Lambda_\cO(\Gamma_1)$-basis of this module denoted by  $\{\mathfrak{c}_1,\mathfrak{c}_2\}$.

\begin{proposition} \label{prop:boundsforH2}
There exist non-zero elements  $\mathcal{D},\mathcal{E}_1 ,\mathcal{E}_2\in \Lambda_\cO(\Gamma_1)$  satisfying \[ \mathcal{D}\cdot e_{\omega^{j}}\BF^{\clubsuit}_{1,\chi} = \mathcal{E}_1(\col^{\clubsuit}\circ\mathrm{res}_{p}(\mathfrak{c}_{1})\mathfrak{c}_{2}-\col^{\clubsuit}\circ\mathrm{res}_{p}(\mathfrak{c}_{2})\mathfrak{c}_{1}),\]
 \[\mathcal{D}\cdot e_{\omega^{j}}\BF^{\spadesuit}_{1,\chi} = \mathcal{E}_{2}(\col^{\spadesuit}\circ\mathrm{res}_{p}(\mathfrak{c}_{1})\mathfrak{c}_{2}-\col^{\spadesuit}\circ\mathrm{res}_{p}(\mathfrak{c}_{2})\mathfrak{c}_{1})\,. \]
The first relation holds in $e_{\omega^{j}}H^{1}_{\mathcal{F}_{\clubsuit}}(\QQ,\mathbb{T})$ while the second holds in $e_{\omega^{j}}H^{1}_{\mathcal{F}_{\spadesuit}}(\QQ,\mathbb{T})$.

Furthermore,
$\mathcal{D}\cdot\mathrm{char}_\cH\left(e_{\omega^{j}}H^{2}(\QQ,\TT)\right)$ divides $\mathcal{E}_{1}\cdot{\rm char}_\cH \left(e_{\omega^{j}}{\rm coker}(\col^{\clubsuit})\right)$.
\end{proposition} 
\begin{proof}
The existence of the non-zero elements $\mathcal{D},\mathcal{E}_1$ and $\mathcal{E}_2$ follows from the fact that  the $\Lambda_\cO(\Gamma_1)$-module $e_{\omega^{j}}H^{1}_{\mathcal{F}_{?}}(\QQ,\mathbb{T})$ has rank one for $? \in \{\clubsuit, \spadesuit\}$, which is a consequence of the {locally restricted Euler system machinery} (used as in the proof of Theorem~\ref{thm_signedmainconjecture}).

The proof proceeds as in the proof of \cite[Proposition 7.4.4]{BLLV}. We set 
\[H^{1}_{/\clubsuit}(\QQ_p,\TT) := H^{1}(\QQ_p,\TT)/H^{1}_{\mathcal{F}_{\clubsuit}}(\QQ_p,\TT)\]
By the Poitou-Tate global duality, we have the following long exact sequence
\begin{multline} \label{eqn:poitou-tate-signed}
0\lra H^{1}_{\mathcal{F}_{\clubsuit}}(\QQ,\TT)/(\Lambda_\cO(\Gamma)\cdot\BF^{\clubsuit}_{1,\chi})\lra H^1(\QQ,\TT)/(\BF^{\clubsuit}_{1,\chi},\BF^{\spadesuit}_{1,\chi})\lra \frac{H^{1}_{/\clubsuit}(\QQ_{p},\TT)}{\mathrm{res}_{/\clubsuit}(\BF^{\spadesuit}_{1,\chi})}\\
\lra\mathrm{Sel}_{\clubsuit}(\QQ,\TT^{\vee}(1))^{\vee}\lra H^2(\QQ,\TT)\lra 0,
\end{multline}
where $\mathrm{res}_{/\clubsuit}$ denotes the composition
\[ H^1(\QQ,\TT) \xrightarrow{\mathrm{res}_{p}} H^1(\QQ_p,\TT)\twoheadrightarrow H^1_{/\clubsuit}(\QQ_p,\TT).\]
The {locally restricted Euler system machinery} shows (see Theorem~\ref{thm_signedmainconjecture}) that
\begin{equation} \label{eqn_ESmachinery}
e_{\omega^{j}}\mathrm{char}_\cH\left(\mathrm{Sel}_{\clubsuit}(\QQ,\TT^{\vee}(1))^{\vee}\right)\, \Big{\vert}\, \mathrm{char}_\cH\left(e_{\omega^{j}} H^{1}_{\mathcal{F}_{\clubsuit}}(\QQ,\TT)\Big{/}\Lambda_{\cO}(\Gamma_1)\cdot e_{\omega^{j}}\BF^{\clubsuit}_{1,\chi}\right)\,.
\end{equation}

Combining (\ref{eqn:poitou-tate-signed}) and (\ref{eqn_ESmachinery}), we deduce that
\begin{equation} \label{eqn:divisibility1}
\mathrm{char}_\cH\left(e_{\omega^{j}}H^{2}(\QQ,\TT)\right)\left(\mathrm{char}_\cH\left(e_{\omega^{j}}{\rm coker}(\col^{\clubsuit}\right)\right)^{-1}e_{\omega^{j}}\col^{\clubsuit}\circ\mathrm{res}_{p}(\BF^{\spadesuit}_{1,\chi})\,\,\Big{\vert}\,\,\mathrm{char}_\cH\left(\frac{e_{\omega^{j}}H^1(\QQ,\TT)}{e_{\omega^{j}}(\BF^{\clubsuit}_{1,\chi},\BF^{\spadesuit}_{1,\chi})}\right).
\end{equation}We set
\[ \mathrm{det}\circ\col(\spadesuit,\clubsuit):= \mathrm{det}\begin{pmatrix}
\col^{\spadesuit}\circ\mathrm{res}_{p}(\mathfrak{c}_{1}) & \col^{\clubsuit}\circ\mathrm{res}_{p}(\mathfrak{c}_{1})\\
\col^{\spadesuit}\circ\mathrm{res}_{p}(\mathfrak{c}_{2}) & \col^{\clubsuit}\circ\mathrm{res}_{p}(\mathfrak{c}_{2})
\end{pmatrix}. \]
Note that
\begin{equation} \label{eqn:divisibility2}
\col^{\clubsuit}\circ\mathrm{res}_{p}(e_{\omega^{j}}\BF^{\spadesuit}_{1,\chi}) = \mathcal{D}^{-1}\mathcal{E}_{2}\,\mathrm{det}\circ\col(\spadesuit,\clubsuit).
\end{equation}
By Proposition~\ref{prop:boundsforH2}, we have
\begin{equation} \label{eqn:divisibility3}
\mathrm{char}_\cH\Bigg(\frac{e_{\omega^{j}}H^1(\QQ,\TT)}{e_{\omega^{j}}(\BF^{\clubsuit}_{1,\chi},\BF^{\spadesuit}_{1,\chi})}\Bigg) = \mathcal{D}^{-2}\mathcal{E}_{1}\mathcal{E}_{2}\,\mathrm{det}\circ\col(\spadesuit,\clubsuit).
\end{equation}
Combining (\ref{eqn:divisibility1}), (\ref{eqn:divisibility2}) and (\ref{eqn:divisibility3}), we deduce the stated divisibility
\[\mathcal{D}\cdot\mathrm{char}_\cH\left(e_{\omega^{j}}H^{2}(\QQ,\TT)\right)\,\,\big{|}\,\, \mathcal{E}_{1}\cdot{\rm char}_\cH\left(e_{\omega^{j}}{\rm coker}(\col^{\clubsuit})\right).\]
\end{proof}
We will now use the bounds for the characteristic ideal of $e_{\omega^{j}}H^{2}(\QQ,\TT)$ obtained in Proposition~\ref{prop:boundsforH2} to bound characteristic ideals of  {analytic Selmer groups}.
\begin{defn}
Let $\frak{a},\frak{b}_1,\frak{b}_2\in \mathcal{H}\setminus\{0\}$ be elements satisfying
\begin{align*}
 \frak{a}\cdot\BF_{1,\chi}^{\lambda,\lambda}& =\frak{b}_1\,(\cL_{\lambda,\lambda}^{(1)}\circ\mathrm{res}_{p}(\mathfrak{c}_{1})\mathfrak{c}_{2} - \cL_{\lambda,\lambda}^{(1)}\circ\mathrm{res}_{p}(\mathfrak{c}_{2})\mathfrak{c}_{1})  \\
  \frak{a}\cdot\BF_{1,\chi}^{\lambda,-\lambda} &=\frak{b}_2\,(\cL^{(1)}_{\lambda,-\lambda}\circ\mathrm{res}_{p}(\mathfrak{c}_{1})\mathfrak{c}_{2} - \cL^{(1)}_{\lambda,-\lambda}\circ\mathrm{res}_{p}(\mathfrak{c}_{2})\mathfrak{c}_{1})\,.  
\end{align*}
\begin{lemma}
\label{lemma_helpfulsymmetriesofcorrectionterms}
$\mathcal{E}_1=\mathcal{E}_2$ and $\frak{b}_1=\frak{b}_2$. Moreover, $\frak{a}\mathcal{E}_1=\frak{b}_1\mathcal{D}$.
\end{lemma}
\begin{proof}
The first equality follows from Proposition~\ref{prop_antisymofcolBF} and the second from Proposition~\ref{prop:loeffler}. The final assertion is immediate by the definitions of $\mathcal{D},\mathcal{E}_1, \frak{a},\frak{b}_1$ and Theorem~\ref{thm_analyticselmervsrankreductionmap} (applied with $r=1$).
\end{proof}
\end{defn}
\begin{proposition} 
\label{prop_analyticMCstep1}
We have the following divisibility of $\mathcal{H}$-ideals
\[ {\frak{b}\,\mathcal{D}}\,\mathrm{char}_{\mathcal{H}}\left(e_{\omega^{j}}\widetilde{H}^2_{\Iw}(\QQ,V,\mathbb{D}_{\chi}^{\lambda,\lambda})\right) \,\,\, \Big{\vert}\,\,\, {\frak{a}\,\mathcal{E}_{1}}\,\mathrm{char}_\cH(\mathrm{coker}\col^{\clubsuit})\,\mathrm{char}_\cH\left(\frac{e_{\omega^{j}}\widetilde{H}^1_{\Iw}(\QQ,V,\mathbb{D}_{\chi}^{\lambda,\lambda})}{\mathcal{H}\cdot e_{\omega^{j}} \BF^{\lambda,\lambda}_{1,\chi}} \right)\,. \]
\end{proposition}
\begin{proof}
 The proof follows very closely that of \cite[Proposition 7.4.6]{BLLV}. Note that we have the following five term exact sequence of $\mathcal{H}$-modules
\begin{multline*} 
0\longrightarrow\frac{e_{\omega^{j}}\widetilde{H}^1_{\Iw}(\QQ,V,\mathbb{D}_{\chi}^{\lambda,\lambda})}{\mathcal{H}\cdot e_{\omega^{j}}\BF^{\lambda,\lambda}_{1,\chi}}\longrightarrow\frac{e_{\omega^{j}}H^1_{\Iw}(\QQ,V)\otimes\mathcal{H}}{\mathcal{H}\cdot e_{\omega^{j}}\BF_{1,\chi}^{\lambda,\lambda}+\mathcal{H}\cdot e_{\omega^{j}}\BF_{1,\chi}^{\lambda,-\lambda}}\longrightarrow\frac{e_{\omega^{j}}H^1_{/\lambda,\lambda}(\QQ_p,V)}{\mathrm{res}_{/\lambda,\lambda}(e_{\omega^{j}}\BF_{1,\chi}^{\lambda,-\lambda})} \\
\longrightarrow e_{\omega^j}\widetilde{H}^2_{\Iw}(\QQ,V,\mathbb{D}_{\chi}^{\lambda,\lambda}) \longrightarrow e_{\omega^j}H^2_{\Iw}(\QQ,V)\otimes_{\Lambda}\mathcal{H} \longrightarrow 0,
\end{multline*} 
where $H^1_{/\lambda,\lambda}(\QQ_p,V) := H^{1}_{\Iw}(\QQ_{p},V)\otimes\mathcal{H}(\Gamma)\big{/}H^1_{\Iw}(\QQ_p,\mathbb{D}_{\chi}^{\lambda,\lambda})$ and $\mathrm{res}_{/\lambda,\lambda}$ is the composition
\[ \mathrm{res}_{/\lambda,\lambda} : H^{1}_{\Iw}(\QQ,V)\otimes\mathcal{H} \xrightarrow{\mathrm{res}_{p}} H^{1}_{\Iw}(\QQ_p,V)\otimes_{\Lambda}\mathcal{H}(\Gamma) \twoheadrightarrow H^1_{/\lambda,\lambda}(\QQ_p,V)\,. \]
By Proposition~\ref{prop:boundsforH2} and the surjectivity of $\cL_{\lambda,\lambda}^{(1)}:H^1_{/\lambda,\lambda}(\QQ_p,V)\lra\mathcal{H}$,
the proof follows.
\end{proof}
We finally conclude with the following divisibility towards analytic main conjectures, which is Theorem~\ref{thmC} in the introduction.
\begin{theorem} \label{thm_analyticmainconjecture} 
In the setting of Theorem~\ref{thm_signedmainconjecture}, we have
\[ \mathrm{char}_{\mathcal{H}}\,e_{\omega^j}\widetilde{H}^2_{\Iw}(\QQ,V,\mathbb{D}_{\chi}^{\lambda})\,\,\, \big{\vert}\,\, \mathrm{char}_\cH(e_{\omega^j}\mathrm{coker}\col^{\clubsuit})\,\, e_{\omega^j}L_{p,NN_{\chi}}(\chi^{-1}\epsilon_f) \,\,L_{p}^{\mathrm{geom}}(\Sym^{2}f_{\lambda}\otimes\chi^{-1})\]
as ideals of $\mathcal{H}$.
\end{theorem}
\begin{proof}
We start off with the following four-term exact sequence induced by the definition of corresponding Selmer complexes:
\begin{multline} \label{eqn:analyticexactsequence2}
0  \lra \frac{e_{\omega^j}\widetilde{H}^{1}_{\Iw}(\QQ,V,\mathbb{D}_{\chi}^{\lambda,\lambda})}{\mathcal{H}\cdot e_{\omega^j}\BF_{1,\chi}^{\lambda,\lambda}}  \longrightarrow \frac{e_{\omega^j}H^{1}_{\Iw}(\QQ_{p},\mathbb{D}_{\chi}^{\lambda,\lambda})/H^{1}_{\Iw}(\QQ_p,\mathbb{D}_{\chi}^{\lambda})}{\cH\cdot e_{\omega^j}\mathrm{res}_{p}^{\rm s}(\BF_{1,\chi}^{\lambda,\lambda})}\\ \longrightarrow e_{\omega^j} \widetilde{H}^{2}_{\Iw}(\QQ,V,\mathbb{D}_{\chi}^{\lambda}) \longrightarrow  e_{\omega^j}\widetilde{H}^{2}_{\Iw}(\QQ,V,\mathbb{D}_{\chi}^{\lambda,\lambda}) \lra 0 
\end{multline}
where $\mathrm{res}_{p}^{\rm s}$ is the composition of the arrows
\[ \widetilde{H}^{1}_{\Iw}(\QQ,V,\mathbb{D}_{\chi}^{\lambda,\lambda}) \longrightarrow H^{1}_{\Iw}(\QQ_{p},\mathbb{D}_{\chi}^{\lambda,\lambda}) \twoheadrightarrow H^{1}_{\Iw}(\QQ_{p},\mathbb{D}_{\chi}^{\lambda,\lambda})/H^{1}_{\Iw}(\QQ_p,\mathbb{D}_{\chi}^{\lambda})\,.\] 
We note that the first injection in (\ref{eqn:analyticexactsequence2}) is a special case of Proposition~\ref{prop_globallocalseqandcomparisonwithBK}(iii) (which tells us that $\widetilde{H}^{1}_{\Iw}(\QQ,V,\mathbb{D}_{\chi}^{\lambda})=0$). The asserted divisibility now follows on combining Proposition~\ref{prop_analyticMCstep1} and the last identity of Lemma~\ref{lemma_helpfulsymmetriesofcorrectionterms}, together with the definition of the geometric $p$-adic $L$-function.
\end{proof}
\bibliographystyle{amsalpha}
\bibliography{references}
\end{document}